\documentclass[12pt,reqno]{amsart}
\title{On the Unitary Cohomology of Semisimple Groups}
\usepackage{amsmath, amsthm, amsopn, amssymb, microtype,mathtools}
\usepackage{mathrsfs} 
\usepackage[all,cmtip]{xy}
\usepackage{enumerate, tikz, etoolbox, intcalc, geometry, caption, subcaption, afterpage}
\usepackage{tikz-cd}
\usepackage[english]{babel}
\usepackage{comment}
\usepackage[utf8]{inputenc}
\usepackage{graphicx}

\newsavebox{\waterbox}
\newsavebox{\earthbox}
\newsavebox{\firebox}
\newsavebox{\airbox}

\sbox{\waterbox}{\raisebox{-0.4ex}{\includegraphics[height=2ex]{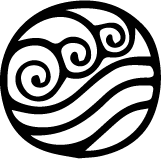}}}
\sbox{\earthbox}{\raisebox{-0.4ex}{\includegraphics[height=1.8ex]{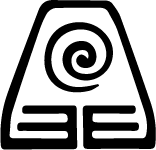}}}
\sbox{\firebox}{\raisebox{-0.4ex}{\includegraphics[height=2ex]{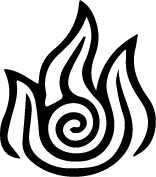}}}
\sbox{\airbox}{\raisebox{-0.4ex}{\includegraphics[height=2ex]{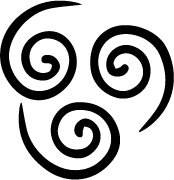}}}

\newcommand{\watersym}{\usebox{\waterbox}}
\newcommand{\earthsym}{\usebox{\earthbox}}
\newcommand{\firesym}{\usebox{\firebox}}
\newcommand{\airsym}{\usebox{\airbox}}

\author[U. Bader]{Uri Bader\textsuperscript{\earthsym}}
\address{university of Maryland, College Park, USA and Weizmann Institute of Science, Rehovot, Israel}
\email{uri.bader@gmail.com}

\author[M. Glasner]{Michael Glasner\textsuperscript{\watersym}}
\address{Weizmann Institute of Science, Rehovot, Israel}
\email{michael.glasner.a@gmail.com}

\author[Y. Gorfine]{Yuval Gorfine\textsuperscript{\airsym}}
\address{Weizmann Institute of Science, Rehovot, Israel} % (Added missing address for Yuval!)
\email{yuval.gorfine@gmail.com}

\author[R. Sauer]{Roman Sauer\textsuperscript{\firesym}}
\address{Karlsruhe Institute of Technology, Karlsruhe, Germany}
\email{roman.sauer@kit.edu}

\thanks{\textsuperscript{\watersym}\textsuperscript{\airsym}Weizmann Institute of Science;}
\thanks{\textsuperscript{\firesym}Karlsruhe Institute of Technology;}
\thanks{\textsuperscript{\earthsym}University of Maryland and Weizmann Institute of Science.}


\usepackage[colorlinks=true,urlcolor=blue,pdfborder={0 0 0}]{hyperref}
\hypersetup{linkcolor=[rgb]{0,0,0.6}}
\hypersetup{citecolor=[rgb]{0,0.6,0}}
\usepackage[nameinlink]{cleveref}
\crefname{theorem}{Theorem}{Theorems}
\crefname{theorem}{Theorem}{Theorems}
\crefname{mainthm}{Theorem}{Theorems}
\crefname{lemma}{Lemma}{Lemmas}
\crefname{lem}{Lemma}{Lemmas}
\crefname{remark}{Remark}{Remarks}
\crefname{prop}{Proposition}{Propositions}
\crefname{defn}{Definition}{Definitions}
\crefname{corollary}{Corollary}{Corollaries}
\crefname{cor}{Corollary}{Corollaries}
\crefname{section}{Section}{Sections}
\crefname{figure}{Figure}{Figures}
\crefname{quest}{Question}{Questions}

\newcommand{\N}{\mathbb{N}}
\newcommand{\Z}{\mathbb{Z}}
\newcommand{\Q}{\mathbb{Q}}
\newcommand{\R}{\mathbb{R}}
\newcommand{\C}{\mathbb{C}}

\newcommand{\calN}{\mathcal{N}}

\newtheorem{theorem}{Theorem}[section]
\newtheorem{lemma}[theorem]{Lemma}
\newtheorem{proposition}[theorem]{Proposition}
\newtheorem{corollary}[theorem]{Corollary}

\newtheorem{question}[theorem]{Question}
\newtheorem{prob}[theorem]{Problem}
\newtheorem{mainthm}{Theorem}

\theoremstyle{definition}
\newtheorem{definition}[theorem]{Definition}
\newtheorem{example}[theorem]{Example}
\newtheorem{remark}[theorem]{Remark}
\newtheorem{terminology}[theorem]{Terminology}
\newtheorem{conjecture}[theorem]{Conjecture}

\newcommand{\supp}{\operatorname{Supp}}

\newcommand{\SL}{\operatorname{SL}}
\newcommand{\PSL}{\operatorname{PSL}}

\newcommand{\SO}{\operatorname{SO}}
\newcommand{\torH}{\mathring{H}}
\newcommand{\barH}{\bar{H}}

\newcommand{\abs}[1]{\left| #1 \right|}

\newcommand{\Ind}{\operatorname{Ind}}
\newcommand{\Res}{\operatorname{Res}}

\makeatletter
\renewcommand\paragraph{\@startsection{paragraph}{4}{\z@}%
  {1.5ex \@plus 1ex \@minus .2ex}{-1em}%
  {\normalfont\normalsize\bfseries}}
\makeatother

\begin{document}

\begin{abstract}
We study the continuous homology and cohomology of semisimple Lie groups with coefficients in arbitrary unitary representations. The case of irreducible representations was determined by Vogan and Zuckerman; we focus on reducible representations. The (co)homology splits into its \emph{Hausdorff} and \emph{torsion} parts. The Hausdorff part is governed by containment of irreducible cohomological representations. We show that the torsion part is governed by \emph{weak containment} of irreducible cohomological representations. Precisely, we show that a unitary representation admits non-zero torsion if and only if there is a cohomological point which is not isolated in its support. We discuss in length examples of rank-$1$ groups, completely determining unitary cohomology for the group $\SO^\circ(n,1)$.

For a simple Lie group, \cite{BaderSauer} showed that the first degree in which it obtains non-trivial cohomology for some unitary representation with no invariant vectors is related to the rank of the group. We discuss the analogous question regarding torsion cohomology, and show that the corresponding first degree could be much higher: in the presence of property (T), it is bounded below by the square root of the dimension of the symmetric space. A technical device that we use  is the restriction to well chosen dense subgroups which satisfy finiteness properties, which we call \emph{cohomological witnesses}. We combine it with results on the unitary cohomology of groups which satisfy finiteness properties. These results are of an independent interest.
\end{abstract}

\maketitle

\section{introduction}\label{sec1}
The goal of this paper is to systematically study the continuous homology and cohomology of a semisimple group $G$ with coefficients in unitary representations. 
We will explain below our conventions regarding the meaning of the term ``semisimple'' (see Terminology~\ref{terminology:semisimple}).
For a unitary representation $\pi$, the continuous homology and cohomology in degree $q$ will be denoted $H_q(G; \pi)$ and $H^q(G;\pi)$ respectively.
These spaces have natural structures of topological vector spaces that are not necessarily Hausdorff.

\begin{mainthm}\label{intro:torsion-reduced}
Let $G$ be a semisimple group and let $\pi$ be a unitary representation.
Then in every degree $q$ both $H_q(G;\pi)$ and $H^q(G;\pi)$ are decomposed, as topological vector spaces, into a direct sum of a Hilbert space and a space with the trivial topology. This decomposition is functorial with respect to the representation $\pi$.
\end{mainthm}

We denote the corresponding decompositions by 
\begin{equation} \label{eq:tor-bar}
   H_q(G;\pi)= \torH_q(G;\pi)\oplus \barH_q(G;\pi), \quad H^q(G;\pi)= \torH^q(G;\pi)\oplus \barH^q(G;\pi), 
\end{equation}
where the spaces $\torH_q(G;\pi)$ and $\torH^q(G;\pi)$, called the \emph{homology torsion} and \emph{cohomology torsion} respectively, have the trivial topology, and the spaces $\barH_q(G;\pi)$ and $\barH^q(G;\pi)$, called the \emph{Hausdorff homology} and \emph{Hausdorff cohomology} respectively\footnote{The Hausdorff (co)homology is also called the \emph{reduced (co)homology}, but we find this terminology confusing and try to avoid it.}, are isomorphic to Hilbert spaces.
There is no particular choice of inner products on the Hausdorff (co)homology, only a topological vector space structure isomorphic to that of a Hilbert space.

\begin{mainthm}\label{intro:homology-cohomology}
Let $G$ be a semisimple group and let $\pi$ be a unitary representation of $G$.
Then $\bar{H}^q(G;\pi)$ is naturally isomorphic to the complex conjugate\footnote{Do not confuse here the bar notation with the complex conjugation.} of $\bar{H}_q(G;\pi)^*$ and  
\[
    \torH_q(G;\pi)=0 \iff \torH^{q+1}(G;\pi)=0.
\]
    Moreover, if $G$ is a Lie group and $d$ is the dimension of the corresponding symmetric space, then we have natural Poincar\'e duality isomorphisem of topological vector spaces
    $H_q(G;\pi)\cong H^{d-q}(G;\pi)$, commuting with the decomposition \eqref{eq:tor-bar}.
\end{mainthm}

In view of the above theorem, we are free to focus here mainly on the continuous cohomology.
This is convenient, as in certain aspects the theory of continuous cohomology is better behaved and better developed than the theory of continuous homology.
For instance, the theory of continuous cohomology of
\emph{irreducible} unitary representations of semisimple groups was satisfyingly settled by Vogan and Zuckerman in the 1980's, relaying on contributions by several authors, in the two papers \cite{VZ} and \cite{Unitarizability}.
In particular, for an irreducible unitary representation $\pi$, at every degree $q$, $H^q(G;\pi)$ is Hausdorff and finite dimensional. 
Moreover, up to isomorphism, there exist only finitely many \emph{cohomological representations}, that is irreducible representations which have non-zero continuous cohomology at some degree.
We denote this finite set of isomorphism types by $\widehat{G}_{\mathrm{cohom}}$.

\begin{prob} \label{prob:main}
    Let $G$ be a semisimple group and let $\pi$ be a unitary representation.
    Find a mechanism that answers the following questions by means 
    of \emph{containment} and \emph{weak containment} of irreducible cohomological representations in $\pi$.
\begin{enumerate}
    \item Does $\pi$ have a Hausdorff cohomology? In which degrees? How much?
    \item Does $\pi$ have a torsion cohomology? In which degrees?
\end{enumerate}
Note that we do not ask ``how much'' for the torsion cohomology, as we do not have good ways to describe an answer to such a question.
\end{prob}

Below we will provide satisfactory answers to question 1 and to the first part of question 2. It turns out that the second part of question 2 is more subtle, and much of this paper will be devoted to dealing with this specific problem.

\begin{remark} \label{rem:measure?}
    One might hope for better answers to the above questions by means of a complete data regarding $\pi$, e.g. by means of its presentation as a direct integral with respect to a measure $\mu$ on the unitary dual of $G$. Containment of an irreducible representation is equivalent to an atom of $\mu$ and weak containment is equivalent to a point in its support. 
    However, we will see below that the partial information about $\mu$ considered in Problem~\ref{prob:main} is enough to answer the above questions.
\end{remark}

Given a unitary representation $\pi$, we denote by $\pi_{\mathrm{cohom}}$ the subrepresentation of $\pi$ consisting of the direct sum of the (finitely many)  isotypic subrepresentations associated with irreducible cohomological representations\footnote{We hold the temptation to denote it $\bar{\pi}$.}, and we denote its orthogonal complement by $\mathring{\pi}$, thus
\[ \pi=\pi_{\mathrm{cohom}} \oplus \mathring{\pi} \qquad \mbox{and} \qquad 
\pi_{\mathrm{cohom}} \cong 
\bigoplus_{\rho \in \widehat{G}_{\mathrm{cohom}}} \mathrm{Hom}_G(\rho,\pi)\otimes \rho. \]
To clarify the notation on the right-hand side, our convention here is that the direct sum is an $\ell^2$-direct sum, we endow the spaces $\mathrm{Hom}_G(\rho,\pi)$ with their natural Hilbert space structures (see Remark \ref{rem:hilbert structure}) and our convention is that the tensor product notation stands for the completed Hilbertian tensor product. The Hilbert space $\mathrm{Hom}_G(\rho,\pi)$ is called the \emph{multiplicity module} of $\rho$ in $\pi$ and is taken with the trivial $G$-action.

The subrepresentations $\pi_{\mathrm{cohom}}$ and $\mathring{\pi}$ are said to be the \emph{Hausdorff part} and the \emph{torsion part} of $\pi$ correspondingly.
The next theorem states that this decomposition is compatible with the decomposition given in \eqref{eq:tor-bar}.

\begin{mainthm} \label{intro-Hausdorff}
Let $G$ be a semisimple group and let $\pi$ be a unitary representation of $G$. Then in any degree $q$, $\barH^q(G;\mathring{\pi})=0$ and $\torH^q(G;\pi_{\mathrm{cohom}})=0$, thus $\torH^q(G;\pi)=\torH^q(G;\mathring{\pi})$ and 
\[ \bar{H}^q(G;\pi)= \bar{H}^q(G;\pi_{\mathrm{cohom}}) \cong \bigoplus_{\rho \in \widehat{G}_{\mathrm{cohom}}} \mathrm{Hom}_G(\rho,\pi) \otimes H^q(G;\rho). \]
\end{mainthm}

Any choice of Hilbert space structures on the finite dimensional spaces $H^q(G;\rho)$ gives rise to a choice of a Hilbert space structure on $\barH^q(G;\pi)$, by the expression above, but there is no canonical such choice.
By Theorems \ref{intro:torsion-reduced} and \ref{intro:homology-cohomology} we obtain similar identities in homology, namely 
$\torH_q(\pi)=\torH_q(\mathring{\pi})$ and 
\[ \bar{H}_q(G;\pi)= \bar{H}_q(G;\pi_{\mathrm{cohom}}) \cong \bigoplus_{\rho \in \widehat{G}_{\mathrm{cohom}}} \mathrm{Hom}_G(\rho,\pi) \otimes H_q(G;\rho). \]

Theorem \ref{intro-Hausdorff} provides a full answer to question 1 of Problem~\ref{prob:main}.
We are left to discuss question 2 regarding the cohomological torsion.
The above expressions show that the Hausdorff cohomology of a unitary representation stems from containment of cohomological irreducible subrepresentations.
In a similar fashion, it turns out that the torsion cohomology stems from \emph{weak containment} of cohomological irreducible subrepresentations.
We will now explain more precisely this phenomenon.

We denote by $\widehat{G}$ the unitary dual of $G$, that is the set of isomorphism types of irreducible unitary representations.
The \emph{support} of a unitary representation $\pi$, $\mathrm{supp}(\pi)\subset \widehat{G}$, is the subset of irreducible representations which are weakly contained in $\pi$.
Running over all the unitary representations $\pi$, the collection of all supports forms the closed subsets of the \emph{Fell topology} on $\widehat{G}$.
The following theorem shows that the questions whether $\pi$ has either cohomology or torsion cohomology depend only on its support.
We thus say that these are \emph{topological properties} of $\pi$.

We denote by $\mathrm{supp}(\pi)'$ the subset of $\mathrm{supp}(\pi)$ consisting of non-isolated points.
These are the irreducible representations $\rho\in \widehat{G}$ which are weakly contained in the orthogonal complement of their own isotypic component in $\pi$. 
Such irreducible representations $\rho$ are said to be \emph{essentially weakly contained} in $\pi$.

\begin{mainthm}\label{intro:topological propery1}
Let $G$ be a semisimple group, and let $\pi$ be a unitary representation of $G$.
Then $\pi$ has cohomology in some degree if and only if it weakly contains a cohomological representation, and $\pi$ has cohomological torsion in some degree if and only if it essentially weakly contains a cohomological representation, that is 
\[ \exists_{q\in\N}~H^q(G;\pi)\neq 0 \iff \mathrm{supp}(\pi) \cap \widehat{G}_{\mathrm{cohom}} \neq \emptyset,  \]
and 
\[ \exists_{q\in\N}~\torH^q(G;\pi)\neq 0 \iff \mathrm{supp}(\pi)' \cap \widehat{G}_{\mathrm{cohom}} \neq \emptyset.  \]
\end{mainthm}

It follows in particular that non-cohomological irreducible representations in $\widehat{G}$ do not contribute at all to the cohomology and 
cohomological representations which are \emph{isolated} in $\widehat{G}$ do not contribute at all to torsion cohomology. 
We will see later that cohomological representations which are not isolated in $\widehat{G}$ do contribute some torsion cohomology.
We view the statements of Theorem~\ref{intro:topological propery1} as satisfying ``if and only if'' criteria for the existence of cohomology and torsion cohomology.
In particular, Theorem~\ref{intro:topological propery1} answers the first part of question 2 of Problem~\ref{prob:main}.
The only remaining question is the second part of this question -  in which degrees does the torsion cohomology appear.
This last question turns out to be a subtle one.
Before addressing it, we address the question raised in Remark~\ref{rem:measure?}, namely, are we addressing the right problem?
The following theorem shows that the existence of torsion cohomology in each degree is a topological property, hence indeed, it is in the scope of Problem~\ref{prob:main}.

\begin{mainthm}\label{intro:topological propery}
    Let $G$ be a semisimple group, and let $\pi$ and $\rho$ be unitary representations of $G$.
    If $\rho$ is weakly contained in $\pi$ and $\torH^q(G;\rho) \ne 0$, then $\torH^q(G;\pi) \ne 0$. In particular, if $\mathrm{supp}(\pi)=\mathrm{supp}(\rho)$ then 
    \[ \torH^q(G;\pi)\neq 0 \iff \torH^q(G;\rho) \neq 0. \]
\end{mainthm}

\begin{remark}
Theorem~\ref{intro:topological propery1} shows that both the existence of cohomology and the existence of torsion in some degree are topological properties, and Theorem~\ref{intro:topological propery} shows that the existence of torsion cohomology in a given degree is a topological property.
That is, these properties are determined by the support of a representation.
However, note that the existence of cohomology in a given degree is \emph{not} a topological property. Moreover, the existence of Hausdorff cohomology, in some or in a given degree, is \emph{not} a topological property.
\end{remark}

A cohomological irreducible unitary representation $\rho$ is said to be $q$-\emph{cohomological} if it is cohomological in degree $q$, that is $H^q(G;\rho) \ne 0$. The set of isomorphism types of such representations is denoted $\widehat{G}_{\mathrm{q-cohom}} \subseteq \widehat{G}_{\mathrm{cohom}}$.
The following is an analogue of Theorem~\ref{intro:topological propery1} that applies to a specific degree $q$.

\begin{mainthm}\label{intro: main thm2}
    Let $G$ be a semisimple group and let $\pi$ be a unitary representation of $G$.
    Then $\pi$ has torsion cohomology in degrees $q$ or $q+1$ if and only if it essentially weakly contains a cohomological representation in degree $q$, that is 
    \[ \torH^q(G;\pi)\neq 0 \quad \mbox{or} \quad \torH^{q+1}(G;\pi)\neq 0 \qquad \Longleftrightarrow \qquad \mathrm{supp}(\pi)' \cap \widehat{G}_{\mathrm{q-cohom}} \neq \emptyset.  \]
\end{mainthm}

\begin{remark}
    For $G=\SL_2(\R)$ there exists a representation $\pi$ such that the above properties are satisfied in all degrees $q=0$, $q=1$ and $q=2$, 
    exhibiting the three different possibilities for the left-hand side property. For $q=0$ there exists torsion in degree $q+1$ but not in degree $q$, for $q=1$ there exists torsion in both degrees $q$ and $q+1$, and 
    for $q=2$ there exists torsion in degree $q$ but not in degree $q+1$.
    This will be discussed in \S\ref{sec6}.
\end{remark}

\begin{remark}
In view of Theorem~\ref{intro:homology-cohomology}, Theorem~\ref{intro: main thm2} can be reformulated nicely by  
 \[ \torH^q(G;\pi)\neq 0 \quad \mbox{or} \quad \torH_{q}(G;\pi)\neq 0 \qquad \Longleftrightarrow \qquad \mathrm{supp}(\pi)' \cap \widehat{G}_{\mathrm{q-cohom}} \neq \emptyset.  \]
\end{remark}

We do not have a satisfying ``if and only if'' criterion for torsion cohomology in degree $q$.  
The following is a one sided implication.
Note that its first part follows trivially from Theorem~\ref{intro: main thm2},
but the ``moreover'' part provides extra information.

\begin{mainthm}\label{intro: main thm1}
    Let $G$ be a semisimple group and let $\pi$ be a unitary representation of $G$.
    Fix a degree $q$. Then we have
     \[ \torH^q(G;\pi)\neq 0 \quad \Longrightarrow \quad \mathrm{supp}(\pi)' \cap \widehat{G}_{\mathrm{q-cohom}} \neq \emptyset \quad \mbox{and}\quad  \mathrm{supp}(\pi)' \cap \widehat{G}_{\mathrm{(q-1)-cohom}} \neq \emptyset.\]
    Moreover, $\pi$ essentially weakly contains irreducible representations $\rho_1\in \widehat{G}_{\mathrm{q-cohom}}$ and $\rho_2\in \widehat{G}_{\mathrm{(q-1)-cohom}}$ such that $\rho_1$ and $\rho_2$ are inseparable by neighborhoods in $\mathrm{supp}(\pi)$. That is, if $U_1$ is a neighborhood of $\rho_1$ in $\widehat{G}$, and $U_2$ is a neighborhood of $\rho_2$, then $\mathrm{supp}(\pi) \cap U_1 \cap  U_2 \ne \emptyset$.
\end{mainthm}

\begin{remark} \label{rem:SL2C}
    The above is \emph{not} an ``if and only if'' criterion.
    For $G=\SO^\circ(2k+1,1)$, there exists a representation $\pi$ such that the right-hand side of the implication in Theorem~\ref{intro: main thm1}, including the ``moreover'' addendum, is satisfied, but the left-hand side is not. In fact, $\pi$ essentially weakly contains an irreducible representation $\rho\in \widehat{G}_{\mathrm{q-cohom}} \cap \widehat{G}_{\mathrm{(q-1)-cohom}}$. See Theorem \ref{greatQuestion}.
\end{remark}

In Problem~\ref{prob:main} we described our main goals in writing this paper.
We now specify the part that we did not achieve here.
A precise conjecture which regards this problem will be given later, see Conjecture~\ref{conj:torsionUltraTopological} and Problem~\ref{prob:remainingDetailed}.
\begin{prob}\label{prob:remaining}
    Find an ``if and only if'' criterion answering the second part of question 2 of Problem~\ref{prob:main}, as guaranteed by Theorem~\ref{intro:topological propery}.
\end{prob}

In some situations, repeated applications of Theorems \ref{intro: main thm2} and \ref{intro: main thm1} in different degrees give enough information to answer Problem~\ref{prob:remaining}.
This is the case whenever the representation in question does not weakly contain irreducible representations that are cohomological in four successive degrees $q-2,q-1,q,q+1$, and are inseparable by neighborhoods in $\widehat{G}$ (see Remark \ref{rem:4-degrees} below).
We exploit this fact in \S\ref{sec6} to obtain a satisfying picture for the cohomology of groups of type $\SO^\circ(2k,1)$. In \S\ref{subsec:6.4}, we complete the picture also for the group $\SO^\circ(2k+1,1)$, using a different approach and machinery from the theory of Lie algebra cohomology. In \S\ref{sec7} We present a different resolution for the specific case of $\SO^\circ(3,1) \cong \PSL_2(\C)$, using an explicit construction and tools from $3$-manifolds theory. For the group $\mathrm{SU}(n,1)$ we get a partial picture, which we complete in the specific case of $\mathrm{SU}(2,1)$ using the Poincar\'e duality given in Theorem~\ref{intro:homology-cohomology}.

\subsection{Higher property (T), torsion and isolation} \label{subsec:isolation}

Fixing for this subsection a non-compact semisimple group $G$ (see Terminology~\ref{terminology:semisimple}),
we will pay special attention to two significant cohomological invariants which answer the following two questions.

\begin{question} \label{que:higherT}
    What is the first degree $n$ in which there exists a unitary representation $V$ with $V^G=\{0\}$ having non-trivial cohomology, that is $H^n(G;V)\neq 0$? 
\end{question}

\begin{question} \label{que:highertor}
    What is the first degree $m$ in which there exists a unitary representation $V$ having non-trivial torsion cohomology, that is $\torH^m(G;V)\neq 0$? 
\end{question}

For the discussion below fix the invariants $n=n(G)$ and $m=m(G)$ answering the above questions correspondingly.
We compare them to the classical invariants $r(G)$, the rank of $G$, and $d(G)$, the dimension of its symmetric space.
The following theorem can be viewed as a generalization of the classical fact that $G$ has property (T) if and only if the trivial representation is isolated in $\widehat{G}$.

\begin{mainthm}\label{intro: theorem isolation}
    Let $G$ be a non-compact semisimple group. Then $n(G)=q$, where $q$ is the first degree in which there exists a non-trivial representation in $\widehat{G}_{q-\mathrm{cohom}}$, 
    and $m(G)=p+1$, where $p$ is the first degree in which there exists a representation in $\widehat{G}_{p-\mathrm{cohom}}$ which is not isolated in $\widehat{G}$.
\end{mainthm}

\begin{remark}
Since the cohomology of irreducible representations is always Hausdorff, the analogue of Question~\ref{que:higherT} regarding the Hausdorff cohomology has the same answer $n(G)$ as Question~\ref{que:higherT}.
\end{remark}

\begin{mainthm} \label{intro: m>n>1}
    Let $G$ be a non-compact semisimple group.
    If $G$ has property $(T)$ then $m(G)>n(G)>1$, otherwise $m(G)=n(G)=1$.
    Denoting the simple factors of $G$ by $G_i$, we have $n(G)=\min n(G_i)$ and $m(G)=\min m(G_i)$.
\end{mainthm}

In view of Theorem~\ref{intro: m>n>1}, we assume now that $G$ is a simple group with property (T).
The following theorem was proven by Dymara-Januszkiewicz under the extra assumption that the corresponding residue field is sufficiently large, see \cite{Dymara-Januszkiewicz}.

\begin{mainthm} \label{intro: p-adic}
    If $G$ is a simple group of higher-rank over a non-archimedean local field then $n(G)=r(G)$ and $m(G)=\infty$, that is the cohomology with unitary coefficients is always Hausdorff.
\end{mainthm}

In view of Theorem~\ref{intro: p-adic}, we assume for the rest of this subsection that $G$ is a simple Lie group.
The recent paper \cite{BaderSauer} was devoted to dealing with the invariant $n(G)$, that is answering Question~\ref{que:higherT} (and the analogous question regarding lattices).
According to \cite[Definition 1.1]{BaderSauer}, $G$ is said to have property (T$_{n(G)-1}$), which is alluded to as \emph{higher property T}.
In \cite[Theorem A]{BaderSauer} it is shown that $n(G)\geq r(G)$ and the invariant $n(G)$ is explicitly computed in \cite[Theorem B and Appendix A]{BaderSauer}.
Generalizing the fact that some rank-$1$ groups satisfy property (T),
we observe the interesting phenomenon that the inequality $n(G)\geq r(G)$ might be strict.
For example, for $\mathrm{Sp}(p,q)$ with $p\geq q$ we have $r(G)=q$ and $n(G)=2q$
and for $E_8^{-24}$ we have $r(G)=4$ and $n(G)=24$.
Examining the table in \cite[Appendix A]{BaderSauer} shows that these examples are in fact extremal, thus the invariant $n(G)$ is commensurable to the rank.

\begin{mainthm}[{\cite[Theorem B and Appendix A]{BaderSauer}}] \label{intro:Lie-n}
 For a non-compact simple Lie group $G$, we have $r(G)\leq n(G) \leq 6r(G)$. For classical groups groups, we have $n(G) \leq 2r(G)$.
\end{mainthm}

The invariant $m(G)$ associated with Question~\ref{que:highertor} was considered in \cite[Definition 30]{BaderNowakDeformation}, along with a terminology which is now dated, so we will not repeat it here.
In view of Theorem~\ref{intro: theorem isolation}, it is related to the phenomenon of isolation of cohomological representations.
This phenomenon was studied by Vogan \cite{voganIsol}, Bergeron \cite{BergeronIsolation, BergeronLefschetz} and Bergeron-Clozel \cite[Section~5.4]{BergeronClozel}. The general philosophy is that ``most'' of the cohomological representations are isolated in $\widehat{G}$, and it is conjectured that they are isolated in the automorphic dual associated with any rational structure under some assumptions \cite[Conjecture~3.1]{BergeronIsolation}. 
Supporting this philosophy, we observe that $m(G)$ grows along with the dimension $d(G)$ of the symmetric space associated with $G$, rather than with its rank. 
On the other hand, non-isolated cohomological (tempered) representations always exist in the mid-dimension.

\begin{mainthm} \label{intro:Lie-m}
    For every non-compact simple Lie group $G$ with property (T),
    we have 
    \[ \sqrt{d(G)} \leq m(G) \leq \lceil d(G)/2 \rceil. \]
\end{mainthm}

It seems within reach to compute the invariant $m(G)$ explicitly for every 
simple Lie group, thus constructing a table analogue to the one given in \cite[Appendix A]{BaderSauer}. Exact computations of $m(G)$ for the groups $\SO^\circ(p,q), \mathrm{SU}(p,q)$ and $Sp(p,q)$ are given in Theorems \ref{thm:so(p,q)}, \ref{thm:su(p,q)}, \ref{thm:sp(p,q)}.

\subsection{Cohomological witnesses and the (co)homology of finite type groups}

Given a locally compact group $G$ and a dense subgroup $\Gamma$,
the restriction of an irreducible unitary representation of $G$ to $\Gamma$ is still irreducible, thus we obtain a map $\widehat{G}\to \widehat{\Gamma}$.
This map is easily seen to be continuous.
We denote by $\widehat{\Gamma}_G$ the closure of the image of this map in $\widehat{\Gamma}$.

\begin{definition}
    We say that a dense subgroup $\Gamma<G$ is a \emph{cohomological witness} if the following conditions are satisfied.
    \begin{enumerate}
       \item The restriction map $\widehat{G}\to \widehat{\Gamma}$ is a homeomorphism onto its image, identifying $\widehat{G}$ as a subspace of $\widehat{\Gamma}$.
       \item For any unitary $G$-representation $\pi$, if $\pi|_\Gamma$ weakly contains a $\Gamma$-representation $\rho$ such that $\barH^q(\Gamma;\rho) \ne 0$, then $\pi$ weakly contains a representation from $\widehat{G}_{q-\mathrm{cohom}}$.
       \item For every representation $\pi$ of $G$, the map induced by the restriction, $H^q(G;\pi) \to H^q(\Gamma;\pi)$, is a topological isomorphism.
    \end{enumerate}
    We say that $\Gamma$ is a \emph{cohomological witness up to degree $N$} if the above conditions hold under the further assumption $q\leq N$.
\end{definition}

The following theorem extends \cite[Theorem F]{BaderSauer}.

\begin{mainthm}\label{intro:witness-Lie}
  Let $k$ be a number field and let $\mathbf{G}$ be a
connected and simply connected almost simple $k$-algebraic group. 
Then $\mathbf{G}(k)<\mathbf{G}(k \otimes \mathbb{R})$ is a cohomological witness.
\end{mainthm}

In fact, we have the following general result.

\begin{mainthm}\label{intro:witness}
Let $K$ be a number field and $\mathbf{G}$ a
connected and simply connected almost simple $K$-algebraic group.
Let $S$ be a finite set of places that contains all the Archimedean places, and let $S \subseteq T$ be another set of places such that $T - S$ contains at least one place over which $\mathbf{G}$ is isotropic. Denote by $T_{\mathrm{is}} \subseteq T$ the set of places in $T$ where $\mathbf{G}$ is isotropic.  Let $\Gamma$ be a group in the (well defined) commensurability class of $\mathbf{G}(\mathcal{O}_T)$ and $G=\prod_{v \in S} \mathbf{G}(K_v)$, along with the natural embedding $\Gamma\to G$. Then $\Gamma$ is a cohomological witness for $G$ up to degree $|T_{\mathrm{is}}|-1$. In particular, if $T$ is infinite it is a cohomological witness.
\end{mainthm}

Note that the countable group $\mathbf{G}(k)$ is not finitely generated.
Approximating it with $S$-arithmetic subgroups, for a finite, but arbitrary large set of places $S$, we will obtain the following.

\begin{mainthm}\label{intro:finite witness}
Let $G=\prod_{i=1}^n \mathbf{G}_i(k_i)$ be a semisimple group, where each $k_i$ is a local field of characteristic zero and each factor $\mathbf{G}_i$ is an almost simple, connected and simply connected $k_i$-algebraic group. Then for every $N$ there exists a countable dense subgroup $\Gamma<G$ of type $\mathrm{FP}_\infty(\Q)$ which is a cohomological witness up to degree $N$. In particular, there exists a countable dense subgroup $\Gamma<G$ of type $\mathrm{FP}_\infty(\Q)$
which is a cohomological witness up to the cohomological degree of $G$, $d(G)$.
\end{mainthm}

The cohomology of groups of finite type with unitary coefficients is nicely behaved,
thus Theorem~\ref{intro:finite witness} will serve as a major tool for proving many of our theorems.
In \S\ref{sec3} we will explain how groups of finite type admit, for any unitary representation, a cochain complex of Hilbert spaces that calculates the corresponding cohomology.
The cochain spaces are amplifications of the given unitary representation and the boundary maps are given as matrices over the group algebra. This allows us to use operator theoretical techniques to study the cohomology. The main idea is that if we consider the matrix coefficients of these representations (or more precisely states of a matrix amplification of $C^*(\Gamma)$), we can use the Fell topology and the compactness of the state space to translate approximate properties of a representation to properties of a weakly contained representation.
With these tools, we prove analogous results to Theorems \ref{intro: main thm1} and \ref{intro: main thm2} for groups with finiteness properties, which are of independent interest.

\begin{mainthm}\label{thm:main_1_finiteness_intro}
    Let $\Gamma$ be a group of type $\mathrm{FP}_\infty(\Q)$ and let $\pi$ be a unitary representation of $\Gamma$. For every degree $q\in\N$ we have
    \[ \torH^q(\Gamma;\pi)\neq 0 ~\Longrightarrow ~ \mathrm{supp}(\pi) \cap \widehat{\Gamma}_{\mathrm{q-cohom}} \neq \emptyset ~~\text{and}~~ \mathrm{supp}(\pi) \cap \widehat{\Gamma}_{\mathrm{(q-1)-cohom}} \neq \emptyset.\]
\end{mainthm}

\begin{mainthm}\label{thm:main_2_finiteness_intro}
    Let $\Gamma$ be a group of type $\mathrm{FP}_\infty(\Q)$ and let $\sigma$ be a unitary representation of $\Gamma$. Fix a degree $q$, and assume that $H^q(\Gamma;\sigma) \ne 0$. Let $\pi$ be a unitary representation that weakly contains $\sigma$. Then, at least one of the following occurs:
    \begin{enumerate}
        \item  $\bar{H}^q(\Gamma;\pi) \ne 0$.
        \item  $\torH^q(\Gamma;\pi) \ne 0$.
        \item  $\torH^{q+1}(\Gamma;\pi) \ne 0$.
    \end{enumerate}
    In other words, $\pi$ is $q$-cohomological or $(q+1)$-cohomological. If it is not $q$-cohomological, then $H^{q+1}(\Gamma;\pi)$ has non-zero torsion cohomology.
\end{mainthm}

In \S\ref{sec5} we use Theorems \ref{thm:main_1_finiteness_intro} and \ref{thm:main_2_finiteness_intro}, together with Theorem~\ref{intro:finite witness}, to deduce Theorems \ref{intro: main thm2} and \ref{intro: main thm1}. 
We also use these theorems to study lattices in semisimple groups.

\begin{mainthm}\label{intro:latt corollary 1}
    Let $G$ be a higher-rank simple Lie group with trivial center and $\Gamma$ a uniform lattice in $G$. Denote by $d(G)=\mathrm{dim}(G/K)$ the dimension of the symmetric space and by $\ell(G)=\mathrm{rank}_\mathbb{C}(G)-\mathrm{rank}_\mathbb{C}(K)$ the fundamental rank. Assume that $\ell(G)>0$. Let $\pi$ be a unitary representation of $\Gamma$ which is not a finite sum of amplifications of finite dimensional representations. Then $\pi$ is cohomological. Moreover, $H^q(\Gamma;\pi)$ has non-zero torsion for every $q \in [\frac{d(G)-\ell(G)}{2}+1,\frac{d(G)+\ell(G)}{2}]$.
\end{mainthm}
\begin{mainthm}\label{intro:latt corollary 2}
    Let $G$ be a higher-rank simple Lie group which is center free, or a higher-rank simple $p$-adic group. If $G$ is real, assume that it has discrete series representations, namely that the fundamental rank is zero. Let $\Gamma$ be a lattice in $G$. Then, $H^*(\Gamma;\pi) \ne 0$ for every unitary representation $\pi$. In particular, if $G$ is $p$-adic, then $H^{\mathrm{rank}(G)}(\Gamma;\pi)\ne 0$ for every  non-trivial unitary representation $\pi$.
\end{mainthm}

Theorems \ref{intro:latt corollary 1} and \ref{intro:latt corollary 2} can be considered as analogues of a classical result of Guichardet, stating that any unitary representation $\pi$ of the free group $F$ has $H^1(F;\pi) \neq 0$ \cite[Example~1]{Guichardet72}.

\subsection{Structure of the paper}
The heart of the paper is in \S\ref{sec3} and \S\ref{sec5}, in which we develop the theory of torsion in cohomology for groups with finiteness properties and semisimple groups, respectively. In \S\ref{sec:2} we set the terminology for the rest of the paper, recall some preliminaries on continuous cohomology and unitary representations and prove the first results. In particular, we explain how to decompose unitary representations of semisimple groups into a direct sum of an atomic representation with Hausdorff cohomology and a representation whose cohomology (if exists) is purely torsion, and prove Theorem \ref{intro-Hausdorff}. In \S\ref{sec:4} we discuss the theory of continuous homology of locally compact groups. We discuss the relation between the homology and cohomology, and prove a version of the Poincar\'e duality theorem, building on results of Blanc, Wigner and Pichaud on differentiable homology and cohomology. We end with proving Theorem \ref{intro:homology-cohomology}. In \S\ref{sec3} we study groups with finiteness properties. We focus on torsion cohomology, and prove Theorems \ref{thm:main_1_finiteness_intro}, \ref{thm:main_2_finiteness_intro}, \ref{intro:latt corollary 1} and \ref{intro:latt corollary 2}. \S\ref{sec5} is devoted to the study of unitary cohomology for semisimple groups, with emphasis on the torsion. We start with proving Theorem \ref{intro:torsion-reduced}.
In \S\ref{subsec:5.1} we prove Theorem~\ref{intro:witness} and deduce its two corollaries, Theorem~\ref{intro:witness-Lie}
and Theorem \ref{intro:finite witness}.
In \S\ref{subsec:5.2} we prove Theorems \ref{intro: main thm2} and \ref{intro: main thm1}, building on the results of \S\ref{sec3} and the key Theorem \ref{intro:finite witness}. We also prove Theorems \ref{intro:topological propery1} and \ref{intro:topological propery}. We explain the difficulty in determining the precise degree where non-zero torsion appears in cohomology (Remark \ref{rem:4-degrees}). 
The theorems of \S\ref{subsec:isolation} are proved in \S\ref{subsec:isolation proofs}.
In \S\ref{sec6} we use our results in order to study unitary cohomology for the real rank-$1$ groups $\mathrm{SO}^\circ(n,1)$ and $\mathrm{SU}(n,1)$. We fully resolve Problem \ref{prob:main} for the group $\SO^\circ(n,1)$. Finally in \S\ref{sec7} we zoom in to the case of $\PSL_2(\C)$ and provide a second resolution of Problem~\ref{prob:main} for this group, using a different approach.

\subsection{AI disclosure}
The paper was essentially written before mid 2026, and no AI tools were used to develop the ideas, arguments, results and proofs in it, nor to draft the manuscript. AI was used in retrospect to referee the paper and check the correctness of the proofs. We also asked it to compose this AI disclosure, but we ended up completely rewriting it. The authors take full responsibility for the results.

 \subsection{Acknowledgments}
We thank Yehuda Shalom for many fruitful discussions. We thank Saar Bader for going over the paper and suggesting useful insights.

\section{Preliminaries and terminology regarding continuous cohomology}\label{sec:2}

In this section we give some background regarding continuous cohomology of locally compact groups.
In \S\ref{prelim: continuous cohomology} we discuss continuous cohomology in general, in \S\ref{subsec:2.2} we discuss the unitary dual and in \S\ref{subsec:Cohomology with unitary coefficients} we combine the two and discuss cohomology with unitary coefficients.
Finally, in \S\ref{subsec:Semisimple groups} we specialize the above to semisimple groups.

\subsection{Continuous cohomology} \label{prelim: continuous cohomology}
Throughout this section, $G$ denotes a locally compact, second countable group. A \emph{Fr\'echet $G$-module} is a Fr\'echet space $E$ equipped with a strongly continuous representation of $G$. To any Fr\'echet $G$-module, one can associate a sequence of invariants $H^q_c(G;E)$, \emph{called the continuous cohomology of $G$ with coefficients in $E$}. These are topological vector spaces that are not necessarily Hausdorff. We will drop the ``c" from the notation and write them simply as $H^q(G;E)$. We will briefly recall the relevant theory. For a full discussion see Blanc's paper~\cite{Blanc} or Guichardet's book~\cite{Guichardet}. A good reference in English is Petersen's thesis~\cite{Petersen}.

An injection of $G$-modules $u\colon E \hookrightarrow F$ is called \emph{strong} if it admits a continuous $\C$-linear left inverse. A general map of $G$-modules $u\colon E \to F$ is called \emph{strong} if the injections $\mathrm{ker} u \hookrightarrow E$ and $E/ \ker u \hookrightarrow F$ are strong.

\begin{definition}\label{def:rel inj}
    A $G$-module $I$ is called \emph{relatively injective} if for any strong injection $u\colon E \hookrightarrow F$ and any continuous linear $G$-map $v\colon E \to I$, there exists a continuous linear $G$-map $w\colon F \to I$ extending $v$.
\[
\begin{tikzcd}[cramped]
	& I \\
	F & E & 0
	\arrow["{\exists ?w}", dotted, from=2-1, to=1-2]
	\arrow["v"', from=2-2, to=1-2]
	\arrow["u", from=2-2, to=2-1]
	\arrow[from=2-3, to=2-2]
\end{tikzcd}
\]
\end{definition}
Consider a resolution of a $G$-module $E$:
\[
\begin{tikzcd}[cramped]
	{\cdot \cdot \cdot} & {I_2} & {I_1} & {I_0} & E & 0
	\arrow["{d_3}"', from=1-2, to=1-1]
	\arrow["{d_2}"', from=1-3, to=1-2]
	\arrow["{d_1}"', from=1-4, to=1-3]
	\arrow["\epsilon"', from=1-5, to=1-4]
	\arrow[from=1-6, to=1-5]
\end{tikzcd}
\]
The resolution is called \emph{strong} if $d_q$ and $\epsilon$ are strong. If the resolution has a continuous (not necessarily $G$-equivariant) cochain contraction, then it is strong~\cite[2.6~Remarque]{Blanc}. 
The resolution is called \emph{relatively injective} if the $G$-modules $I_q$ are relatively injective. The category $\mathcal{F}_G$ of Fr\'echet $G$-modules admits enough relatively injective objects. Namely, every Fr\'echet $G$-module $E$ has a strong, relatively injective resolution. Any two such resolutions are homotopy equivalent. We can therefore make the following definition, which is not vacuous and does not depend on the chosen resolution.
\begin{definition}
    Let $E$ be a $G$-module and $(I_q,d_q)$ a strong, relatively injective resolution of $E$. Denote by $E^G$ the space of $G$-invariants of $E$. The continuous cohomology of $G$ with coefficients in $E$ is the homology of the complex:
\[\begin{tikzcd}[cramped]
	{\cdot \cdot \cdot} & {I_2^G} & {I_1^G} & {I_0^G} & 0
	\arrow["{d_3}"', from=1-2, to=1-1]
	\arrow["{d_2}"', from=1-3, to=1-2]
	\arrow["{d_1}"', from=1-4, to=1-3]
	\arrow[from=1-5, to=1-4]
\end{tikzcd}\]
Namely, $H^q(G;E) = \ker d_{q+1}/\operatorname{Im} d_q$.
\end{definition}

Since $I_q^G$ are Fr\'echet spaces and the maps $d_q$ are continuous, the cohomology $H^q(G;E)$ is a topological vector space. The topology does not depend on the chosen resolution. Since $\operatorname{Im} d_q$ need not be closed, $H^q(G;E)$ might not be Hausdorff. We write  
\begin{align*}
\bar{H}^q(G;E) &= H^q(G;E)/\overline{\{0\}};\\
\torH^q(G;E) &= \overline{\{0\}}\subset H^q(G;E).
\end{align*} These spaces are called the \emph{Hausdorff cohomology}\footnote{
The Hausdorff cohomology is also known in the literature as the reduced cohomology. We find this terminology confusing and will not use it
.} and the \emph{torsion cohomology} of $G$ with coefficients in $E$ respectively.
\begin{definition}
    We use the following terminology regarding the cohomology of $G$ with coefficients in $E$ in a given degree $q$.
    \begin{enumerate}
        \item The cohomology is \emph{Hausdorff} if $H^q(G;E) \cong \bar{H}^q(G;E)$.
        \item The cohomology is \emph{non-Hausdorff}, or \emph{has non-zero torsion}, if $\torH^q(G;E) \ne 0$.
        \item The cohomology is \emph{purely non-Hausdorff}, or \emph{purely torsion}, if $H^q(G;E) \cong \torH^q(G;E)$.
    \end{enumerate}
\end{definition}
We will give an explicit construction of a strong, relatively injective resolution. Let $E$ be a Fr\'echet $G$-module. Consider the spaces $L^2_\mathrm{loc}(G^q,E)$, of locally square integrable maps from $G^q$ to $E$. By definition this is the inverse limit:
$$L^2_\mathrm{loc}(G^q,E) = \varprojlim L^2(K,E_r)$$
Where the inverse limit is taken over all continuous semi-norms $r$ on $E$ and all compact subsets $K \subseteq G^q$. Since $G$ is assumed to be second countable, the inverse limit can be taken over a countable system, and $L^2_\mathrm{loc}(G^q,E)$ is again a Fr\'echet space. We endow the space $L^2_\mathrm{loc}(G^q,E)$ with a $G$-action given by
\[[g.f](\_) = \pi(g)f(g^{-1}\cdot\_)\]
and define the coboundary map $d_q\colon L^2_{\mathrm{loc}}(G^{q},E) \rightarrow L^2_{\mathrm{loc}}(G^{q+1},E)$ as
\begin{align*}d_q(f)(g_0,...,g_q) &= \sum_{j=0}^{q}(-1)^jf(g_0,...,\hat{g_j},...,g_q);\\
\epsilon(e)(g) &= e.
\end{align*}
This yields the following cochain complex known as the \emph{homogeneous bar resolution}:
\[
\begin{tikzcd}[cramped]
	{\cdot \cdot \cdot} & {L^2_{\mathrm{loc}}(G^3,E)} & {L^2_{\mathrm{loc}}(G^2,E)} & {L^2_{\mathrm{loc}}(G,E)} & E & 0
	\arrow["{d_3}"', from=1-2, to=1-1]
	\arrow["{d_2}"', from=1-3, to=1-2]
	\arrow["{d_1}"', from=1-4, to=1-3]
	\arrow["\epsilon"', from=1-5, to=1-4]
	\arrow[from=1-6, to=1-5]
\end{tikzcd}
\]
This is a strong, relatively injective resolution for $E$ \cite[Theorem~C.7]{Petersen}.
\begin{remark}\label{rem:inhomogeneous}
    To calculate the cohomology, one needs to take invariants in the resolution above (or any other strong, relatively injective resolution) and take the homology of the resulting complex. The $G$-invariant vectors in $L^2_{\mathrm{loc}}(G^{q+1},E)$ are just left $G$-equivariant maps. Such a function $f$ is determined by the values $f(1,g_1, g_1g_2,.....,g_1\cdot \cdot \cdot g_q)$ for $(g_1,...,g_q) \in G^q$. This identification yields an isomorphism $(L^2_{\mathrm{loc}}(G^{q+1},E))^G \cong L^2_{\mathrm{loc}}(G^q,E)$. Conjugating the coboundary maps $d_q$ with these isomorphisms, we get the following cochain complex which calculates $H^n(G;E)$. Its elements are called \emph{inhomogeneous} cochains.
    \[\begin{tikzcd}[cramped]
	{\cdot\cdot\cdot} & {L^2_{\mathrm{loc}}(G^2,E)} & {L^2_{\mathrm{loc}}(G,E)} & E & 0
	\arrow[from=1-2, to=1-1]
	\arrow["{d_2'}"', from=1-3, to=1-2]
	\arrow["{d_1'}"', from=1-4, to=1-3]
	\arrow[from=1-5, to=1-4]
    \end{tikzcd}\]
The coboundary maps $d_q'$ are given by the formula:    
    $$d_q'(f)(g_1,....,g_q) = \pi(g_1)f(g_2,...,g_q) + \sum_{j=1}^{q-1} (-1)^jf(g_1,...,g_j g_{j+1},...,g_q) + (-1)^q f(g_1,...,g_{q-1})$$
\end{remark}

\begin{remark}\label{rem:lp loc}
    The spaces $L^2_\mathrm{loc}(G^q,E)$ can be replaced by the spaces of continuous functions, $C(G^q,E)$. We chose to give the $L^2_{\mathrm{loc}}$-resolution as it will be easier to work with in several occasions in this paper. Similarly, we can work with the spaces $L^p_{\mathrm{loc}}$ for $1 \leq p < \infty$. These spaces are defined analogously to $L^2_{\mathrm{loc}}$. They give a strong, relatively injective resolution \cite[Corollaire~3.5]{Blanc}, and we have a complex of inhomogeneous cochains for them \cite[Proposition~4.10]{Blanc}.
\end{remark}
Let $V$ be a Fr\'echet $G$-module. For $p \in [1, \infty)$, we denote by $\ell^p(V) = \ell^p(\N;V)$ the Fr\'echet space of all sequences $(v_i)_i \in V$ such that for any continuous semi-norm $q$ on $V$, $q(v_i)$ is $p$-summable. If $T\colon E \to F$ is a map of Fr\'echet spaces, then there is a natural induced map $\tilde{T}\colon \ell^p(E) \to \ell^p(F)$, making $\ell^p$ into a functor. In the sequel we will need to understand the cohomology of $\ell^p(V)$ in terms of the cohomology of $V$.

\begin{theorem}\label{thm:cohom of amplif}
    Let $(V,\pi)$ be a Fr\'echet $G$-module and let $p \in [1, \infty)$.
    \begin{enumerate}
        \item $H^q(G;V)$ is Hausdorff if and only if $H^q(G;\ell^p(V))$ is Hausdorff.
        \item The natural map $\bar{H}^q(G;\ell^p(V)) \to \ell^p(\bar{H}^q(G;V))$ is an isomorphism.
    \end{enumerate}
\end{theorem}
The theorem is a consequence of the following lemma:

\begin{lemma}
    The functor $\ell^p$ is exact on the category of Fr\'echet spaces. In other words if $T\colon E \to F$ is a surjective continuous linear map of Fr\'echet spaces and $\tilde{T}\colon\ell^p(E) \to \ell^p(F)$ is the induced map on the amplifications then $\ker(\tilde{T}) = \ell^p(\ker(T))$ and $\mathrm{Im}(\tilde{T}) = \ell^p(F)$.
\end{lemma}

\begin{proof}
    The fact that $\ker(\tilde{T}) = \ell^p(\ker(T))$ is obvious. We will prove that $\tilde{T}$ is surjective. Let $p_n$ be an increasing sequence of semi-norms generating the topology of $E$. Denote by $q_n$ the associated quotient semi-norms, defined by:
    $$q_n(f) = \inf \{p_n(e) \mid T(e) = f\}$$
    For an element $f=(f_i)_i$ of $\ell^p(E)$ (resp. $\ell^p(F)$), we denote by $\|p_n(f)\|_p$ (resp. $\|q_n(f)\|_p$) the $p$-norm of the $\ell^p$-sequence $p_n(f_i)$ (resp. $q_n(f_i)$). The image of $\tilde{T}$ contains all finitely supported elements, which are dense in $\ell^p(F)$. Let $f=(f_i)_i$ be an element of $\ell^p(F)$. We will show that we can find a Cauchy sequence of finitely supported elements of $\ell^p(F)$ which converges to $f$, and which can be lifted to a Cauchy sequence in $\ell^p(E)$.
    
    For each $n \geq 1$, we choose a finitely supported element $f^n=(f_i^n)_i \in \ell^p(F)$ such that $\|(q_{n+1}(f-f^n)\|_p \leq 2^{-(n+2)}$. Let $e^1=(e_i^1)_i$ be some lift of $f^1$. Since the element $f^n-f^{n-1} \in \ell^p(F)$ is finitely supported, it can can be lifted to a finitely supported element $s^n=(s_i^n)_i \in \ell^p(E)$. Furthermore, by the definition of $q_n$, we can require that in each coordinate $i$ we have $p_n(s_i^n) \leq 2q_n(f_i^{n} -f_i^{n-1})$. But the element $f^n-f^{n-1}$ satisfies $\|(q_n(f^n-f^{n-1})\|_p \leq 2^{-(n+1)}$ (this uses the fact that the sequence of semi-norms $p_n$ is increasing), so the coordinate-wise inequalities imply that $\|p_n(s^n)\|_p \leq 2^{-n}$. We can now define: $$e^n=(e_i^n)_i = (e_i^{n-1}+s_i^n)_i = (e_i^1 + \sum_{k=2}^n s_i^k )_i$$
    Then $e^n$ is a lift of $f^n$. $e^n$ is a Cauchy sequence in $\ell^p(E)$ with respect to all the defining semi-norms $\|p_k(\cdot)\|_p$, since $\|p_n(s^n)\|_p \leq 2^{-n}$. Therefore $e^n$ converges in $\ell^p(E)$ to some element $e=(e_i)_i$. Continuity of $\tilde T$ now gives $\tilde T(e)=f$, and $\tilde T$ is surjective as needed.
\end{proof}
We can now prove Theorem \ref{thm:cohom of amplif}.
\begin{proof}[Proof of Theorem \ref{thm:cohom of amplif}]
    We start with Claim $(1)$. The first direction is trivial: if $V$ has non-zero torsion in cohomology then so does $\ell^p(V)$, since it contains $V$ as a direct summand. Assume that $H^q(G;V)$ is Hausdorff. We will use the inhomogeneous cochain complexes $L^p_{\mathrm{loc}}(G^n,V)$ and $L^p_{\mathrm{loc}}(G^n,\ell^p(V)) $ (see Remark \ref{rem:inhomogeneous} and Remark \ref{rem:lp loc}) to compare the relevant cohomologies. Note that the $p$ chosen for the $L^p_\mathrm{loc}$-resolution is the same as the $p$ appearing in the $\ell_p$-amplification. We denote the coboundary maps corresponding to these complexes by $d_n$ and $d_n^\infty$ respectively. An element in $L^p_{\mathrm{loc}}(G^n,\ell^p(V))$ can be regarded as a sequence $(f_i)_i$ with $f_i \in L^p_{\mathrm{loc}}(G^n,V)$, such that for any continuous semi-norm $r$ on $V$, the sequence of $p$-norms of the functions $r\circ f_i$ restricted to $K$ is $p$-summable for every compact $K \subseteq G^n$. In other words, $L^p_{\mathrm{loc}}(G^n,\ell^p(V)) \cong \ell^p(L^p_\mathrm{loc}(G^n,V))$, and the boundary map $d_n^\infty$ is given by $d_n^\infty((f_i)_i)=\tilde d_n((f_i)_i)=(d_n(f_i))_i$. Namely, the functor $\ell^p$ takes the Fr\'echet complex of inhomogeneous cochains which calculates $H^*(G;V)$ to the Fr\'echet complex of inhomogeneous cochains which calculates $H^*(G;\ell^p(V))$. Since $\ell^p$ is an exact functor on the category of Fr\'echet spaces, and the map $d_q$ has a closed image, the map $\tilde d_q=d_q^\infty$ also has a closed image. This shows Claim $(1)$. Claim $(2)$ follows from the same identification of the complexes, and from the fact that the functor $\ell^p$ commutes with closures of images. 
\end{proof}

\subsection{The unitary dual}\label{subsec:2.2}
We recall some basic definitions and notions regarding the unitary dual of locally compact groups. As in the previous subsection, $G$ stands for a locally compact, second countable group. All representations are assumed to be unitary, and strongly continuous. Furthermore, we keep a standing assumption that unitary representations are on separable Hilbert spaces. This is slightly more convenient, since we use direct integral decompositions in some of the proofs. For a more thorough introduction to the unitary dual and Fell topology, we refer the reader to \cite[\S 7]{Folland}, \cite[\S II.F]{bdlv} and \cite{Dixmier_cstar}.

A common theme in abstract harmonic analysis is that unitary representations should be studied via their associated \emph{matrix coefficients}. Given a representation $\pi$ on a Hilbert space $\mathcal H_{\pi}$ and two vectors $u,v \in \mathcal H_{\pi}$, the \emph{matrix coefficient associated with $u$ and $v$}, denoted by $c_{u,v}: G \rightarrow \mathbb{C}$, is the continuous function $g \mapsto \langle \pi(g)u,v\rangle $. Matrix coefficients lead to the notion of \emph{weak containment of representations}, which we now recall. 
\begin{definition}\label{def:weak cont}
    Let $\pi$ be a unitary representation of $G$, and $S$ a set of unitary representations of $G$. Then $\pi$ is said to be \emph{weakly contained} in $S$, denoted $\pi \prec S$, if every matrix coefficient of $\pi$ can be approximated, uniformly on compact subsets of $G$, by sums of matrix coefficients of representations from $S$.
\end{definition}
Note that if the group $G$ is countable, the approximation in Definition \ref{def:weak cont} is in the topology of pointwise convergence. One equivalent characterization of weak containment that will be important for us is in terms of representations of $L^1(G)$. Unitary representations of $G$ correspond to non-degenerate $*$-representations of the Banach $*$-algebra $L^1(G)$. For a unitary representation $\pi$, the associated $L^1(G)$ representation (which is also denoted by $\pi$) is given by $\pi(f)=\int\pi(g)f(g)d\mu(g)$ for $f \in L^1(G)$, where $\mu$ is the Haar measure on $G$. In this language, an equivalent definition is the following: $\pi$ is weakly contained in a set $S$ if and only if $|| \pi(f)|| \leq \sup_{\rho \in S}|| \rho(f) ||$ for every $f \in L^1(G)$.

We denote the set of equivalence classes of irreducible unitary representations of $G$ by $\widehat{G}$. This set is called the \emph{unitary dual} of $G$. We equip this set with the structure of a topological space.
\begin{definition}\label{def:Fell}
    A set $S \subseteq \widehat{G}$ is \emph{closed} if it is closed with respect to weak containment; namely, for every $\pi \in \widehat{G}$ for which $\pi \prec S$ we have that $\pi \in S$.
\end{definition}
The closed subsets in Definition \ref{def:Fell} define a topology on $\widehat{G}$ called the \emph{Fell topology}. Here and thereafter, we denote by $\lambda$ the left regular representation of $G$.
\begin{definition}\label{def:tempered and discrete}
    Let $\pi$ be an irreducible, unitary representation of $G$. Then,
    \begin{enumerate}
        \item $\pi$ is said to be in the \emph{discrete series} if it is a subrepresentation of $\lambda$.
        \item $\pi$ is called \emph{tempered} if it is weakly contained in $\lambda$.
    \end{enumerate}
\end{definition}

As a topological space, $\widehat{G}$ rarely has nice separation properties; the groups for which some separation properties hold are rather special.
\begin{definition}\label{def:typeI}
    A group $G$ is called \emph{type I} if $\widehat{G}$ equipped with the Fell topology is a $\mathrm{T}_0$-space. It is called \emph{CCR} if $\widehat{G}$ is a $\mathrm{T}_1$-space.
\end{definition}
Compact and abelian groups are of course CCR, as their unitary dual is Hausdorff. Semisimple algebraic groups over local fields, as well as semisimple Lie groups, are also CCR. Algebraic groups over local fields are always type I. A fundamental result of Thoma states that one can not expect separation properties from the unitary dual of discrete groups: a discrete group $\Gamma$ is type I if and only if it is virtually abelian. 

The unitary dual of $G$ is endowed with a measurable structure, called the \emph{Mackey-Borel structure} of $\widehat{G}$, which is defined in terms of the measurability of matrix coefficients as a functions in the unitary representation. It turns out the this measurable structure coincides with the Borel structure of the Fell topology if and only if $G$ is type I. In this case it is a standard Borel space. The importance of this fact comes from the theory of direct integrals and disintegration of unitary representations.

Let $(X, \Sigma)$ be standard Borel space and $\mu$ a measure on $X$. For the notions of \emph{measurable field of Hilbert spaces} $\{\mathcal{H}_{\alpha} \mid \alpha \in X\}$ and \emph{direct integral of Hilbert spaces} $\int^{\oplus}\mathcal H_{\alpha} d\mu(\alpha)$, on which one can define a \emph{measurable field of representations} and a \emph{direct integral of representations}, we refer to~\cite[\S 7.4]{Folland}. We will mostly be interested in disintegration of representations over \emph{standard} measure spaces. Such a disintegration into irreducible components always exists.

\begin{theorem}\cite[Theorem~7.29-7.30]{Folland}\label{thm:disintegration of unirpes}
    Let $\pi$ be a unitary representation of $G$. Then there exists a standard measure space $(X,\Sigma,\mu)$ and a measurable field of representations $\pi_{\alpha}$ on $X$ such that $\pi \cong \int^{\oplus}\pi_{\alpha}d\mu(\alpha)$, and $\pi_\alpha$ is $\mu$-almost everywhere irreducible.
\end{theorem}
Given a disintegration of a unitary representation, it will be important to know that the components of the decomposition are weakly contained in the representation. We include a proof of this fact for lack of good reference.
\begin{lemma}\label{lemma:disintegratipon weakly contained}
    Let $\pi \cong \int^{\oplus}\pi_{\alpha}d\mu(\alpha)$ be a disintegration of $\pi$ over a standard measure space $(X, \Sigma, \mu)$. Then, $\mu$-almost every $\pi_\alpha$ is weakly contained in $\pi$. Moreover, there exists a conull set $Y \subseteq X$ such that $\pi$ is weakly equivalent to the set $\{\pi_\alpha \mid \alpha \in Y\}$. 
\end{lemma}
\begin{proof}
    Let $\{f_n\}$ be a countable dense subset of $L^1(G)$. We have that $\pi(f_n) = \int^{\oplus}\pi_\alpha(f_n)d\mu(\alpha)$, and hence $||\pi(f_n)||= \mathrm{ess} \ \mathrm{sup}_\alpha \{||\pi_\alpha(f_n)||\}$. Therefore, there exists a conull set $Y_n$ such that $||\pi_\alpha(f_n)|| \leq || \pi(f_n) ||$ for every $\alpha \in Y_n$. Let $Y$ be the conull set $\cap_n Y_n$. Then, for every $f \in L^1(G)$ and every $\alpha \in Y$, $|\| \pi_\alpha(f)|| \leq ||\pi(f)||$. Hence, for every $\alpha \in Y$ we have that $\pi_\alpha \prec \pi$. In the other direction, we want to show that $\pi$ is weakly contained in $\{\pi_\alpha \mid \alpha \in Y\}$. Again from the fact that for every $f \in L^1(G)$ we have that $||\pi(f)||= \mathrm{ess} \ \mathrm{sup}_\alpha \{||\pi_\alpha(f)||\}$, we get in particular that $||\pi(f)|| \leq \mathrm{sup}_{\alpha \in Y} \{||\pi_\alpha(f)||\}$. This finishes the proof.
\end{proof}
The \emph{support} of a unitary representation~$\pi$ is the set of irreducible representations that are weakly contained in $\pi$. It is denoted by $\mathrm{supp}(\pi) \subseteq \widehat{G}$. From the above lemma, together with Theorem \ref{thm:disintegration of unirpes}, we get the following corollary.
\begin{corollary}
    Let $\pi$ be a unitary representation of $G$. Then $\pi$ is weakly equivalent to its support $\mathrm{supp}(\pi)$.
\end{corollary}
One possible way to decompose a representation into a direct integral of irreducible representations is over the unitary dual. If $A \subseteq \widehat{G}$ is a measurable subset of $\widehat{G}$ on which the Mackey Borel structure is standard, there is a measurable field of representations $\pi_\alpha$ on $A$ such that $\pi_\alpha$ belongs to the equivalence class $\alpha$ for every $\alpha \in A$. If $\mu$ is a measure on $\widehat{G}$ such that $\mu(\widehat{G}\setminus A)=0$, we may form the direct integral $\int^\oplus\pi_\alpha d\mu(\alpha)$. If $G$ is type I we can always take $A=\widehat{G}$ and decompose every representation with respect to a unique measure class. More precisely, we have the following \cite[Theorem~7.32]{Folland}.
\begin{theorem}\label{thm:type I groups}
    Let $G$ be a type I group, and $\pi$ a unitary representation of $G$. Then, there exist finite measures $\mu_1,\mu_2,\dots\mu_\infty$, that are mutually singular, such that $\pi$ is equivalent to $\rho_1 \oplus 2\rho_2 \cdots \oplus \infty \rho_\infty$, where $\rho_i = \int^\oplus\pi_\alpha d\mu_i(\alpha)$. These measures are uniquely determined up to equivalence.
\end{theorem}
For a type I group $G$, the closure of the union of the supports of the measures in Theorem~\ref{thm:type I groups} is precisely $\mathrm{supp}(\pi)$. We say that $\pi$ \emph{essentially weakly contains} an irreducible representation~$\rho$ if $\rho$ is weakly contained in the orthogonal complement to the isotypic subrepresentation of $\pi$ that consists of an amplification of $\rho$. We denote the set of non-isolated points in $\mathrm{supp}(\pi)$ by $\mathrm{supp}(\pi)'$. A representation~$\rho$ is essentially weakly contained in $\pi$ if and only if $\rho \in \mathrm{supp}(\pi)'$

When dealing with weak containment of irreducible representations, the following two lemmas will be useful.
\begin{lemma}\cite[Proposition~F.1.4]{bdlv}\label{lem:irred rep weak cont}
    Let $\pi$ and $\rho$ be two representations of $G$ such that $\pi \prec \rho$. Assume that $\pi$ is irreducible. Then every matrix coefficient associated with $\pi$ can be approximated, uniformly on compacta, by matrix coefficients associated with $\rho$ (and not just by convex combinations).
\end{lemma}
The following is an easy consequence of Lemma \ref{lem:irred rep weak cont}. We leave to the reader to verify its proof.
\begin{lemma}\label{weakContain}
    Let $G$ be a locally compact group, $\sigma_1$ and $\sigma_2$ unitary representations of $G$, and $\pi$ an irreducible unitary representation of $G$. If $\pi \prec \sigma_1 \oplus \sigma_2$, then $\pi \prec \sigma_1$ or $\pi \prec \sigma_2$.
\end{lemma}

\subsection{Cohomology with unitary coefficients} \label{subsec:Cohomology with unitary coefficients}
Let $G$ be a locally compact, second countable group. 
%Our main object of interest in this paper is the continuous cohomology $H^*(G;V)$, where $V$ is a unitary $G$-module. 
If $\pi$ is a unitary representation of $G$ on a Hilbert space $V$, we will use interchangeably the notations $H^*(G;\pi)$ and $H^*(G;V)$ which refer to the same object for us. This inconsistency is a product of the fact that we will sometimes need to emphasize representation-theoretic aspects of $\pi$, and other times just work with the space $V$ itself.

For a unitary representation $\pi$, we will say that it is \emph{cohomological} if the cohomology $H^*(G;\pi)$ is non-zero. We will say that it is $q$-\emph{cohomological} if it is cohomological in degree $q$; namely, if $H^q(G;\pi) \ne 0$. The set of irreducible cohomological representations of $G$ is called the \emph{cohomological dual} of $G$, and is denoted by $\widehat{G}_{\mathrm{cohom}} \subseteq \widehat{G}$. The $q$-\emph{cohomological dual} is the set of irreducible $q$-cohomological representations, $\widehat{G}_{\mathrm{q-cohom}} \subseteq \widehat{G}_{\mathrm{cohom}}$. The following proposition tells us that the Hausdorff cohomology with unitary coefficients works well with disintegration. In many cases it allows us to fully understand the Hausdorff cohomology with unitary coefficients in terms of the cohomological dual, having Theorem \ref{thm:disintegration of unirpes} in hand.
\begin{proposition}\cite[Theorem~7.2]{Blanc}\label{prop:red cohom direct integral}
    Let $\pi$ be a unitary representation of $G$, and suppose that $\pi$ decomposes as $\pi=\int_X^{\oplus}\pi_td\mu(t)$ over some standard Borel space $(X, \mu)$. If $\bar{H}^q(G;\pi_t)=0$ for $\mu$-almost every $t \in X$, then $\bar{H}^q(G;\pi)=0$.
\end{proposition}

The Shapiro lemma relates the cohomology with unitary coefficients of a locally compact group and a cocompact subgroup. Here and thereafter, $\Ind$ stands for unitary induction.
\begin{lemma}\cite[Theorem~8.8]{Blanc}\label{lem: shapiro lemma}
    Let $\Lambda$ be a closed cocompact subgroup of $G$ such that there exists a $G$-invariant measure on $G / \Lambda$, and $\pi$ a unitary representation of $\Lambda$. Then $H^*(\Lambda;\pi) \cong H^*(G;\mathrm{Ind}_\Lambda^G(\pi))$.
\end{lemma}
The Shapiro Lemma \ref{lem: shapiro lemma} will be extremely useful for us in the specific case where $\Lambda \leq G$ is a cocompact lattice. In this case, as we have $\mathrm{Ind}_\Lambda^G(\pi|_{\Lambda})=\pi \otimes L^2(G / \Lambda)=\pi \oplus (\pi \otimes L^2_0(G/\Lambda))$ for a unitary $G$-representation $\pi$. We therefore get an isomorphism $H^*(\Lambda;\pi|_{\Lambda}) \cong H^*(G;\pi) \oplus H^*(G;\pi \otimes L^2_0(G/\Lambda))$. 

\subsection{Semisimple groups} \label{subsec:Semisimple groups}
We set the following terminology, which is not standard.
\begin{terminology}\label{terminology:semisimple}
    A \emph{semisimple group} is a locally compact group $G$  that is a product $G=\prod G_i$, where each factor $G_i$ is one of the following:
    \begin{enumerate}
        \item A connected, simple Lie group with finite center;
        \item A group of the form $G_i=\mathbf{G}_i(k_i)$, where $k_i$ is a local field of characteristic zero and $\mathbf{G}_i$ is a connected and simply connected almost simple algebraic $k_i$-group.
    \end{enumerate} 
\end{terminology}
In this subsection we will describe the cohomological dual of semisimple groups, and explain how to understand Hausdorff cohomology in this case.

Assume that $G$ is a simple $p$-adic group. In this case, the situation is extremely simple. $\widehat{G}_{\mathrm{cohom}}$ consists of two representations: the trivial representation and a discrete series representation called the \emph{Steinberg} representation, denoted by St. See~\cite{Casselman_SteinbergRep} for the definition of the Steinberg representation. The following result was obtained by Casselman~\cite{Casselman_cohom_padic}.
\begin{theorem}\label{thm:cass_padic}
    If $G$ is a simple $p$-adic group, then $\widehat{G}_{\mathrm{cohom}}=\{\mathbf{1},\mathrm{St\}}$. Furthermore, we have:
    \begin{align*}
    H^*(G;\mathbf{1})&=H^0(G;\mathbf{1})=\mathbb{C}\\
    H^*(G;\mathrm{St})&=H^{\mathrm{rank}_k(\mathbf{G})}(G;\mathrm{St})=\mathbb{C}
    \end{align*}
\end{theorem}

In the case of Lie groups, the situation is more complicated. In a remarkable work, Vogan and Zuckerman obtained a complete description of $\widehat{G}_\mathrm{cohom}$ in terms of $\theta$-stable parabolic subalgebras of the Lie algebra $\mathfrak{g}$  \cite{VZ}. For the sake of this paper, the following facts that arise from their work will be sufficient.
\begin{theorem}\label{thm:vogan-zuckerman}
    Let $G$ be a simple Lie group. Then, $\widehat{G}_{\mathrm{cohom}}$ is a finite set. For $\pi \in \widehat{G}_{\mathrm{q-cohom}}$, the cohomology $H^q(G;\pi)$ is Hausdorff and finite dimensional.
\end{theorem}

Simple Lie groups, as well as simple $p$-adic groups, are type I. If $G$ is type I and $H$ is any other locally compact group, then any irreducible representation of $G \times H$ is an (exterior) tensor product\footnote{We reserve the notation $\otimes$ for inner tensor product, and use $\times$ for exterior tensor products of representations of  direct products of groups.} $\pi \times \rho$ of irreducible representations $\pi \in \widehat{G}$ and $\rho \in \widehat{H}$~\cite[Theorem 7.1]{Folland}. The unitary dual $\widehat{G \times H}$ therefore identifies with the product $\widehat{G} \times \widehat{H}$. One may hope the cohomological dual of $G \times H$ to have a similar description as the product of $\widehat{G}_{\mathrm{cohom}}$ and $\widehat{H}_{\mathrm{cohom}}$. In the case of products of simple groups this is the case, and we may fully describe the cohomological dual of semisimple groups using Theorems \ref{thm:vogan-zuckerman} and \ref{thm:cass_padic}.
We quote the following form of the product formula from \cite[Theorem~5.3]{BaderSauer}.
\begin{theorem}\label{thm:product formula}
    For $i \in \{1,\dots,m\}$, let $k_i$ be a local field of characteristic zero and $\mathbf{G}_i$ be a connected, almost simple algebraic $k_i$-group. Denote by $G_i$ the group of $k_i$-points $\mathbf{G}_i(k_i)$, and let $G=\prod G_i$. Then, $\widehat{G}_{\mathrm{cohom}}= \prod (\widehat{G}_i)_{\mathrm{cohom}}$. If $\pi= \times\pi_i$, then
    \[
    H^*(G;\pi) \cong \bigotimes_{i=1}^m H^*(G_i,\pi_i)
    \]
    In particular, $\widehat{G}_{\mathrm{cohom}}$ is finite, and for each $\pi \in \widehat{G}_{\mathrm{cohom}}$ the cohomology $H^*(G;\pi)$ is Hausdorff and finite dimensional. 
\end{theorem}

Now that we have a full description of $\widehat G_\mathrm{cohom}$, it is natural to ask if we can describe the cohomology of general (reducible) unitary representations of $G$. We can already at this point obtain a complete description of the Hausdorff cohomology of a general unitary representation. Let $\pi$ be an arbitrary unitary representation of $G$. Decompose $\pi$ as $\pi=\pi_\mathrm{cohom} \oplus \mathring\pi$, where $\pi_\mathrm{cohom}$ is the direct sum of all of the cohomological, irreducible subrepresentations of $G$. We have:

\begin{theorem}[Theorem \ref{intro-Hausdorff}]\label{thm: reduced cohom}
    Let $G$ be a semisimple group and let $\pi$ be a unitary representation of $G$. Then in any degree $q$, $\barH^q(G;\mathring{\pi})=0$ and $\torH^q(G;\pi_{\mathrm{cohom}})=0$, thus $\torH^q(G;\pi)=\torH^q(G;\mathring{\pi})$ and 
    \[ \bar{H}^q(G;\pi)= \bar{H}^q(G;\pi_{\mathrm{cohom}}) \cong \bigoplus_{\rho \in \widehat{G}_{\mathrm{cohom}}} \mathrm{Hom}_G(\rho,\pi) \otimes H^q(G;\rho). \]
\end{theorem}
\begin{remark}\label{rem:hilbert structure}
    In the above theorem, $\mathrm{Hom}_G(\rho,\pi)$ is endowed with the Hilbert space structure given by $\langle T,S\rangle = \langle Tv,Sv\rangle$ for an arbitrary unit vector $v$. This definition is independent of $v$ by Schur's Lemma. The associated Hilbert norm coincides with the operator norm.
\end{remark}

\begin{proof}
    We have that $H^*(G;\pi)=H^*(G;\pi_\mathrm{cohom}) \oplus H^*(G;\mathring{\pi})$. By Theorem \ref{thm:product formula}, $\pi_\mathrm{cohom}$ is the finite direct sum of finitely many isotypic representations, and for each irreducible representation the cohomology is Hausdorff and finite dimensional. Theorem \ref{thm:cohom of amplif} tells us that the cohomology of each isotypic representation is Hausdorff, and therefore $H^*(G;\pi_\mathrm{cohom})$ is Hausdorff. For a Hilbert space $V$, $\ell^2(V)$ is just the Hilbertian tensor product $V \otimes \ell^2$, we therefore obtain a formula for the cohomology $H^*(G;\pi_\mathrm{cohom})$:
    \[ H^q(G;\pi_{\mathrm{cohom}})=\bar{H}^q(G;\pi_{\mathrm{cohom}}) \cong \bigoplus_{\rho \in \widehat{G}_{\mathrm{cohom}}} \mathrm{Hom}_G(\rho,\pi) \otimes H^q(G;\rho). \]
    
    We are left to show that $\barH^q(G;\mathring{\pi})=0$. Since $G$ is type I, we may decompose $\mathring{\pi}$ as a direct integral over the unitary dual $\widehat{G}$ with respect to a unique measure class $\mu$, $\mathring{\pi}=\int_{\widehat{G}}^{\oplus} \pi_\alpha d\mu(\alpha)$. As $\mathring{\pi}$ admits no subrepresentations from $\widehat{G}_{\mathrm{cohom}}$, for $\mu$-almost every representation the Hausdorff cohomology vanishes. Proposition \ref{prop:red cohom direct integral} implies that $\bar{H}(G;\mathring{\pi})=0$.
\end{proof}

\section{Continuous homology and Poincar\'e Duality}\label{sec:4}
In this section we study the theory of continuous \emph{homology} of locally compact groups. This theory exists in the literature, but is less standard than the theory of continuous cohomology, and we review it in \S\ref{subsec:4.1}. In \S\ref{subsec:The relation between homology and cohomology} we relate both the Hausdorff homology and cohomology and the torsion in homology and cohomology. The main result of this subsection is Theorem \ref{thm:indexShift}. 
Finally, we will discuss Poincar\'e duality for Lie groups in \S\ref{subsec:4.3}, using differentiable cohomology and homology.

\subsection{Continuous homology}\label{subsec:4.1}
While less standard in the literature than cohomology, one can develop a theory of continuous homology for locally compact groups with coefficients in topological vector spaces. For a reference on continuous homology, we refer the reader to \cite{BlancHomology} and to \cite[\S C.16]{Petersen}.

One difference is that we need to broaden our category. While we are able to develop cohomology theory in the category of continuous Fr\'echet modules, for homology we need to consider topological vector spaces that are not necessarily Fr\'echet. Let $G$ be a locally compact, second countable group and denote by $\mathcal{C}_G$ the category of complete, locally convex $G$-modules. 
The definition of a \emph{strong} map of $G$-modules from \S\ref{prelim: continuous cohomology} carries over to the category $\mathcal{C}_G$, and analogous to the definition of a relatively injective $G$-module (Definition \ref{def:rel inj}), we have the notion of a relatively \emph{projective} module $E \in \mathcal{C}_G$ (see \cite[Definition~C.18]{Petersen}). Blanc showed that there are enough relatively projective objects in $\mathcal{C}_G$ \cite{BlancHomology}. For a $G$-module $E \in \mathcal{C}_G$, we denote by $E_G$ the closed co-invariants of $G$ in $E$:
$$E_G=E / \overline{\mathrm{span}(\{g.v-v \mid g \in G, v \in E \})}$$
Consider a strong, relatively projective resolution:
\[
\begin{tikzcd}[cramped]
	{\cdot \cdot \cdot} & {P_2} & {P_1} & {P_0} & E & 0
	\arrow["{\partial_3}", from=1-1, to=1-2]
	\arrow["{\partial_2}", from=1-2, to=1-3]
	\arrow["{\partial_1}", from=1-3, to=1-4]
	\arrow["\epsilon", from=1-4, to=1-5]
	\arrow[from=1-5, to=1-6]
\end{tikzcd}
\]
The continuous homology of $G$ with coefficients in $E$, denoted $H_*(G;E)$, is the homology of the complex:
\[
\begin{tikzcd}[cramped]
	{\cdot \cdot \cdot} & {(P_2)_G} & ({P_1)_G} & {(P_0)_G} & 0
	\arrow["{\partial_3}", from=1-1, to=1-2]
	\arrow["{\partial_2}", from=1-2, to=1-3]
	\arrow["{\partial_1}", from=1-3, to=1-4]
	\arrow[from=1-4, to=1-5]
\end{tikzcd}
\]
The homology $H_*(G;E)$ is a topological vector space which is not necessarily Hausdorff. It is independent of the resolution. Its torsion $\overline{\{0\}}$ is denoted by $\torH_*(G;E)$, and its maximal Hausdorff quotient is denoted by $\bar{H}_*(G;V)$

To establish the duality between continuous homology and continuous cohomology we will construct a strong, relatively projective resolution which is dual to the $L^2_{\mathrm{loc}}$-resolution discussed in \S\ref{prelim: continuous cohomology}.

For a compact subset $K \subseteq G^q$, the space $L^2(K,E)$ is the complete locally convex space which is the inverse limit of $L^2(K,E_r)$ over a separating net of semi-norms $r$. If $E$ is a Fr\'echet space, the inverse limit is over a countable system, and $L^2(K,E)$ is a Fr\'echet space. We define the space $L^2_c(G^q,E)$ as the strict inductive limit:

$$L^2_\mathrm{c}(G^q,E)= \varinjlim L^2(K,E)$$

Were the limit is taken over a countable exhaustion of $G^q$ by compact subsets $K$. This is the space of compactly supported functions from $G^q$ to $E$, endowed with a topology that makes it a complete locally convex topological vector space whenever $E$ is \cite[Theorem~II.6.6]{SchaeferWolff}. Note that this time, even if $E$ is Fr\'echet, $L^2_\mathrm{c}(G^q,E)$ is not a Fr\'echet space. Let $\partial_q:L^2_\mathrm{c}(G^{q+1},E) \rightarrow L^2_\mathrm{c}(G^q,E)$ be the map:
$$\partial_qf(g_1,\dots,g_{q})= \sum_{j=0}^{q}(-1)^j\int_G f(g_1,\dots,g_j,g,g_{j+1},\dots,g_{q})d\mu(g), $$
Let $\epsilon: L^2_c(G,E) \rightarrow E$ be the integration map $\epsilon(f)= \int_G fd\mu(g)$. Then, the resolution
\[
\begin{tikzcd}[cramped]
	{\cdot \cdot \cdot} & {L^2_c(G^3,E)} & {L^2_c(G^2,E)} & {L^2_c(G,E)} & E & 0
	\arrow["{\partial_3}", from=1-1, to=1-2]
	\arrow["{\partial_2}", from=1-2, to=1-3]
	\arrow["{\partial_1}", from=1-3, to=1-4]
	\arrow["\epsilon", from=1-4, to=1-5]
	\arrow[from=1-5, to=1-6]
\end{tikzcd}
\]
 is a strong, relatively projective resolution \cite[Theorem~C.21]{Petersen}.
\begin{remark}\label{rem:inhomogeneousHomology}
    The space $(L^2_c(G^{q+1},E))_G$ can be identified with the space $L^2_c(G^q,E)$, via the map $T_q$:
    $$T_q(\tilde{f})(g_1,\dots,g_q)=\int_G \pi(g^{-1})f(g,gg_1,gg_1g_2,\dots,gg_1 \cdots g_q)d\mu(g)$$
    Where $f$ is a function representing $\tilde{f}$. Using this identification, we get the complex of \emph{inhomogeneous} chains, similar to the  complex of inhomogeneous cochains for cohomology (Remark \ref{rem:inhomogeneous}). This complex consists of boundary maps $\partial_q':L^2_c(G^{q},E) \rightarrow L^2_c(G^{q-1},E)$ given by:
    \begin{align*}
    \partial_q'(f)(g_1,\dots,g_{q-1})=\int_G[\pi(g^{-1})f(g,g_1,\dots,g_{q-1})+\sum_{j=1}^{q-1}(-1)^j f(g_1,\dots,g,g^{-1}g_j,g_{j+1},\dots,g_{q-1}) \\
    +(-1)^{q}f(g_1,\dots,g_{q-1},g)]d\mu(g) 
    \end{align*}
    and calculates $H_*(G;E)$\cite[Theorem~C.31]{Petersen}. 
\end{remark}
\subsection{The relation between homology and cohomology} \label{subsec:The relation between homology and cohomology}
Restricting our attention to $G$-modules that are reflexive Banach spaces, we are able to relate the continuous homology and cohomology.

\begin{theorem}\label{thm:indexShift}
    Let $G$ be a locally compact, second countable group and $V$ be a continuous representation of $G$ on a reflexive Banach space. We denote the strong dual of $V$ by $V^*$, on which we have the contragredient representation. Then,
    \begin{enumerate}
        \item The Hausdorff cohomology $\bar{H}^q(G;V^*)$ is a reflexive Fr\'echet space. It is naturally isomorphic to the strong dual of the Hausdorff homology of $V$:
        $$\bar{H}_q(G;V)^* \cong \bar{H}^q(G;V^*)$$
        \item There is an index shift in the torsion:
        $$\mathring{H}_q(G;V) \ne 0 \iff \mathring{H}^{q+1}(G;V^*) \ne 0$$
    \end{enumerate}
\end{theorem}

\begin{remark}
    This theorem does not claim that the Hausdorff homology $\bar{H}_q(G;V)$ is always reflexive, only that the dual is reflexive and topologically isomorphic to $\bar{H}^q(G;V^*)$. The proof will show that the canonical injection of $\bar{H}_q(G;V)$ in to the double dual $\bar{H}_q(G;V)^{**}$ is a bijection, however we do not know if in this generality it is necessarily continuous. It is however true that the double dual topology is stronger then the natural topology of $\bar{H}_q(G;V)$. As an example of such a pathology the reader should keep in mind the following: Let $E$ be a reflexive Banach space and denote $E$ endowed with the weak topology by $E^w$. Since bounded subsets agree with weakly bounded subsets, the strong dual of $E^w$ is topologically isomorphic to the strong dual of $E$. As a result the bidual of $E^w$ is $E$ and the canonical injection is the identity $E^w \to E$, which is a bijection but is not continuous.
    
    Note however, that this pathology can not occur if $\bar{H}_q(G;V)$ is itself a Banach space. This will be the case when $V$ is a unitary representation and $G$ is either a semisimple group or a discrete group with finiteness properties (see e.g. the proof of Theorem \ref{intro:torsion-reduced} and Remark \ref{rem:finiteness prop homology}). See also Corollary \ref{cor:homIsNotMetr}.

\end{remark}

In case $V$ is a Hilbert space, and the representation is unitary, we get the following corollary.

\begin{corollary}\label{cor:index shift unitary}
    Let $G$ be a locally compact, second countable group and $V$ a unitary representation of $G$. Then, 
    $$\mathring{H}_q(G;V) \ne 0 \iff \mathring{H}^{q+1}(G;V) \ne 0$$
    and $\bar{H}^q(G;V)$ is topologically isomorphic to the complex conjugate of  $\bar{H}_q(G;V)^*$.
\end{corollary}

\begin{proof}
     Since $V$ is unitary, the contragredient representation $V^*$ is unitarily equivalent to the complex conjugate representation $\bar{V}$. Consequently, the homology and cohomology with coefficients in $V$ and $\bar{V}$ are complex conjugates of one another as topological vector spaces. The statment now follows from Theorem  \ref{thm:indexShift}.
\end{proof}

In order to prove Theorem \ref{thm:indexShift}, we will show that the resolutions that calculate the homology and cohomology which we discussed in \S\ref{subsec:4.1} are dual to one another.
\begin{lemma}\label{lem:dual pair}
    Let $G$ and $V$ be as in Theorem \ref{thm:indexShift}. Then $L^2_{\mathrm{loc}}(G^{q},V^*)$ and $L^2_c(G^{q},V)$ is a dual pair of reflexive locally convex topological vector spaces. That is, each is reflexive and is the strong dual of the other.
\end{lemma}

\begin{proof}
    We denote by $\{K_m\}$ a filtration of $G^q$ by compact subsets. Let $E_m=L^2(K_m,V)$. Since $V$ is reflexive it has the Radon-Nikodym property and thus, $E_m^* = L^2(K_m,V^*)$. Therefore:
    $$L^2_{\mathrm{loc}}(G^q,V^*)= \varprojlim E_m^*$$
    $$L^2_c(G^q,V)= \varinjlim E_m$$
    where the projective limit is taken with respect to the restriction maps $E^*_{m+1} \rightarrow E^*_m$ and the inductive limit is taken with respect to the inclusion maps $E_m \rightarrow E_{m+1}$. This way, $L^2_{\mathrm{loc}}(G^{q},V^*)$ inherits the topology of a Fr\'echet space, whereas $L^2_c(G^q,V)$ inherits the topology of an LB-space. We would like to say that in this case, the strong dual of the inductive limit is the projective limit of the strong duals, and vice versa. This will show that $L^2_{\mathrm{loc}}(G^{q},V^*)$ and $L^2_c(G^{q},V)$ are the strong duals of one another, and furthermore that they are both reflexive. Algebraically, the statement holds naively -- we can identify the continuous dual of each of these spaces with the other space.  However, the topological statement is more involved, and follows from work of Grothendieck. A topological vector space is called \emph{barrelled} if every closed, convex, balanced and absorbing set is a neighborhood of zero. A locally convex topological vector space is called \emph{distinguished} if its strong dual is barrelled. Every Fr\'echet space can be represented (non-uniquely) as a projective limit of a sequence of Banach spaces. Grothendieck showed that a Fr\'echet space $F$ is distinguished precisely when for every such representation $F=\varprojlim F_m$, its strong dual is the inductive limit $F^*=\varinjlim F_m^*$ \cite{Distinguished}. Reflexive spaces are always barrelled, and hence are always distinguished, and we observe that our space $L^2_{\mathrm{loc}}(G^q,V^*)$ is indeed reflexive as a projective limit of reflexive spaces.
\end{proof}
In order to pass to the complexes that calculate the homology and cohomology, we need to take $G$-invariants in $L^2_{\mathrm{loc}}(G^{q},V^*)$ and $G$-co-invariants in $L^2_c(G^{q},V)$. We will use the identification of the spaces of invariants and co-invariants with the inhomogeneous chains and cochains (Remarks \ref{rem:inhomogeneous} and \ref{rem:inhomogeneousHomology}). We leave it to the reader to check that the duality between the spaces of inhomogeneous chains and cochains extends to a full duality of the complexes, namely that the boundary maps $d_q'$ and $\partial_q'$ are dual to one another as well.

The proof of item $(1)$ of Theorem \ref{thm:indexShift} is rather easy to verify algebraically but requires subtle arguments from the theory of locally convex spaces in order to prove topologically. In particular, we will employ Pt\'ak's theory, which extends the open mapping theorem beyond the category of Fr\'echet spaces. A locally convex topological vector space $X$ is called \emph{Pt\'ak} if a subspace of its continuous dual is weak-$^*$-closed if and only if its intersection with any equicontinuous subset $A \subseteq X^*$ is weak-$^*$-closed in $A$. We will not need the definition, but only the following facts:
\begin{enumerate}
    \item Fr\'echet spaces are Pt\'ak \cite[Example~IV.8.1]{SchaeferWolff}.
    \item Strong duals of reflexive Fr\'echet spaces are Pt\'ak \cite[Example~IV.8.2]{SchaeferWolff}.
    \item Closed subspaces and Hausdorff quotients of Pt\'ak spaces are Pt\'ak \cite[Theorem~IV.8.2]{SchaeferWolff} and \cite[Corollary~IV.8.3]{SchaeferWolff}.
\end{enumerate}
The importance of Pt\'ak spaces comes from the following result \cite[Corollary~IV.8.1]{SchaeferWolff}.
\begin{theorem}\label{ptak}
    Every continuous map from a Pt\'ak space onto a barrelled space is open.
\end{theorem}
\begin{proof}[Proof of item $(1)$ in Theorem \ref{thm:indexShift}]
We denote the inhomogeneous $q$-chains, $q$-cycles and $q$-boundaries by $C_q$, $Z_q$ and $B_q$, respectively. Similarly, we denote the $q$-cochains, $q$-cocycles and $q$-coboundaries by $C^q$, $Z^q$ and $B^q$, respectively. Recall that the complexes $(C_\bullet,\partial)$ and $(C^\bullet,d)$ are topologically dual. For a closed subspace of $F \leq C_q$, we denote by $F^\perp$ its annihilator in the strong dual $C^q=(C_q)^*$. We have the following relations between the cycles, boundaries, cocycles and coboundaries: 
$$\overline{B_q}^\perp=Z^q, \quad Z_q^\perp=\overline{B^q}$$

The first equality is general for operators between locally convex spaces, the second equality uses the reflexivity of $C_q$. Using this it is easy to construct an algebraic isomorphism $\bar{H}_q(G;V)^* \cong \bar{H}^q(G;V^*)$. The subtlety in the following proof comes from the need to verify that this isomorphism is topological.

\paragraph{\textbf{Step I}} We start by showing that there is a continuous bijection
$$\overline{B_q}^\perp / Z_q^\perp \rightarrow (Z_q / \overline{B_q})^*.$$
Note that the left-hand side is just $\bar{H}^q(G;V^*)$ and the right-hand side is the strong dual of $\bar{H}_q(G;V^*)$. First, we identify the dual of the space $C_q/\overline{B_q}$. Algebraically, the dual of the quotient by a closed subspace is precisely the annihilator $\overline{B_q}^\perp$. in order to see that this isomorphism is topological, we note that the transpose of the quotient map $C_q \to C_q/ \overline{B_q}$ is a continuous injection $(C_q/\overline{B_q})^* \hookrightarrow C_q^*$. As the strong dual of a Fr\'echet space, $C_q$ is a DF-space. Hausdorff quotients of DF-spaces are again DF-spaces \cite[Proposition~8.3.16]{BonerPerezcarreras}, and duals of DF-spaces are Fr\'echet. Therefore $(C_q/\overline{B_q})^*$ is a Fr\'echet space, and the above injection gives a continuous bijection of Fr\'echet spaces $(C_q/\overline{B_q})^* \to\overline{B_q}^\perp$. It is therefore a topological isomorphism by the open mapping theorem. The (continuous) injection of $Z_q / \overline{B_q}$ as a subspace of $C_q / \overline{B_q}$ induces a continuous map of the strong duals which is surjective by Hahn-Banach:
$$(C_q / \overline{B_q})^*=\overline{B_q}^\perp \to (Z_q / \overline{B_q})^*.$$
Algebraically, the kernel of this map is precisely $Z_q^\perp$, and we get the continuous bijection
$$\overline{B_q}^\perp / Z_q^\perp \rightarrow (Z_q / \overline{B_q})^*.$$
\paragraph{\textbf{Step II}} We will show that the map from the previous step is open. The space $\overline{B_q}^\perp / Z_q^\perp$ is Fr\'echet, and hence Pt\'ak. In order to use Theorem \ref{ptak} we need to show that $(Z_q / \overline{B_q})^*$ is barrelled. First, consider the space $C_q / \overline{B_q}$. Since $(Z^q)^{\perp}=\overline{B_q}$ we get a continuous map $C_q / \overline{B_q} \to (Z^q)^*$, which is surjective by Hahn-Banach. $C_q$ is a Pt\'ak space as the strong dual of a reflexive Fr\'echet space. The space $(Z^q)^*$ is reflexive and hence barrelled \cite[Theorem~20.7]{KelleyNamioka}. Therefore the map $C_q / \overline{B_q} \to (Z^q)^*$ is open. Thus, $C_q / \overline{B_q}$ is topologically isomorphic to the strong dual of a reflexive space, and it is hence reflexive as well. We pass to the subspace $Z_q/ \overline{B_q}$. Recall that a locally convex space is said to be \emph{semi-reflexive} if its canonical map into its strong bi-dual is a bijection. Subspaces of locally convex spaces do not inherit reflexivity, but they do inherit semi-reflexivity \cite[Theorem~20.2]{KelleyNamioka}. Therefore $Z_q/ \overline{B_q}$ is semi-reflexive. The strong dual of semi-reflexive spaces is barrelled \cite[Theorem~IV.5.5]{SchaeferWolff}. Hence, the map from Step I is a topological isomorphism by Theorem \ref{ptak}.

We are left to show that the Hausdorff cohomology is a reflexive Fr\'echet space. In general, reflexivity does not pass to arbitrary subquotients for Fr\'echet spaces. A Fr\'echet space is said to be \emph{totally reflexive} if all its Hausdorff quotients are reflexive. Totally reflexive Fr\'echet spaces were characterized by Valdivia as those that are subspaces of a product of reflexive Banach spaces \cite[Theorem~3]{valdivia}. The space $L^2_{\mathrm{loc}}(G^n,V^*)$ of cochains is such. Hence the Hausdorff cohomology is reflexive.
\end{proof}

\begin{proof}[Proof of item $(2)$ in Theorem \ref{thm:indexShift}]
  $H^{q+1}(G;V^*)$ being Hausdorff is equivalent to the image of $d$ being closed, and $H_q(G;V)$ being Hausdorff is equivalent to the image of $\partial = d^*\colon C_{q+1} \rightarrow C_q$ being closed. Write $d = i \circ \tilde{d}$ as the composition of $\tilde{d}\colon C^q \rightarrow \overline{B^{q+1}}$ and $i\colon \overline{B^{q+1}} \rightarrow C^{q+1}$. This way, $d^* =  \tilde{d}^* \circ i^*$. Banach's theorem on Fr\'echet spaces states that if $F_1$ and $F_2$ are Fr\'echet spaces, then $T\colon F_1 \rightarrow F_2$ is surjective if and only if $T^*\colon F_2^* \rightarrow F_1^*$ is injective and its image is closed in the weak-$^*$ topology on $F_1^*$.

For the first direction, assume that $d$ has a closed image. Then $\tilde{d}$ is surjective. By Banach's theorem $\tilde{d}^*$ has a closed image in the weak-$^*$ topology, and so also in the strong topology.

For the other direction, assume that $d^*$ has a closed image. Since $i^*$ is surjective, the image of $\tilde{d}^*$ is strongly closed. We claim that it is also closed in the weak-$^*$ topology. The image is a convex set, so it is closed in the weak topology by Hahn-Banach. Reflexivity implies that the weak topology coincides with the weak-$^*$ topology. The map $\tilde{d}^*$ is injective since $\tilde{d}$ has a dense image. So $\tilde{d}^*$ has a weak-$^*$ closed image and  is injective. By Banach's theorem, $\tilde{d}$ is surjective, and $d$ has a closed image.
\end{proof}

We can now deduce the following:

\begin{corollary}\label{cor:homIsNotMetr}
    Let $G$ be a locally compact, second countable group and $V$ be a continuous representation of $G$ on a reflexive Banach space. The space $\bar{H}_q(G;V)$ is a complete locally convex space. It is metrizable if and only if it is isomorphic to a Banach space, and in this case it is reflexive.
\end{corollary}
\begin{proof}
    The proof of Theorem \ref{thm:indexShift} realizes $\bar{H}_q(G;V)$ as a Hausdorff subquotient of a Pt\'ak space. Therefore, it is a Pt\'ak space and in particular complete. Assume that it is metrizable, namely, that it is a Fr\'echet space. The proof of Theorem \ref{thm:indexShift} also shows that it is semi-reflexive, and hence it is a reflexive Fr\'echet space. By Theorem \ref{thm:indexShift}, its strong dual is also a Fr\'echet space. As mentioned in the proof of Lemma \ref{lem:dual pair}, the strong dual of a reflexive Fr\'echet space is an inductive limit of Banach spaces, that is, an LB-space. But proper LB-spaces are never metrizable \cite[Proposition~8.5.18]{BonerPerezcarreras}, and hence the dual $\bar{H}^q(G;V^*)$ is a Banach space. Reflexivity now yields that $\bar{H}_q(G;V)$ is a Banach space as well. It was already established to be reflexive under these conditions. 
\end{proof}

\subsection{Differentiable homology and cohomology and Poincar\'e duality}\label{subsec:4.3}

Until the end of this section, we assume that $G$ is a connected, unimodular Lie group.
The main result of this section is the Poincar\'e duality theorem \ref{thm:Poincare duality}.
As a theorem about \emph{differentiable} cohomology and homology, Poincar\'e duality is a result of Blanc and Wigner \cite{BlancWigner}. To relate the differentiable homology and cohomology to the continuous homology and cohomology, we use results of Pichaud \cite{PichaudRegularization} and Blanc \cite{Blanc} respectively. 

 We briefly recall the theory of \emph{differentiable} or \emph{smooth} cohomology of $G$. For further information, the reader is referred to \cite{HochschildMostow}, \cite[Parties I,VI]{PichaudRegularization}, and \cite{BlancWigner}. 

Let $E$ be a complete, locally convex $G$-module. A vector $v \in E$ is called \emph{smooth} if the map $g \mapsto g.v$ is a smooth map on $G$. The subspace of smooth vectors in $E$ is denoted by $E^\infty$. A  G-module $E$ is called \emph{smooth} if $E=E^\infty$ and the map $v \mapsto \{g \mapsto g.v\}$ is a topological isomorphism of $E$ onto its image in $C^\infty(G,E)$, endowed with the $C^\infty$-topology \cite[Subsection~0.2.3]{BorelWallach}. If $E$ is a continuous representation of $G$ on a complete, locally convex topological vector space, then $E^\infty$ inherits a complete, locally convex topology from $C^\infty(G,E)$ that makes it into a differentiable $G$-module~\cite[Proposition~4.6]{Neeb}. We denote the category of complete, locally convex differentiable $G$-modules by~$\mathcal{D}_G$. This is a subcategory of $\mathcal{C}_G$, and $E \mapsto E^\infty$ is a functor $\mathcal{C}_G \to \mathcal{D}_G$.

An object in $\mathcal{D}_G$ is called relatively injective (resp. projective) if it satisfies Definition \ref{def:rel inj} (resp.~\cite[Definition~C.18]{Petersen}) within the category $\mathcal{D}_G$. The category $\mathcal{D}_G$ has enough relatively injective objects: there is a strong, relatively injective resolution for every object $E \in \mathcal{D}_G$ (\cite[Section 5]{Blanc}, \cite[Chapter~XI.5]{BorelWallach}). This allows to define the differentiable cohomology of $G$ with coefficient in a differentiable module $G$, denoted $H^*_d(G;E))$. A-priori, an element $I \in \mathcal{D}_G$ might be relatively injective in $\mathcal{D}_G$, but not relatively injective in $\mathcal{C}_G$, perhaps resulting in $H^*(G;E) \ne H_d^*(G;E)$. This is not the case. The following follows from \cite[Theorem~5.2]{Blanc} and \cite[Lemma IX.5.2]{BorelWallach}.
\begin{lemma}\label{lem:smooth and cont cohom}
    Let $E$ be a Fr\'echet $G$-module, then $H^*(G;E)\cong H^*(G;E^\infty) \cong H^*_d(G;E^\infty)$.
\end{lemma}

The category $\mathcal{D}_G$ also has enough relatively projective objects \cite[Proposition~VI.4]{PichaudRegularization}, allowing one to develop a differentiable homology theory. The differentiable homology is denoted by $H_{d,*}(G;E)$ for $E \in \mathcal{D}_G$. We have an analogue of Lemma \ref{lem:smooth and cont cohom} following from~\cite[Propositions~VI.3, VI.4, VI.5]{PichaudRegularization}.
\begin{lemma}\label{lem:smooth and cont hom}
    Let $E \in \mathcal{C}_G$. Then $H_*(G;E)\cong H_*(G;E^\infty)\cong H_{d,*}(G;E^\infty)$.
\end{lemma}

We are now able to formulate and prove the Poincar\'e duality theorem.
\begin{theorem}[Poincaré duality]\label{thm:Poincare duality}
    Let $G$ be a connected, unimodular Lie group and $V$ be a reflexive Banach $G$-module. Denote by $n=\mathrm{dim}(G / K)$, where $K \subseteq G$ is a maximal compact subgroup. Then $H^q(G;V)$ is topologically isomorphic to $H_{n-q}(G;V)$. In particular, $H^q(G;V)$ is non-Hausdorff if and only if $H^{n-q+1}(G;V^*)$ is non-Hausdorff.
\end{theorem}

\begin{proof}
    By Lemmas \ref{lem:smooth and cont cohom} and \ref{lem:smooth and cont hom}, $H^q(G;V)\cong H^q_d(G;V^\infty)$ and $H_{n-q}(G;V)\cong H_{d,n-q}(G;V^\infty)$. Poincar\'e duality for differentiable cohomology $H^q_d(G;V^\infty)\cong H_{d,n-q}(G;V^\infty)$ is  a result of Blanc and Wigner \cite[\S6]{BlancWigner}. The ``in particular'' part follows from Corollary \ref{cor:index shift unitary}.
\end{proof}
We also get the following general corollary.
\begin{corollary}\label{cor:cohom is banach}
     Let $G$ be a connected, unimodular Lie group and $V$ be a reflexive Banach $G$-module. Then the Hausdorff homology $\bar{H}_q(G;V)$ and Hausdorff cohomology $\bar{H}^q(G;V)$ are both reflexive Banach spaces.
\end{corollary}
\begin{proof}
    Theorem \ref{thm:Poincare duality} implies that $\bar{H}_q(G;V)$ is isomorphic to $\bar{H}^{n-q}(G;V)$ which is a Fr\'echet space. By Corollary \ref{cor:homIsNotMetr} $\bar{H}_q(G;V)$ is a Banach space. The statement for $\bar{H}^q(G;V)$ now follows by Theorem \ref{thm:indexShift}.
\end{proof}
\begin{remark}
    Theorem \ref{thm:Poincare duality} and Corollary \ref{cor:cohom is banach} were formulated for unimodular groups. The result of Blanc-Wigner on Poincar\'e duality for differentiable cohomology \cite[\S6]{BlancWigner}) has a version for non-unimodular Lie groups as well. However, the formulation is more involved and we are interested mostly in the semisimple case. 
\end{remark}
We now state and prove Theorem \ref{intro:homology-cohomology}.
\begin{theorem}[Theorem \ref{intro:homology-cohomology}]\label{thm:homology and cohomology}
    Let $G$ be a semisimple group, and let $\pi$ be a unitary representation of $G$.
    Then $\bar{H}^q(G;\pi)$ is naturally isomorphic to the complex conjugate of $\bar{H}_q(G;\pi)^*$ and  
    \[
        \torH_q(G;\pi)=0 \iff \torH^{q+1}(G;\pi)=0.
    \]
    Moreover, if $G$ is a Lie group and $n$ is the dimension of the corresponding symmetric space, then we have natural topological vector space isomorphisms
    $H_q(G;\pi)\cong H^{n-q}(G;\pi)$, commuting with the decomposition \eqref{eq:tor-bar}.
\end{theorem}
\begin{proof}
    The first part of the theorem follows from Corollary \ref{cor:index shift unitary}, and the second from Theorem \ref{thm:Poincare duality}. The fact that the Poincar\'e isomorphism commutes with the decomposition \eqref{eq:tor-bar} comes from the fact that this decomposition arises from a decomposition of the representation $\pi$ itself, $\pi=\pi_\mathrm{cohom} \oplus \mathring\pi$, as in Theorem \ref{thm: reduced cohom}. The Poincar\'e isomorphism holds for each of the components separately.
\end{proof}
\section{groups with finiteness properties}\label{sec3}
 In this section we make some contributions to the theory of continuous cohomology with coefficients in a unitary representation for countable groups that admit some finiteness properties. As we will see, for these groups we have a convenient cochain complex of Hilbert spaces that calculates the continuous cohomology. We first describe the general theory of cohomology with Fr\'echet coefficients, in order to fix notations and bridge some gaps in the literature. We then specialize the discussion to unitary coefficients, with an emphasis on torsion cohomology. In this section, $\Gamma$ stands for an arbitrary countable group.
 
    A group~$\Gamma$ is said to be of type $\mathrm{FP}_\infty(\mathbb{Q})$ if there is a projective resolution of the trivial $\mathbb{Q}[\Gamma]$-module $\mathbb{Q}$ which is finitely generated.

\begin{example}\label{ex:lattices finite type}
    By the work of Borel and Serre, lattices in semisimple groups are of type $\mathrm{FP_\infty(\Q)}$ \cite{borel+serre}. 
\end{example}

If $\Gamma$ is of type $\mathrm{FP}_\infty(\mathbb{Q})$, there always exists a resolution of the trivial $\mathbb{Q}[\Gamma]$-module $\mathbb{Q}$ which is free and finitely generated and \emph{free} in every degree \cite[Theorem~VIII.4.5]{Brown}. Consider such a resolution:
\[\begin{tikzcd}[cramped]
	\cdots & {P_2} & {P_1} & {P_0} & {\mathbb{Q}} & 0
	\arrow[ "{\partial_3}", from=1-1, to=1-2]
	\arrow["{\partial_2}", from=1-2, to=1-3]
	\arrow["{\partial_1}", from=1-3, to=1-4]
	\arrow["{\epsilon}", from=1-4, to=1-5]
	\arrow[from=1-5, to=1-6]
\end{tikzcd}\]
with $P_q \cong \mathbb{Q}[\Gamma]^{m_q}$, considered as row vectors. In such coordinates the boundary maps $\partial_q: P_q \rightarrow P_{q-1}$ are given uniquely as  right multiplication by $m_{q}\times m_{q-1}$ matrices over $\mathbb{Q}[\Gamma]$. For a Fr\'echet $\Gamma$-module $(E,\pi)$, we consider the Fr\'echet cochain complex $(\mathrm{Hom}_{\Q}(P_q,E),d_q)$ where $d_q$ are the dual maps to $\partial_q$. We make $\mathrm{Hom}_{\Q}(P_q,E)$ a Fr\'echet $\Gamma$-module by equipping it with the topology of pointwise convergence and the $\Gamma$-action given by $[g.f](\_) = \pi(g)f(g^{-1}\_)$. Since $\mathrm{Hom}_\Q(-,E)$ is exact we obtain a $\Q[\Gamma]$-resolution of $E$:

\[\begin{tikzcd}[cramped]
	{\cdot\cdot\cdot} & {\mathrm{Hom}_\Q(P_2,E)} & {\mathrm{Hom}_\Q(P_1,E)} & {\mathrm{Hom}_\Q(P_0,E)} & E & 0
	\arrow["{d_3}"', from=1-2, to=1-1]
	\arrow["{d_2}"', from=1-3, to=1-2]
	\arrow["{d_1}"', from=1-4, to=1-3]
	\arrow["\epsilon"', from=1-5, to=1-4]
	\arrow[from=1-6, to=1-5]
\end{tikzcd}\]
We claim that this resolution is strong and relatively injective. To see the relative injectivity note that under the identification $P_q \cong \mathbb{Q}[\Gamma]^{m_q}$ we have:
$$\mathrm{Hom}_\Q(P_q,E) \cong \mathrm{Hom}_\Q(\mathbb{Q}[\Gamma]^{m_q},E) \cong \mathrm {Map}(\Gamma,E)^{m_q} \cong C(\Gamma,E)^{m_q}.$$
This is a continuous isomorphism of $\Gamma$-modules where the action on $C(\Gamma,E)$ is the same one described in Section \ref{prelim: continuous cohomology} (on the space $L^2_\mathrm{loc}$, see also Remark \ref{rem:lp loc}), which makes it into a relatively injective $\Gamma$-module.
We will show that the resolution is strong. The chain complex $P_q$ above is acyclic and thus it is contractible as a $\Q$-module complex. Let $h_q\colon P_q\to P_{q+1}$ be a chain contraction. The dual maps $h_q^*\colon \mathrm{Hom}_\Q(P_{q+1},E) \to \mathrm{Hom}_\Q(P_q,E)$ are easily checked to be continuous and are a $\C$-linear cochain contraction, proving the resolution is strong.

\begin{remark}
    In order to see that the resolution is strong also in degree $-1$ we apply the same proof, simply denoting $P_{-1} =\Q$ and $\epsilon = \partial_0$ (this part of the proof does not use the freeness of the $P_i$).
\end{remark}

Since the resolution is strong and relatively injective we may take the invariants of the resolution in order to calculate the continuous cohomology $H^q(\Gamma; E)$. With $\mathrm{Hom}_\Q(P_q,E)^\Gamma = \mathrm{Hom}_{\Q[\Gamma]}(P_q,E)$ we obtain the cochain complex:

\[\begin{tikzcd}[cramped]
	{\cdot\cdot\cdot} & {\mathrm{Hom}_{\Q[\Gamma]}(P_2,E)} & {\mathrm{Hom}_{\Q[\Gamma]}(P_1,E)} & {\mathrm{Hom}_{\Q[\Gamma]}(P_0,E)} & 0
	\arrow["{d_3}"', from=1-2, to=1-1]
	\arrow["{d_2}"', from=1-3, to=1-2]
	\arrow["{d_1}"', from=1-4, to=1-3]
	\arrow[from=1-5, to=1-4]
\end{tikzcd}\]
Choosing a basis for $P_q$ as above we obtain a natural isomorphism $\mathrm{Hom}_{\Q[\Gamma]}(P_q,E) \cong E^{m_q}$, now considered as column vectors. This yields the following complex calculating the continuous cohomology:

\[\begin{tikzcd}[cramped]
	{\cdot\cdot\cdot} & {E^{m_2}} & {E^{m_1}} & {E^{m_0}} & 0 & {}
	\arrow["{d_3}"', from=1-2, to=1-1]
	\arrow["{d_2}"', from=1-3, to=1-2]
	\arrow["{d_1}"', from=1-4, to=1-3]
	\arrow[from=1-5, to=1-4]
\end{tikzcd}\]
The coboundary maps in this complex are given by the same matrices representing $\partial_q$, now acting via the given representation $\pi$ on column vectors.

For a group $\Gamma$ of type $\mathrm{FP}_\infty(\Q)$ together with a free resolution, the matrices over $\Q[\Gamma]$ defining $d_q$ are independent of the module $(E,\pi)$, and are canonically defined  once we fix a basis for the $P_q$'s. By abuse of notation we will identify $d_q$ with the element defining it in $M_{m_q\times m_{q-1}}(\Q[\Gamma])$. Formally one needs to apply $\pi$ to the entries of these matrices first.

While the cochain complex $E^{m_q}$ is less canonical then the bar complex, it is much simpler and easier to handle. For example, if $(E,\pi)$ is a unitary representation, then the spaces $E^{m_q}$ are Hilbert spaces and so is the Hausdorff cohomology $\bar{H}^q(G;E)$; a fact that is not at all obvious when looking at the bar complex.

\begin{remark}\label{rem:finiteness prop homology}
    One can adapt this construction to calculate homology instead of cohomology. We sketch the construction and leave the details to the reader. Assume that $E$ is a reflexive Banach $G$-module. Given a projective resolution $P_q$ as above, consider the complex $P_q \otimes_\Q E$. We endow $P_q \otimes_\Q E$ with the finest locally convex topology such that for all finite dimensional subspaces $L \subseteq P_q$, the maps $L\otimes_\Q E \to P_q \otimes_\Q E$ are continuous. As above, $(P_q\otimes_\Q E, \partial_q \otimes \mathrm{Id})$ is a strong relatively projective resolution of $E$. The complex of co-invariants is topologically isomorphic to $(E^{m_q},\tilde{\partial}_q)$ under the identification $(P_q\otimes_\Q E)_\Gamma \cong (\Q[\Gamma]^{m_q}\otimes_\Q E)_\Gamma \cong \Q[\Gamma]^{m_q}\otimes_{\Q[\Gamma]} E \cong E^{m_q}$. Here  $\tilde{\partial}_q$ denotes the matrix given by taking the elements of the matrix representing $\partial_q$ and replacing all elements of $\Gamma$ by their inverses. The reason for this is that the in the identification $(\Q[\Gamma]^{m_q}\otimes_\Q E)_\Gamma \cong \Q[\Gamma]^{m_q}\otimes_{\Q[\Gamma]} E$ we are considering $\Q[\Gamma]^{m_q}$ as a right $\Q[\Gamma]$ module, where the $\Gamma$-action is given by left   multiplication by the inverse. $E^{m_q}$ are considered as row vectors and $\tilde{\partial}_q$ are acting by right matrix multiplication. This isomorphism is topological because $P_q\otimes_\Q E$ is a Pt\'ak space and $E^{m_q}$ is Banach, thus barrelled. 
\end{remark}

\subsection{Unitary Cohomology}
The rest of the section will study cohomology of unitary representations. For simplicity we fix a countable group $\Gamma$ with property $\mathrm{FP}_\infty(\mathbb{Q})$, although much of what we do holds up to a given degree for groups with weaker finiteness properties. We will denote unitary $\Gamma$-representations by the letters $V$ and $W$, as opposed to $E$ and $F$ which we use for general Frech\'et modules.

We fix a finitely generated free resolution $P_q$ for $\Gamma$ and retain the notation of the previous subsection.  Denote by $C^*(\Gamma)$ the \emph{universal $C^*$ algebra of $\Gamma$}. The boundary maps $d_q$, will be considered as elements from $M_{m_{q} \times m_{q-1}}(\C^*(\Gamma))$.  We define three elements of $M_{m_q}(C^*(\Gamma))$, called respectively the \emph{Lapian}, \emph{Lacian} and \emph{Laplacian} in degree $q$:
\begin{align*}
\Delta_{q}^- &= d_{q}d_{q}^*\\
\Delta_{q}^+ &= d_{q+1}^*d_{q+1}\\
\Delta_{q} &= \Delta_{q}^+ + \Delta_{q}^-
\end{align*}
These three operators are positive and self-adjoint.

Given a unitary $\Gamma$-representation $(V,\pi)$, and an element $x \in M_d(C^*(\Gamma))$ we will denote the operator $\pi(x) \in M_d(\mathcal{B}(V))$ by $x_V$. For example, the Laplacian in degree $q$ acting on $V^{m_q}$ will be denoted $\Delta_{q,V}$.
The cohomology of a unitary representation $V$ is calculated from the following co-chain complex:

\[\begin{tikzcd}[cramped]
	{\cdot \cdot \cdot} & {V^{m_2}} & {V^{m_1}} & {V^{m_0}} & 0
	\arrow["{d_{3,V}}"', from=1-2, to=1-1]
	\arrow["{d_{2,V}}"', from=1-3, to=1-2]
	\arrow["{d_{1,V}}"', from=1-4, to=1-3]
	\arrow[from=1-5, to=1-4]
\end{tikzcd}\]
We recall the following key fact and include a proof for convenience. See also~\cite[Lemma~1.18 on p.~25]{LuckBook}.

\begin{lemma}[Hodge-de Rahm decomposition] \label{lemma:Hodge-de Rahm}
We have $\ker(\Delta_{q,V}) = \mathrm{Im}(d_{q,V})^{\perp} \cap \ker(d_{q+1,V})$. Thus $\ker(\Delta_{q,V})$ is canonically isomorphic to $\bar{H}^q(\Gamma;V)$.
\end{lemma}
\begin{proof}
    Both $\Delta_{q,V}^-$ and $\Delta_{q,V}^+$ are positive operators, therefore $\ker(\Delta_{q,V}) = \ker(\Delta_{q,V}^-) \cap \ker(\Delta_{q,V}^+)$. Now $\ker(\Delta_{q,V}^+) = \ker(d_{q+1,V})$ and  $\ker(\Delta_{q,V}^-) = \ker(d_{q,V}^*) = \mathrm{Im}(d_{q,V})^{\perp}$. Therefore $\ker(\Delta_{q,V}) = \mathrm{Im}(d_{q,V})^{\perp} \cap \ker(d_{q+1,V})$ as needed.
\end{proof}
Lemma \ref{lemma:Hodge-de Rahm} relates the Hausdorff cohomology to the kernel of the Laplacian $\Delta_{q,V}$. In order to study torsion cohomology, we will need to consider `almost kernels' of $\Delta_{q,V}^+$ and $\Delta_{q,V}^-$.
\begin{definition}
    Let $U_1$ and $U_2$ be two Hilbert spaces, and $T\colon U_1 \rightarrow U_2$ a continuous linear map. An \emph{almost kernel} of $T$ is a sequence of unit vectors $v_i$ which satisfy $v_i \perp \ker(T)$ for every $i$, and $T(v_i) \rightarrow 0$.
\end{definition}
\begin{lemma}\label{lemma:non-Hausdorff characterization}
    The following are equivalent:
    \begin{enumerate}
        \item The cohomology $H^q(\Gamma;V)$ has non-zero torsion.
        \item The operator $\Delta_{q,V}^-$ on $V^{m_q}$ has an almost kernel.
        \item The operator $\Delta_{q-1,V}^+$ on $V^{m_{q-1}}$ has an almost kernel.
    \end{enumerate}
\end{lemma}
\begin{proof}
    The open mapping theorem states that a map $T\colon U_1 \rightarrow U_2$ between two Hilbert spaces has a non-closed image if and only if it has an almost kernel. Moreover, the image of $T$ is closed if and only if the image of $T^*\colon U_2 \rightarrow U_1$ is closed. Therefore, $H^q(\Gamma;V)$ being non-Hausdorff is equivalent to $d_{q,V}$ having an almost-kernel, and also to $d_{q,V}^*$ having an almost kernel. These are equivalent to $\Delta_{q-1,V}^+$ and $\Delta_{q,V}^-$ having an almost kernel, respectively. 
\end{proof}
Our strategy is to study torsion cohomology via weak containment of representations. Broadly speaking, weak containment  $\rho \prec \pi$ allows one to approximate a matrix coefficients of $\rho$ by sums of matrix coefficients of $\pi$. Identifying an element $v \in V$ with the corresponding matrix coefficient $g \mapsto \langle gv,v \rangle$, and considering elements of the cochain space $V^{m_q}$ as $m_q$-tuples of matrix coefficients, we translate the approximation of matrix coefficients to a relation between kernels and almost kernels of Laplacians. One technical difficulty in doing so is that we need to consider an element $\bar{v} \in V^{m_q}$ not just as a tuple of matrix coefficients of $\Gamma$, but as a state on the $C^*$-algebra $M_{m_q}(C^*(\Gamma))$. We will make a small digression to consider this adjustment. 

Given a unitary $\Gamma$-representation $V$, the space $V^{m_q}$ is naturally an $M_{m_q}(C^*(\Gamma))$-module. Every unit vector $\bar{v} \in V^{m_q}$ corresponds to a vector state $\varphi_{\bar{v}}$ on $M_{m_q}(C^*(\Gamma))$, by
$$\varphi_{\bar{v}}(x)=\langle x\bar{v},\bar{v}\rangle$$ 
Since $\Gamma$ is discrete $M_{m_q}(C^*(\Gamma))$ is unital, and hence the state space of $M_{m_q}(C^*(\Gamma))$ is compact. Similarly to unitary representations of groups, we say that an $M_{m_q}(C^*(\Gamma))$-representation $U_1$ is \emph{weakly contained} in an $M_{m_q}(C^*(\Gamma))$-representation $U_2$ if we can approximate every state coming from $U_1$ by convex combinations of states coming from $U_2$. A state $\varphi \in \mathcal S (M_{m_q}(C^*(\Gamma)))$ is extremal if and only if its GNS representation is an irreducible $M_{m_q}(C^*(\Gamma))$-representation. We need the following elementary lemma that relates representations of $\Gamma$ (or equivalently, of $C^*(\Gamma)$) to representations of $M_{m_q}(C^*(\Gamma))$.
\begin{lemma}\label{lemma:reps of amplification}
    Every unital $*$-representation of $M_{m_q}(C^*(\Gamma))$ is an amplification $U^{m_q}$ for some $\Gamma$-representation $U$. The $M_{m_q}(C^*(\Gamma))$-representation $U^{m_q}$ is irreducible if and only if $U$ is an irreducible $\Gamma$-representation. Given two $\Gamma$-representations $U_1$ and $U_2$, we have that $U_1 \prec_\Gamma U_2$ if and only if $U_1^{m_q} \prec_{M_{m_q}(C^*(\Gamma))} U_2^{m_q}$. 
\end{lemma}

\begin{proof}
    The first claim is proved in \cite[Proposition~12]{BaderNowak}. The claim on irreducibility follows from the von-Neumann bicommutant theorem and the observation that if $\pi$ is a $*$-representation of $C^*({\Gamma})$ on a Hilbert space $U$, then $M_{m_q}(\pi(C^*(\Gamma)))' = \pi(C^*(\Gamma))'\otimes I_{m_q}$. As weak containment is determined by the kernels of the representations, the claim on weak containment follows from the fact that the kernel of the amplification of a representation $\pi$ is the amplification of the kernel of $\pi$.
\end{proof}

The first main result of this chapter is:

\begin{theorem}[Theorem \ref{thm:main_1_finiteness_intro}]\label{thm:main thm finit prop 1}
    Let $\Gamma$ be a group of type $\mathrm{FP}_\infty(\Q)$, and let $\pi$ be a unitary representation of $\Gamma$. For every degree $q\in\N$ we have
    \[ \torH^q(\Gamma;\pi)\neq 0~~ \Longrightarrow~~ \mathrm{supp}(\pi) \cap \widehat{\Gamma}_{\mathrm{q-cohom}} \neq \emptyset ~~\text{and}~~ \mathrm{supp}(\pi) \cap \widehat{\Gamma}_{\mathrm{(q-1)-cohom}} \neq \emptyset.\]
\end{theorem}
\begin{proof}
    We denote the Hilbert space of $\pi$ by $V$ and assume that $\torH^q(\Gamma;\pi) \ne 0$. We begin by showing that $\mathrm{supp}(\pi) \cap \widehat{\Gamma}_{\mathrm{(q-1)-cohom}} \neq \emptyset$. 
    
    Since $\pi$ has non-zero torsion cohomology in degree $q$, part $(3)$ of Lemma \ref{lemma:non-Hausdorff characterization} tells us that $d_{q,V}$, or equivalently $\Delta_{q-1,V}^+$, has an almost kernel $\bar{v}_i \in V^{m_{q-1}}$. By definition, $\bar{v}_i \in \ker(d_{q,V})^{\perp}$. Since $\mathrm{Im}(d_{q-1,V}) \subseteq \ker(d_{q,V})$, we obtain that $\bar{v}_i \in \mathrm{Im}(d_{q-1,V})^\perp=\ker(d_{q-1,V}^*)=\ker(\Delta_{q-1,V}^-)$.

    Let $\varphi_i=\varphi_{\bar{v}_i}$ be the states corresponding to $\bar{v}_i$ on the $C^*$-algebra $M_{m_{q-1}}(C^*(\Gamma))$. Since the state space $\mathcal{S}(M_{m_{q-1}}(C^*(\Gamma)))$ is compact, we may assume, upon passing to a subsequence, that $\varphi_i$ converges to a state $\varphi$. By Lemma \ref{lemma:reps of amplification}, the GNS representation of $\varphi$ is an amplification $U^{m_{q-1}}$ of some unitary $\Gamma$-representation $\sigma \prec_\Gamma \pi$ on a space $U$. If $\bar{u} \in U^{m_{q-1}}$ is the cyclic vector representing $\varphi$ as a vector state, then $\varphi_{\bar{u}}=\varphi$.
    
    We claim that $\bar{u}$ represents a non-trivial cohomology class in $\bar{H}^{q-1}(\Gamma;\sigma)$. Since \[\varphi_i(\Delta_{q-1}^+) = \|d_{q,V}( \bar{v}_i)\|^2 \xrightarrow[i \rightarrow \infty]{}0,\] 
    we conclude that $\varphi(\Delta_{q-1}^+) = 0$, that is $\|\Delta_{q-1,U}^+(\bar{u})\| = 0$. Similarly, since \[\varphi_i(\Delta_{q-1}^-) = \|d_{q-1,V}^*( \bar{v}_i)\|^2 =0,\] we obtain that $\Delta_{q-1,U}^-( \bar{u} )= 0$. Therefore we have that $0 \neq \bar{u} \in \ker(\Delta_{q-1,U})$, so by Lemma \ref{lemma:Hodge-de Rahm} we have $\bar{H}^{q-1}(\Gamma;\sigma) \neq 0$. The representation $\sigma$ need not be irreducible, but by Proposition \ref{prop:red cohom direct integral} and Lemma \ref{lemma:disintegratipon weakly contained}, it weakly contains an irreducible representation with non vanishing Hausdorff $(q-1)$-cohomology.
    
    To construct a representation in $\mathrm{supp}(\pi) \cap \widehat{\Gamma}_{\mathrm{q-cohom}}$, we may repeat the above proof verbatim, using this time $(2)$ of Lemma \ref{lemma:non-Hausdorff characterization} instead of $(3)$, and thus an almost kernel for $\Delta_{q,V}^-$ instead of $\Delta_{q-1,V}^+$.\qedhere
    
\end{proof}
The following is our second main result for this section, which is a partial converse to Theorem \ref{thm:main thm finit prop 1}.

\begin{theorem}[Theorem \ref{thm:main_2_finiteness_intro}]\label{thm:main thm finit prop 2}
    Let $\Gamma$ be a group of type $\mathrm{FP}_\infty(\Q)$, and let $\sigma$ be a unitary representation of $\Gamma$ with $H^q(\Gamma;\sigma) \ne 0$. Let $\pi$ be a unitary representation that weakly contains~$\sigma$. Then, at least one of the following occurs:
    \begin{enumerate}
        \item  $\bar{H}^q(\Gamma;\pi) \ne 0$.
        \item  $\torH^q(\Gamma;\pi) \ne 0$.
        \item  $\torH^{q+1}(\Gamma;\pi) \ne 0$.
    \end{enumerate}
    In other words, $\pi$ is $q$-cohomological or $(q+1)$-cohomological. If it is not $q$-cohomological, then $H^{q+1}(\Gamma;\pi)$ has non-zero torsion.
\end{theorem}

\begin{proof}
If $\torH^q(\Gamma;\sigma) \ne 0$, then $\sigma$ weakly contains an irreducible representation with non-zero Hausdorff cohomology in degree~$q$ by Theorem \ref{thm:main thm finit prop 1}. 
If $\barH^q(\Gamma;\sigma)\ne 0$, then the same conclusion holds by Proposition~\ref{prop:red cohom direct integral}. 
In either case, we thus may assume that $\sigma$ is irreducible and that $\bar{H}^q(\Gamma;\sigma) \ne 0$. 

Denote the Hilbert spaces of $\pi$ and $\sigma$ by $V$ and $U$ respectively. Since $\bar{H}^{q}(\Gamma;\sigma) \ne 0$, $\mathrm{ker}(\Delta_{q,U}) \neq 0$  by Lemma \ref{lemma:Hodge-de Rahm}. Let $0 \ne \bar{u} \in \ker(\Delta_{q,U})$ be a unit vector, and denote by $\varphi$ the state $\varphi_{\bar{u}}$ on $M_{m_q}(C^*(\Gamma))$. Since $U \prec_\Gamma V$, Lemma \ref{lemma:reps of amplification} implies that $U^{m_q} \prec_{M_{m_q}(C^*(\Gamma))} V^{m_q}$. Thus $\varphi$ is the weak-$^*$ limit of a sequence of vector states $\varphi_i$ corresponding to unit vectors $\bar{v}_i \in V^{m_q}$. Since $U$ is irreducible, $U^{m_q}$ is irreducible as an $M_{m_q}(C^\ast\Gamma)$-representation by Lemma~\ref{lemma:reps of amplification}. Thus, every vector state from $U^{m_q}$ is approximated by vector states from $V^{m_q}$ rather than by convex combinations. See Lemma~\ref{lem:irred rep weak cont}, which is formulated for groups but is in fact a claim on representations of $C^*$-algebras.

From $\varphi(\Delta^2_{q}) =\|\Delta_{q,U}( \bar{u})\|^2= 0$ we conclude that
$$\|\Delta_{q,V}( \bar{v}_i)\|^2 = \varphi_i(\Delta^2_q) \xrightarrow[i \rightarrow \infty]{}0.$$
If $\mathrm{ker}(\Delta_{q,V}) \neq 0$, 
then we are in case $(1)$ by Lemma~\ref{lemma:Hodge-de Rahm}. Assume now that $\mathrm{ker}(\Delta_{q,V}) = 0$. Since $\overline{\mathrm{Im}(d_{q,V})} = \mathrm{ker}(d_{q+1,V})$, we obtain the orthogonal decomposition
\[V^{m_q} = \mathrm{ker}(d_{q+1,V}) \oplus \mathrm{ker}(d^*_{q,V}) = \mathrm{ker}(\Delta_{q,V}^+) \oplus \mathrm{ker}(\Delta_{q,V}^-).\] 
This decomposition is invariant under the self-adjoint operators $\Delta_{q,V}^+$ and $\Delta_{q,V}^-$. We denote orthogonal projections onto $\ker(\Delta_{q,V}^+)$ and $\ker(\Delta_{q,V}^-)$ by $P_+$ and $P_-$ respectively. We have:
$$\|\Delta_{q,V} (\bar{v}_i)\|^2 = \|\Delta_{q,V}^+(\bar{v}_i)\|^2 +\|\Delta_{q,V}^-(\bar{v}_i)\|^2 = \|\Delta_{q,V}^+(P_-\bar{v}_i)\|^2 +\|\Delta_{q,V}^-(P_+\bar{v}_i)\|^2.$$
From $\|\Delta_{q,V}( \bar{v}_i)\|^2 \to 0$ we get $\|\Delta_{q,V}^-(P_+\bar{v}_i)\|\to 0$ and $ \|\Delta_{q,V}^+(P_-\bar{v}_i)\| \to 0$ as $i\to\infty$.
Upon passing to a subsequence we may assume that $\|P_+\bar{v}_i\|$ or $\|P_-\bar{v}_i\|$ is bounded below by strictly positive constant.

In the first case, $P_+\bar{v}_i/\|P_+\bar{v}_i\|$ is an almost kernel of $\Delta_{q,V}^-$ and so $\torH^q(\Gamma;\pi) \ne 0$ by Lemma~\ref{lemma:non-Hausdorff characterization}. In the second case, $P_-\bar{v}_i/\|P_-\bar{v}_i\|$ is an almost kernel of $\Delta_{q,V}^+$ and so $\torH^{q+1}(\Gamma;\pi) \ne 0$ by Lemma~\ref{lemma:non-Hausdorff characterization}. 
\end{proof}

Our next goal is to prove Proposition \ref{prop:non-red is topological finit prop}, which says that if the cohomology of a representation $\rho$ is purely torsion in a certain degree $q$, then any representation $\pi$ that weakly contains $\rho$ also has non-zero torsion in degree $q$. That is, with extra information we may choose option $(2)$ in Theorem \ref{thm:main thm finit prop 2}. This will be an important step in proving that for semisimple groups, torsion in cohomology is a topological property (Theorem \ref{thm:non-Hausdorff is topological}). The proof of \ref{prop:non-red is topological finit prop} proves to be more difficult than what one may expect. The source of the difficulty is in determining the precise degree where non-zero torsion appears.

In the proof of Proposition \ref{prop:non-red is topological finit prop} we will use the following lemma. Two unitary representations $\pi$ and $\rho$ of a group $\Gamma$ are called \emph{disjoint} if they admit no isomorphic subrepresentations. Disjointness is defined in the same way for $*$-representations of $C^*$-algebras.
\begin{lemma}\label{assympot ortho of states}
    Let $A$ be a unital separable $C^*$-algebra together with a $*$-representation of $A$ on a Hilbert space~$V$. Let $u_n$ and $w_n$ be two sequences of unit vectors in $V$ such that the states $\varphi_{u_n}$ and $\varphi_{w_n}$ converge to states whose GNS representations are disjoint. Then $u_n$ and $w_n$ are asymptotically orthogonal:
    $$\langle w_n, u_n \rangle \rightarrow 0$$
\end{lemma}
\begin{proof}
    Denote the limit state of $\varphi_{u_n}$ and $\varphi_{w_n}$ by $\varphi_u$ and $\varphi_w$, where $u$ and $w$ are the associated cyclic vectors in the corresponding GNS representations $U$ and $W$. The space $V^2$ is an $M_2(A)$-representation, and the (column) vector $\bar{s}_n=(u_n, w_n )^t$ corresponds to a positive functional on this $C^*$-algebra. To avoid confusion, we denote positive functionals on $M_2(A)$ by $\psi$ and not by $\varphi$, which is reserved for positive functionals on $A$. Upon passing to a subsequence, we may assume that $\psi_{\bar{s}_n}$ converges to a positive functional $\psi_{\bar{s}}$, where $\bar{s}$ is a vector in some $M_2(A)$-representation $L'$. We consider the subspaces $L_1$ and $L_2$ of $L'$ 
    \[
    L_1=\begin{pmatrix} 1 &0 \\ 0 & 0 \end{pmatrix} L'~~\text{and}~~L_2=\begin{pmatrix} 0 & 0 \\ 0 & 1 \end{pmatrix} L', 
    \]
     which become $A$-representations via the diagonal embedding $A\hookrightarrow M_2(A)$. Obviously, $L_1$ and $L_2$ are isomorphic as $A$-representations.  If we denote $L=L_1\cong L_2$, we see that $L'$ is just the amplification $L^2$. Let $\iota$ be the embedding of $A$ into the upper left corner of $M_2(A)$. Then $\psi_{\bar{s}_n}(\iota (x))=\varphi_{u_n}(x)$ for every $x \in A$. Taking limits, it follows that $\psi_{\bar{s}}(\iota(x))=\varphi_{u}(x)$ for every $x \in A$. Hence the $A$-representation $L$ contains a copy of the GNS representation $V$ of $\varphi_{u}$. 
     
     Similarly, using the embedding of~$A$ into the lower right corner, it follows that $L$ contains a copy of the GNS representation $W$ of $\varphi_{w}$. As $U$ and $W$ are disjoint, $L$ decomposes as an orthogonal sum $L=(U \oplus W) \oplus L_0$, and $\bar{s} = ((u, w), 0)^t$ in this decomposition. Because $U$ and $W$ are orthogonal in~$L$ we obtain that 
     \[
\langle w_n,u_n \rangle=\psi_{\bar{s}_n}\Bigl(\begin{pmatrix} 0 & 1 \\ 0 & 0 \end{pmatrix}\Bigr)\xrightarrow{n\to\infty}\psi_{\bar{s}}\Bigl(\begin{pmatrix} 0 & 1 \\ 0 & 0 \end{pmatrix}\Bigr)=\langle w, v\rangle=0.\qedhere
     \]
\end{proof}
 
\begin{proposition}\label{prop:non-red is topological finit prop}
    Let $\Gamma$ be a group of type $\mathrm{FP}_\infty(\mathbb{Q})$. Let $\rho$ and $\pi$ be two unitary representations of $\Gamma$ such that $H^q(\Gamma;\rho)=\torH^q(\Gamma;\rho)\ne 0$ and $\rho \prec \pi$. Then $\torH^q(\Gamma;\pi) \ne 0$.
\end{proposition}

\begin{proof}
    Denote the representation spaces of $\pi$ and $\rho$ by $V$ and $U$. By Theorem \ref{thm:cohom of amplif} we may and will replace the representation $\pi$ with the infinite amplification $\infty \pi$ without loss of generality. 
    As a consequence and by $\rho \prec \pi$, each state of $U$ or $U^{m_q}$ can be approximated by states from $V$ or $V^{m_q}$, and not just by convex combinations of states. 
    
    By $\torH^q(\Gamma;\rho)\ne 0$ and Lemma \ref{lemma:non-Hausdorff characterization}, $\Delta_{q,U}^-$ has an almost kernel. That is, there exists a sequence of unit vectors $\bar{u}_i \in U^{m_q}$ such that $\bar{u}_i \in  \ker(\Delta_{q,U}^-)^\perp = \ker(\Delta_{q,U}^+)$ and $\Delta_{q,U}^-(\bar{u}_i) \rightarrow 0$.

   \paragraph{\textbf{Step I}}
   Fix $i\in\N$ and denote $\bar{u}=\bar{u}_i\in \ker\Delta_{q,U}^+$. It is $\Delta_{q,U}^-(\bar{u}) \ne 0$.  The state $\varphi_{\bar{u}}$ on $M_{m_q}(C^*(\Gamma))$ corresponding to $\bar{u}$ is approximated by a sequence of states $\varphi_{\bar{v}_n}$ corresponding to vectors $\bar{v}_n \in V^{m_q}$. We have the orthogonal decomposition $V^{m_q}=\ker(\Delta_{q,V}^-) \oplus \ker(\Delta_{q,V}^-)^{\perp}$, which is preserved by the operators $\Delta_{q, V}^\pm$. Let $P^-$ be the orthogonal projection of $V^{m_q}$ onto $\ker(\Delta_{q,V}^-)$, and $P^+$ the orthogonal projection onto $\ker(\Delta_{q,V}^-)^{\perp}$. Denote
    $$\bar{a}_n=P^+(\bar{v}_n), \quad \bar{b}_n=P^-(\bar{v}_n)$$
    so that $\bar{v}_n=\bar{a}_n+\bar{b}_n$.
    \paragraph{\textbf{Step II}} Next we show that $\bar{b}_n\to 0$. Assume the contrary. Upon passing to a subsequence, we may assume that $\| \bar{b}_n\|>\varepsilon$ uniformly for some $\varepsilon > 0$. By compactness of the state space we may also assume that $\varphi_{\bar{b}_n}$ converges to some non-zero positive linear functional $\varphi$ on $M_{m_q}(C^*(\Gamma))$. By Lemma \ref{lemma:reps of amplification}, the GNS representation of $\varphi$ is an amplification $W^{m_q}$ of some unitary $\Gamma$-representation $W \prec_\Gamma V$. Then $\varphi=\varphi_{\bar{b}}$ for some cyclic vector $\bar{b} \in W^{m_q}$. We claim that $\bar{b}\in W^{m_q}$ represents a non-zero Hausdorff cohomology class in $\bar H^q(\Gamma; W)$: Because of $\bar a_n \in \ker(\Delta_{q,V}^-)^{\perp} \subseteq \ker(\Delta_{q,V}^+)$ we have 
    \[\varphi_{\bar{b}}(\Delta_q^+)=\lim_{n \rightarrow \infty}\langle \Delta_{q,V}^+(\bar{b}_n),\bar{b}_n\rangle=\lim_{n \rightarrow \infty}\langle \Delta_{q,V}^+(\bar{a}_n+\bar{b}_n), \bar{a}_n+\bar{b}_n\rangle=\lim_{n \rightarrow \infty}\varphi_{\bar{v}_n}(\Delta_q^+)=\varphi_{\bar{u}}(\Delta_q^+) =0.\]
    Hence $\Delta_{q,W}^+(\bar{b})=0$. Furthermore $\Delta_{q,V}^-(\bar{b}_n)=0$ for every $n$ so $\Delta_{q,W}^-(\bar{b})=0$. Hence $\Delta_q(\bar{b})=0$, and $\bar{b}$ represents a non-zero Hausdorff cohomology class as claimed.

    Next we claim that the $\Gamma$-representations $W$ and $U$ are disjoint. By assumption, $\bar{H}^q(\Gamma;U)=0$. Hence $U$ has no subrepresentation with non-zero Hausdorff cohomology in degree $q$. On the other hand, $W$ has the property that every subrepresentation has non-zero \emph{Hausdorff} cohomology in degree $q$. To see this, note that $\bar{b}$ is a cyclic vector for the $M_{m_q}(C^*(\Gamma))$-representation $W^{m_q}$. For every subrepresentation $W_0$ of $W$, the projection $P_0$ of $W^{m_q}$ onto $W_0^{m_q}$ commutes with $M_{m_q}(C^*(\Gamma))$. Since $\bar{b}$ is cyclic, $P_0(\bar{b}) \ne 0$. But $\Delta_q \in M_{m_q}(C^*(\Gamma))$, so $\Delta_{q,W}$ commutes with $P_0$ and hence $\Delta_{q,W}(P_0(\bar{b}))=0$. This shows that $\bar{H}^q(\Gamma;W_0) \ne 0$, and therefore $W$ and $U$ are disjoint.
    
    The positive functionals $\varphi_{\bar{b}_n}$ and $\varphi_{\bar{v}_n}$ converge to functionals whose GNS representations are disjoint. By Lemma \ref{assympot ortho of states} we have $\langle \bar{v}_n,\bar{b}_n \rangle \rightarrow 0$. Hence $\langle\bar{b}_n,\bar{b}_n \rangle=\langle \bar{a}_n + \bar{b}_n,\bar{b}_n \rangle=\langle\bar{v}_n,\bar{b}_n \rangle \rightarrow 0$, which is a contradiction. 
    \paragraph{\textbf{Step III}} Since $\bar{b}_n=P^-(\bar{v}_n)$ decays to zero, we have $\varphi_{\bar{a}_n} \rightarrow \varphi_{\bar{u}}$. Remembering that $\bar{u}=\bar{u}_i$ -- an element in a sequence which is an almost kernel of $\Delta_{q,U}^-$ -- we showed that $\varphi_{\bar{u}_i}$ is approximated by a sequence $(\varphi_{\bar{v}_{i,j}})_j$, where each $\bar{v}_{i,j} \in V^{m_q}$ is a unit vectors in $\ker(\Delta_{q,V}^-)^\perp$. For each $i$ we choose $\bar{w}_i=\bar{v}_{i,j} \in\ker(\Delta_{q,V}^-)^\perp $ so that $\|\Delta_{q,V}^-(\bar{w}_{i}) \| \leq 2\|\Delta_{q,U}^-(\bar{u}_i) \|$. Then the sequence $\bar{w}_i$ is an almost kernel of $\Delta_{q,V}^-$. Thus $\torH^q(\Gamma;\pi) \ne 0$. 
\end{proof}

\subsection{Consequences for lattices}\label{subsec:3.2}
Theorem \ref{thm:main thm finit prop 2} and Proposition \ref{prop:non-red is topological finit prop} give far reaching consequences regarding the abundance of cohomological representations for certain groups. Guichardet \cite[Example~1]{Guichardet72} showed that for the free group $F_2$, every unitary representation admits non-zero cohomology in degree $1$. In this subsection we will prove two analogues of this theorem for higher-rank lattices. 

Let $G$ be a semisimple Lie group, and $\Gamma \leq G$ a lattice. Extensive research had been carried by many authors regarding the cohomology of the left regular representation, $H^*(\Gamma;\lambda_{\Gamma})$. This cohomology relates to the $\ell^2$-invariants of $\Gamma$: the case of non-zero Hausdorff cohomology corresponds to non-zero $\ell^2$-betti numbers, and the case of \emph{non-Hausdorff} cohomology corresponds to \emph{Novikov-Shubin} invariants different from $\infty^+$. We refer the reader to \cite{LuckBook}.

Fix a maximal compact subgroup $K$ in $G$. The \emph{fundamental rank} of $G$ is $\ell(G)=\mathrm{rank}_\mathbb{C}(G)-\mathrm{rank}_\mathbb{C}(K)$. By a celebrated result of Harish-Chandra, the fundamental rank is zero precisely when $G$ has discrete series representations. The fundamental rank relates to the $\ell^2$-invariants of lattices in $G$, as the following theorem shows.

\begin{theorem}[{\cite{CheegerGromov}, \cite[Section~11]{LohoueSalah} and \cite[Theorem~1.1]{Olbrich}}]\label{thm:l2 invariants}
   Let $\Gamma \leq G$ be a lattice in a connected semisimple Lie group with finite center and without compact factors.  
    \begin{enumerate}
        \item $\bar{H}^q(\Gamma;\lambda_\Gamma) \ne 0$ (equivalently, $b^{(2)}_q(\Gamma) \ne 0$) if and only if  $q=\frac{\mathrm{dim(G/K)}}{2}$ and $\ell(G)=0$.
        \item Assume further that $\Gamma$ is uniform. Then $H^q(\Gamma;\lambda_\Gamma)$ has non-zero torsion if and only if $q \in [\frac{\mathrm{dim(G/K)}-\ell(G)}{2}+1,\frac{\mathrm{dim(G/K)}+\ell(G)}{2}]$ and $\ell(G)>0$.
    \end{enumerate}
\end{theorem}
 We will now prove our results regarding the abundance of cohomological representations for higher-rank lattices.
\begin{theorem}[Theorem \ref{intro:latt corollary 1}]
    Let $G$ be a higher-rank simple Lie group which is center free and $\Gamma$ a uniform lattice in $G$. Denote by $d=\mathrm{dim}(G/K)$ the dimension of the symmetric space and by $\ell(G)=\mathrm{rank}_\mathbb{C}(G)-\mathrm{rank}_\mathbb{C}(K)$ the fundamental rank. Assume that $\ell(G)>0$. If $\pi$ is a unitary representation of $\Gamma$ that is not a finite sum of amplifications of finite dimensional representations, Then $\pi$ is cohomological. Moreover, $\torH^q(\Gamma;\pi) \ne 0$ for every $q \in [\frac{d-\ell(G)}{2}+1,\frac{d+\ell(G)}{2}]$.
\end{theorem}
\begin{proof}
    The infinite dimensional representation $\pi$ always weakly contains the regular representation $\lambda_\Gamma$. If $\pi$ contains some weakly mixing representation this is \cite[Corollary~D]{BoutHoud}. If $\pi$ does not contain a weakly mixing representation, our assumption implies that $\pi$ contains a representation of the form $\bigoplus \pi_i$, where the $\pi_i$ are finite dimensional and $\dim(\pi_i) \rightarrow \infty$. In this case, $\pi$ weakly contains the regular representation by Charmenability and property $(T)$ \cite[Theorem~1.3]{LevitSlutskiVigdorovich}. By Theorem \ref{thm:l2 invariants}, $H^q(\Gamma;\lambda_{\Gamma})$ has non-zero torsion for $q \in [\frac{\mathrm{dim(G/K)}-\ell(G)}{2}+1,\frac{\mathrm{dim(G/K)}+\ell(G)}{2}]$. Therefore Proposition \ref{prop:non-red is topological finit prop} yields that $H^q(\Gamma;\pi)$ also has non-zero torsion in this range. 
\end{proof}
\begin{theorem}[Theorem \ref{intro:latt corollary 2}]
    Let $G$ be a higher-rank simple Lie group which is center free, or a higher-rank simple $p$-adic group. If $G$ is real, assume that it has discrete series representations, namely that the fundamental rank is zero. Let $\Gamma$ be a lattice in $G$. Then, $H^*(\Gamma;\pi) \ne 0$ for every unitary representation $\pi$. In particular, if $G$ is $p$-adic, then $H^{\mathrm{rank}(G)}(\Gamma;\pi)\ne 0$ for every  non-trivial unitary representation $\pi$.
\end{theorem}
\begin{proof}
    If $\pi$ is weakly mixing, then it  weakly contains the left regular representation \cite[Corollary~D]{BoutHoud}, which is cohomological by \ref{thm:l2 invariants}. Therefore $\pi$ is cohomological as well by Theorem \ref{thm:main thm finit prop 2}. Assume now that $\pi$ is not weakly mixing. We may assume without loss of generality that $\pi$ is finite dimensional. As such, it is a \emph{tracial} representation, and one may consider its associated $\ell^2$-betti number (see \cite[Definition~3.3]{FournierFacioSauer}). By Theorem \ref{thm:l2 invariants} for Lie groups and Theorem \ref{thm:cass_padic} for $p$-adic groups, $\Gamma$ has precisely one non-zero $\ell^2$-betti number, and hence $\Gamma$ has non-zero $\ell^2$-Euler characteristic. The $\ell^2$-Euler characteristics associated with different tracial representations coincide (\cite[Proposition~3.5]{FournierFacioSauer}, and we get that some $\ell^2$-betti number associated with $\pi$ is non-zero. This means that $H^q(\Gamma;\pi)\ne 0$ for some $q \geq 0$, as we wanted. If $G$ is a higher-rank $p$-adic group, the only degree where non-trivial representations admit cohomology is at the rank.
\end{proof}

We conclude this section with a complement to Theorem~\ref{thm:l2 invariants} concerning lattices in semisimple Lie groups with infinite center. 

\begin{lemma}\label{lem: non-Hausdorff for amenable central subgroups}
Let $A\hookrightarrow \Gamma\twoheadrightarrow Q$ be a central extension of groups of type $\mathrm{FP}_\infty(\C)$, and assume that $A$ is infinite. 
Assume that $H^i(Q, \lambda_Q)=0$ for $i\le n-1$ and $H^{n}(Q,\lambda_Q)\ne 0$. 
Then $H^{n}(\Gamma, \lambda_\Gamma)\ne 0$ or $H^{n+1}(\Gamma, \lambda_\Gamma)\ne 0$. If, in addition, $H^n(Q, \lambda_Q)$ is purely torsion, then $H^{n}(\Gamma, \lambda_\Gamma)\ne 0$. 
Furthermore, $H^\ast(\Gamma, \lambda_\Gamma)$ is purely torsion in all degrees. 
\end{lemma}

\begin{example}
The following example shows that in the previous statement it is possible that $H^n(\Gamma, \lambda_\Gamma)=0$ if $H^n(Q, \lambda_Q)$ is not purely torsion. 
Let $Q=\pi_1(\Sigma_g)$ be the fundamental group of a surface of genus $g\ge 2$. The unit sphere bundle $ST\Sigma_g$ over $\Sigma_g$ is an orientable $S^1$-bundle over $\Sigma_g$. Let $\Gamma=\pi_1(ST\Sigma_g)$. On the level of fundamental groups, one obtains a central extension $\Z\hookrightarrow \Gamma\twoheadrightarrow Q$. Then $H^1(\Gamma, \lambda_\Gamma)=0$ since $\Gamma$ is non-amenable and has vanishing first $\ell^2$-Betti number by the Cheeger-Gromov theorem. Moreover, $H^1(Q,\lambda_Q)=\bar H^1(Q, \lambda_Q)\ne 0$ by non-amenability and a positive first $\ell^2$-number of $\Sigma_g$. 
\end{example}

\begin{proof}
The trivial representation of $A$ is weakly contained in $\lambda_A$, hence by continuity of unitary induction $\lambda_Q=\lambda_{\Gamma/A}$ -- viewed as a $\Gamma$-representation -- is weakly contained in $\lambda_{\Gamma}$. 

Next we consider the Hochschild-Serre spectral sequence for our group extension with coefficients in $\lambda_Q$. 
The $E^2$-term is given by 
\[
E_2^{p,q}=H^p(Q, H^q(A, \lambda_Q)).
\]	
Since the $A$-action on $\ell^2(Q)$ is trivial, we have 
$H^q(A, \lambda_Q)\cong H^q(A,\C)\otimes_\C \ell^2(Q)$ and the $Q$-action on this module is the diagonal action of $Q$ on $H_q(A,\C)$, which is trivial by centrality, and the canonical $Q$-action on $\ell^2(Q)$. So $H^q(A,\lambda_Q)$ is just a finite-dimensional amplification of the regular representation of~$Q$. 
By assumption the first $n$ vertical lines $E_2^{i,\ast}$ for $i\in\{0,1,\dots, n-1\}$ vanish. 
This implies $H^i(\Gamma, \lambda_Q)=0$ for $i\le n-1$. Further, 
$H^n(Q, \lambda_Q)\cong E_2^{n,0}\cong E_\infty^{n,0}$ since every differential $d^r\colon E_r^{n-r,r-1}\to E_r^{n,0}$, $r\ge 2$, starts from a zero entry. 
So we obtain that 
\[H^n(\Gamma, \lambda_Q)\cong \bigoplus_{i=0}^n E_\infty^{i, n-i}\cong E_\infty^{n,0}\cong H^n(Q,\lambda_Q)\ne 0.\]
Now the statement follows from Theorem~\ref{thm:main thm finit prop 2} and the fact that $\lambda_\Gamma$ weakly contains $\lambda_Q$. If $H^n(Q,\lambda_Q)$ is purely torsion, we use Proposition~\ref{prop:non-red is topological finit prop} instead. 

The $\ell^2$-Betti numbers of $\Gamma$ vanish by the Cheeger-Gromov theorem~\cite[Corollary~6.75 on p.~275]{LuckBook}. Hence $H^\ast(\Gamma, \ell^2(\Gamma))$ is purely torsion. 
\end{proof}

For lattices in semisimple Lie groups with infinite center we complement Theorem~\ref{thm:l2 invariants} based 
on Lemma~\ref{lem: non-Hausdorff for amenable central subgroups}. The following result is certainly not optimal. 

\begin{theorem}
Let $G$ be a connected semisimple Lie group with finite center and without compact factors. Assume that $\pi_1(G)$ is infinite. Let $\Gamma$ be a lattice in its universal cover~$\widetilde G$. 
\begin{enumerate}
\item If $\Gamma$ is uniform and $\ell(G)>0$, then $H^q(\Gamma, \lambda_\Gamma)\ne 0$ for $q=\frac{\mathrm{dim(G/K)}-\ell(G)}{2}+1$. 
\item If $\Gamma$ is uniform and $\ell(G)=0$, then $H^q(\Gamma, \lambda_\Gamma)\ne 0$ for $q=\frac{\mathrm{dim}(G/K)}{2}$ or $q=\frac{\mathrm{dim}(G/K)}{2}+1$. 
\item If $\Gamma$ is non-uniform and $\ell(G)=0$, then $H^q(\Gamma, \lambda_\Gamma)\ne 0$ for some $q\in [2, \frac{\mathrm{dim}(G/K)}{2}+1]$. 
\end{enumerate} 
Morever, $H^\ast(\Gamma, \lambda_\Gamma)$ is purely torsion in all degrees. 
\end{theorem}

\begin{proof}
Let $\pi\colon\widetilde G\to G$.
Because of lack for a reference, we show that every lattice $\Gamma<\widetilde G$ is a central extension of a lattice $\Lambda<G$. To this end, it suffices to show that $\Lambda=\pi(\Gamma)$ is discrete in $G$~\cite[Theorem~1.13 on p.~23]{raghunathan}. Suppose that $\pi(\Gamma)$ is not discrete, and let $H\subset G$ be the topological closure of $\pi(\Gamma)$. By the Borel density theorem~\cite[Section~5]{raghunathan}, $\Gamma$ is Zariski dense 
in~$\mathrm{Ad}(G)$, hence $\Lambda$ is Zariski dense in~$G$. In particular, $\Lambda$ and $H$ are not abelian. 
The commutator lift map $f\colon G\times G\to \widetilde G, (g,h)\mapsto [\tilde g,\tilde h]$ is well-defined and continuous. There are non-commuting $\lambda, \lambda'\in \Lambda$ arbitrarily close to the identity (otherwise $H$ had to be abelian), and so $1\ne f(\lambda, \lambda')\in\Gamma$ is arbitrarily close to the identity, which contradicts the discreteness of $\Gamma$. The lattice $\Gamma<\widetilde G$ is uniform if and only if the lattice $\Lambda<G$ is uniform. 

The first and second part of the theorem follow from Lemma~\ref{lem: non-Hausdorff for amenable central subgroups} and Theorem~\ref{thm:l2 invariants}. The third part follows similary using that a non-uniform lattice in~$G$ has the same non-vanishing of $\ell^2$-Betti numbers as a uniform lattice in~$G$ by Gaboriau's theorem~\cite{gaboriau}. 
\end{proof}

\subsection{$\ell^2$-invisibility}We use a modification of the proof of Lemma~\ref{lem: non-Hausdorff for amenable central subgroups} to make some slight progress in a question about $\ell^2$-invariants. 
One calls a group $\Gamma$ \emph{$\ell^2$-invisible} if $H^\ast(\Gamma, \lambda_\Gamma)$ vanishes in all degrees. A group of type $\mathrm{FP}_\infty(\C)$ is $\ell^2$-invisible if all its $\ell^2$-Betti numbers vanish and all its Novikov-Shubin invariants have the value~$\infty^+$. The following conjecture by Gromov was shown to follow from the Novikov conjecture by Yu~\cite{Yu}. See~\cite{Lott} or~\cite[Chapter~12]{LuckBook} for more background. 

\begin{conjecture}[Zero-in-the-spectrum conjecture]
    The fundamental group of a closed aspherical manifold is not $\ell^2$-invisible.
\end{conjecture}

Apart from manifold groups the following, more general question from~\cite{SauerThumann} is wide open. 
We emphasize that there are $\ell^2$-invisible groups of type $\mathrm{F}_\infty$ (but infinite cohomological dimension). Thompson's group~V is such an example~\cite{SauerThumann}. 

\begin{question}
    Is there an $\ell^2$-invisible group of type $\mathrm{F}$? 
\end{question}

The following result relies on a modification of the proof of Lemma~\ref{lem: non-Hausdorff for amenable central subgroups}. It shows that amenable extensions of groups with a non-vanishing $\ell^2$-Betti numbers (under some mild conditions) are not $\ell^2$-invisible. We will tacitly use some modern technology about $\ell^2$-Betti numbers, like L\"uck's dimension theory, that is not used elsewhere in the paper. 
We refer to L\"uck's book~\cite{LuckBook} as a background reference. 

\begin{theorem}\label{thm: non-Hausdorff for amenable normal subgroups}
Let $A\hookrightarrow \Gamma\twoheadrightarrow Q$ be an extension of groups such that 
\begin{enumerate}
\item $A$ is amenable and all of its Betti numbers are finite, and 
\item $\Gamma$ and $Q$ are of type $\mathrm{FP}_\infty(\C)$, and 
\item $Q$ is sofic or locally indicable. 
\end{enumerate}
If $\Gamma$ is $\ell^2$-invisible, then all $\ell^2$-Betti numbers of $Q$ vanish. 
\end{theorem}

Before we turn to the proof, we comment on the surprising condition of soficity or locally indicability. 
This is related to L\"uck's twisted $\ell^2$-Betti number conjecture. The $i$-th $\ell^2$-Betti number can expressed 
via L\"uck's dimension function for modules over the group von Neumann algebra~$\calN(\Gamma)$: 
\[ 
b_i^{(2)}(\Gamma)=\dim_{\calN(\Gamma))} H_i\bigl(\Gamma, \calN(\Gamma)\bigr)\]
Let $V$ be a finite-dimensional $\Gamma$-representation which is not necessarily unitary. By regarding $V\otimes_\C \calN(\Gamma)$ as a left $\Gamma$-module by the diagonal action and as a right $\calN(\Gamma)$-module by the action on the right factor, the homology $H_i(\Gamma, V\otimes_\C \calN(\Gamma))$ is still a $\calN(\Gamma)$-module. A special case of L\"uck's conjecture says that for every group $\Gamma$ of type $\mathrm{FP}_\infty(\C)$ and every finite-dimensional $\Gamma$-representation~$V$ we have
\begin{equation}\label{eq: twisted betti} \dim_\C(V)\cdot b_i^{(2)}(\Gamma)=\dim_{\calN(\Gamma)} H_i\bigl(\Gamma, V\otimes_\C\calN(\Gamma)\bigr).
\end{equation}
The following result by Boschheidgen-Jaikin-Zapirain~\cite{BoschheidgenJaikin} and Kielak-Sun~\cite{kielak+sun} is the reason for the third condition in 
Theorem~\ref{thm: non-Hausdorff for amenable normal subgroups}.

\begin{theorem}\label{thm: proved case of twisted betti}
The equality~\eqref{eq: twisted betti} holds for sofic and locally indicable groups of type $\mathrm{FP}_\infty(\C)$. 
\end{theorem}

The following proof is quite similar to the proof of Lemma~\ref{lem: non-Hausdorff for amenable central subgroups}. Since L\"uck's twisted $\ell^2$-Betti number conjecture is formulated for homology, we work with homology. The use of Theorem~\ref{thm:main thm finit prop 2} in the proof of Lemma~\ref{lem: non-Hausdorff for amenable central subgroups} can be replaced in homological setting by a nice short argument involving the reduced $C^\ast$-algebra. 

\begin{proof}[Proof of Theorem~\ref{thm: non-Hausdorff for amenable normal subgroups}]
The epimorphism $\Gamma\to Q$ is denoted by~$\pi$. 
The trivial representation of $A$ is weakly contained in $\lambda_A$, hence by continuity of unitary induction $\lambda_Q=\lambda_{\Gamma/A}$ -- viewed as a $\Gamma$-representation -- is weakly contained in $\lambda_{\Gamma}$. 
This implies that the $\ast$-homomorphism $\C[\Gamma]\to \C[Q]$ extends to a $\ast$-homomorphism $C_r^\ast(\Gamma)\to C_r^\ast(Q)$. By~\cite[Lemma~12.3 on p.~442]{LuckBook} the $\ell^2$-invisibility and $\mathrm{FP}_\infty(\C)$-property of~$\Gamma$ imply that  
\[ H_\ast(\Gamma, C_r^\ast(\Gamma))=0.\]
Let $P_\ast$ be a projective $\C[\Gamma]$-resolution of~$\C$. 
Since $C_r^\ast(\Gamma)\otimes_{\C[\Gamma]}P_\ast$ is a projective $C_r^\ast(\Gamma)$-complex with vanishing homology, it is homotopy equivalent as $C_r^\ast(\Gamma)$-complex to the zero complex by the fundamental theorem of homological algebra. Therefore 
\[ C_r^\ast(Q)\otimes_{\C[\Gamma]} P_\ast\cong C_r^\ast(Q)\otimes_{C_r^\ast(\Gamma)} C_r^\ast(\Gamma)\otimes_{\C[\Gamma]}P_\ast\simeq 0.\]
Here we view $C_r^\ast(Q)$ as a $C_r^\ast(\Gamma)$-module. In particular, we obtain that 
\[ \calN(Q)\otimes_{\C[\Gamma]} P_\ast\cong \calN(Q)\otimes_{C_r^\ast(Q)}C_r^\ast(Q)\otimes_{\C[\Gamma]}P_\ast\simeq 0.\]
Thus, 
\begin{equation*} H_\ast(\Gamma, \calN(Q))=0. 
\end{equation*}	
Next we show that the $\ell^2$-Betti numbers of $Q$ vanish. Assume the contrary. Let $n\in\N$ be such that 
\[
\dim_{\calN(Q)}H_i(Q,\calN(Q))=
\begin{cases}
0 &\text{for $0\le i\le n$;}\\
>0 & \text{for $i=n+1$.}
\end{cases}\]

Next we consider the Hochschild-Serre spectral sequence for our group extension with coefficients in the $\C[\Gamma]$-module $\calN(Q)$. It follows from L\"uck's dimension theory that the subcategory of $\calN(Q)$-modules of dimension zero is a Serre subcategory of the category of all $\calN(Q)$-modules. This essentially means that we the homological algebra  dimension zero modules behaves almost as if they would be actually zero. 
The $E^2$-term is given by 
\[
E^2_{pq}=H_p(Q, H_q(A, \calN(Q))).
\]	
Since the $A$-action on $\calN(Q)$ is trivial, we have 
$H_q(A, \lambda_Q)\cong H_q(A,\C)\otimes_\C \calN(Q)$ and the $Q$-action on this module is the diagonal action of $Q$ on $H_q(A,\C)$ and the canonical $Q$-action on $\calN(Q)$. By Theorem~\ref{thm: proved case of twisted betti} we have $\dim_{\calN(Q)}E^2_{pq}=0$ for every $p\in \{0,\dots, n\}$ and every $q\in\N$. Similar as in the proof of Lemma~\ref{lem: non-Hausdorff for amenable central subgroups}, this implies that \[\dim_{\calN(Q)}E^\infty_{n+1,0}=\dim_{\calN(Q)}E^2_{n+1,0}=\dim_{\calN(Q)}H_{n+1}(Q,\calN(Q))>0.\] 
Hence $\dim_{\calN(Q)}H_{n+1}(\Gamma, \calN(Q))>0$. This is a contradiction. 
\end{proof}

\section{The cohomology of semisimple Groups}\label{sec5}
In this section we turn our attention to semisimple groups and study their unitary cohomology, with an emphasis on torsion. Let $G$ be a semisimple group, as in Terminology \ref{terminology:semisimple}. We start by showing that for a unitary representation $\pi$ of $G$, both the homology and the cohomology of $G$ with coefficients in $\pi$ decompose as a direct sum of a Hiblert space and a space with the trivial topology, thus completing the proof of Theorem \ref{intro:torsion-reduced}. For this, considering a uniform lattice in $G$ suffices, together with the results of 
\S\ref{sec3}. In order to study torsion however, a lattice is not enough. In \S\ref{subsec:5.1} we introduce the notion of a \textit{cohomological witness}: a dense countable subgroup which reflects the unitary cohomology theory of $G$ up to a given degree. We show how to construct cohomological witnesses in an arithmetic way, that is, as projections of $S$-arithmetic lattices in semisimple groups with many factors. In \S\ref{subsec:5.2} we use the cohomological witness to study torsion cohomology for $G$. In \S\ref{subsec:isolation proofs} we discuss the relation between Hausdorffness of cohomology and isolation of cohomological representations, and provide proofs for the theorems in \S\ref{subsec:isolation}.

To start, we consider a unitary representation $\pi$ of $G$ and decompose it as in Theorem \ref{thm: reduced cohom} to $\pi=\pi_\mathrm{cohom} \oplus \mathring{\pi}$. We would like to claim that this decomposition yields a canonical decomposition of the homology and cohomology into a direct sum of a Hiblert space and a space with the trivial topology.
\begin{theorem}[Theorem \ref{intro:torsion-reduced}]\label{thm:decomp-homology-cohomology}
    Let $G$ be a semisimple group and let $\pi$ be a unitary representation.
    Then in every degree $q$ both $H_q(G;\pi)$ and $H^q(G;\pi)$ are decomposed canonically, as topological vector spaces, into a direct sum of a Hilbert space and a space with the trivial topology. 
\end{theorem}
Note that the facts regarding cohomology already follow from Theorem \ref{thm: reduced cohom}. For Lie groups, the corresponding facts about homology follow from Theorem \ref{thm:homology and cohomology}. However, we consider semisimple groups that are not real as well, for which we do not have a Poincar\'e duality result.
\begin{proof}
    As mentioned above, the facts regarding the cohomology follow from Theorem \ref{thm: reduced cohom}. Furthermore, Theorem \ref{thm:homology and cohomology} implies that for any degree $q$, $\torH_q(G;\pi_\mathrm{cohom})=0$ and $\bar{H}_q(G;\mathring{\pi})=0$. We are left to show that $\bar{H}_q(G;\pi)=H_q(G;\pi_\mathrm{cohom})$ is isomorphic, as a topological vector space, to a Hilbert space.

    Let $\Gamma$ be a uniform lattice in $G$, whose existence is guaranteed by \cite[Corollary~1.12]{BorelHarder}. By the homological Shapiro lemma \cite[Lemma~C.35]{Petersen}, $H_q(\Gamma;\pi_\mathrm{cohom}) 
    \cong H_q(G;\mathrm{Ind}_\Gamma^G(\mathrm{\pi_\mathrm{cohom}}))$, and the latter decomposes as
    $$H_q(G;\pi_\mathrm{cohom}) \oplus H_q(G;\pi_\mathrm{cohom} \otimes L^2_0(G/\Gamma)).$$
     $H_q(G;\pi_\mathrm{cohom})$ is a Hausdorff topological vector and is complete by Theorem \ref{cor:homIsNotMetr}. By remark \ref{rem:finiteness prop homology}, $H_q(\Gamma;\pi_\mathrm{cohom})$ can be calculated via a chain complex of Hilbert spaces. Therefore $H_q(\Gamma;\pi_\mathrm{cohom}) / \overline{\{0\}}$ is topologically isomorphic to a Hilbert space. A complete Hausdorff subspace of such a space is again isomorphic to a Hilbert space, and the proof is complete.
\end{proof}

\subsection{The cohomological witness}\label{subsec:5.1}
Let $L$ be a locally compact group, and $\Gamma \leq L$ a dense countable subgroup. The restriction of an irreducible unitary representation of $L$ to $\Gamma$ is still irreducible, and we obtain a map $\widehat{L}\to \widehat{\Gamma}$. This map is easily seen to be continuous. We denote by $\widehat{\Gamma}_L$ the closure of the image of this map in $\widehat{\Gamma}$.
\begin{definition}
    We say that a dense subgroup $\Gamma<L$ is a \emph{cohomological witness} if the following conditions are satisfied.
    \begin{enumerate}
       \item The restriction map $\widehat{L}\to \widehat{\Gamma}$ is a homeomorphism onto its image, identifying $\widehat{L}$ as a subspace of $\widehat{\Gamma}$.
       \item For every $L$-representation $\pi$, if $\pi|_{\Gamma}$ weakly contains a $\Gamma$-representation $\rho$ such that $\barH^q(\Gamma;\rho) \ne 0$, then $\pi$ weakly contains a representation from $\widehat{L}_{q-\mathrm{cohom}}$.
       \item For every representation $\pi$ of $L$, the map induced by the restriction, $H^q(L;\pi) \to H^q(\Gamma;\pi)$, is a topological isomorphism.
    \end{enumerate}
    $\Gamma$ is said to be a \textit{cohomological witness up to degree} $N$ if the above conditions are satisfied for $q \leq N$.
\end{definition}
The next rather general theorem enables us to construct cohomological witnesses arithmetically. 
\begin{theorem}[Theorem \ref{intro:witness}]\label{main:wittness}
    Let $K$ be a number field and $\mathbf{G}$ a
    connected and simply connected almost simple $K$-algebraic group.
    Let $S$ be a finite set of places that contains all the Archimedean places, and let $S \subseteq T$ be another set of places such that $T - S$ contains at least one place over which $\mathbf{G}$ is isotropic. Denote by $T_{\mathrm{is}} \subseteq T$ the set of places in $T$ where $\mathbf{G}$ is isotropic.  Let $\Gamma$ be a group in the (well defined) commensurability class of $\mathbf{G}(\mathcal{O}_T)$ and $G=\prod_{v \in S} \mathbf{G}(K_v)$, along with the natural embedding $\Gamma\to G$. 
    %If $T$ is infinite, denote $N=\infty$. Otherwise denote $N=\abs{T}-3$. 
    Then $\Gamma$ is a cohomological witness for $G$ up to degree $|T_{\mathrm{is}}|-1$. In particular, if $T$ is infinite it is a cohomological witness.
\end{theorem}
In order to prove Theorem \ref{main:wittness} we need some preparation. We define the group $M=\mathbf{G}(\mathbb{A}_T)$, so that $\Gamma$ is a lattice in $M$ (where $\mathbb{A}_T$ is the ring of $T$-adeles). Write $M=G \times H$ where $H=\mathbf{G}(\mathbb{A}_{T-S})$. Therefore $H$ is the restricted product of $p$-adic groups. We consider $\Gamma$ both as a lattice in $M$ and as a (dense) subgroup of $G$, via the projection $M \to G$. 

The first observation that we need is that pulling back representations from $G$ to $M$ yields a topological isomorphism in cohomology.
\begin{lemma}\label{lem:topological HS}
    Let $G$ and $M$ be the groups constructed above, and $\pi$ a unitary representation of $M$ that factors through $G$. The projection $M \to G$ induces a topological isomorphism $H^q(G;\pi) \cong H^q(M;\pi)$ in every degree $q \geq 0$.
\end{lemma}
\begin{proof}
We use the Hochschild-Serre spectral sequence \cite[Theorem 9.1]{Blanc} with respect to the short exact sequence $1 \rightarrow H \rightarrow M \rightarrow G \rightarrow 1$. We have $H^q(H;\pi)=0$ unless $q=0$, because $H$ acts trivially and it is the (restricted) product of simple $p$-adic groups and compact groups, and therefore never has cohomology with trivial coefficients for $q > 0$. In degree zero, $H^0(H;\pi)$ is just $\pi$ as a $G$-representation. Thus the spectral sequence stabilizes immediately, and yields an isomorphism $H^q(G;\pi)\xrightarrow{\cong} H^q(M;\pi)$. The isomorphism is continuous because it is induced by the projection $M \to G$. The inverse is continuous since it is induced by the inclusion of $G \hookrightarrow M$ into $M$. 
\end{proof}
A fundamental spectral gap result due to Clozel states that the trivial representation is isolated in the automorphic dual associated with any rational structure on an almost simple, simply connected algebraic group over a local field \cite[Theorem~3.1]{Clozel}. We will require the following generalization due to Bader and Sauer.
\begin{theorem}\cite[Theorem~4.13]{BaderSauer}\label{clozel}
    Let $F$ be a number field and $\mathbf{L}$ be an almost simple, connected and simply connected $F$-algebraic group. Let $S'$ be a set of places that contains all the Archimedean ones, and let $\Delta \subset \mathbf{L}(F)$ be an $S'$-arithmetic subgroup. Let $L$ be the restricted product of the groups $\mathbf{L}(F_v)$ for $v \in S'$, and consider $\Delta$ as a lattice in $L$. Fix $v' \in S'$ such that $\mathbf{L}(F_{v'})$ is non-compact. Then, for any unitary $\mathbf{L}(F_{v'})$-representation $U$, the tensor product $U \otimes L^2_0(L/\Delta)$ does not weakly contain the trivial representation.
\end{theorem}
Note that Clozel's theorem only promises the spectral gap of $L^2_0(L/\Delta)$. The above result by Bader-Sauer is stronger, as it ensures that the spectral gap of $L^2_0(L/\Delta)$ is preserved under tensor products. The absorption of spectral gap under tensor products is in fact a general phenomenon for unitary representations of semisimple groups \cite{Gorfine}.

The next piece of information that we need is the following result of Bader-Sauer, which promises the vanishing of cohomology below the rank of a semisimple group for representations with spectral gap for each factor.
\begin{theorem}\cite[Theorem~5.6]{BaderSauer}\label{thm:BS vanishing semisimple}
    Let $L=\prod_i \mathbf{L}(F_i)$ be a non-compact semisimple group where, for each $i$, $F_i$ is a local field and $\mathbf{L}_i$ is a connected almost simple $F_i$-group. Let $L_i=\mathbf{L}_i(F_i)$. Let $\pi$ be a unitary representation such that $\pi|_{L_i}$ does not weakly contain the trivial representation for each non-compact factor $L_i$. Then $H^q(L;\pi)=0$ for every $q < \mathrm{rank}(L)$.
\end{theorem}

The following theorem is based on \cite[\S6]{BaderSauer}. It enables us to use a Shapiro-like lemma for not necessarily uniform lattices in semisimple groups, and relates the cohomology of a semisimple group and of its lattice below the rank.
\begin{theorem}\label{lem:poly shapiro}
    Let $F$ be a number field and $\mathbf{L}$ be an almost simple, connected and simply connected $F$-algebraic group. Let $S'$ be a finite set of places containing all the Archimedean places, and $\Delta=\mathbf{L}(\mathcal O_{S'})$ an $S'$-arithmetic group, which is a lattice in $L=\prod_{v \in S'}\mathbf{L}(F_v)$. Denote by $S'_{\mathrm{is}}$ the subset of places in $S'$ over which $\mathbf{L}$ is isotropic. Let $\pi$ be a unitary $L$-representation, and $q \leq \abs{S'_\mathrm{is}}-1$. Then restriction yields a topological isomorphism $H^q(L;\pi) \cong H^q(\Delta;\pi)$.
\end{theorem}
\begin{proof}
    Assume first that $\Delta$ is a uniform lattice. Let $\pi$ be a unitary $L$-representation. By Shapiro Lemma \ref{lem: shapiro lemma}, the map $H^q(L;\mathrm{Ind}_\Delta^L(\pi)) \to H^q(\Delta;\pi)$ is a topological isomorphism. To finish the proof in this case, we that need $H^q(\pi \otimes L^2_0(L / \Delta))=0$ for $q \leq \abs{S'_\mathrm{is}}-1$. But this is provided by Theorem \ref{thm:BS vanishing semisimple}, once we note that $\pi \otimes L^2_0(L / \Delta)$ has a spectral gap for each non-compact factor by Theorem \ref{clozel}.
    
    Assume now that $\Delta$ is non-uniform. Note that in this case, $S'=S'_\mathrm{is}$ so we will not differentiate between these for the rest of the proof. We use the polynomial cohomology and terminology discussed in \cite[\S6]{BaderSauer}. By the work of Bestvina-Eskin-Wortman (\cite[Corollary~5]{BEW13} and \cite[Theorem~6.11]{BaderSauer}), the polynomial filling rank of $\Delta$, $\mathrm{rank}_{\mathrm{pol}}(\Delta)$, satisfies $\mathrm{rank}_{\mathrm{pol}}(\Delta) \geq \abs{S'}$. Moreover, $\Delta$ is a universally integrable lattice \cite[Proposition~6.19]{BaderSauer}. Fix a degree $q$ which satisfies $q \leq \abs{S'}-1$. In particular, $q < \mathrm{rank}_{\mathrm{pol}}(\Delta)$.
    
    First, let $\pi$ be a unitary $\Delta$-representation. The map $H^q(L;\mathrm{Ind}_\Delta^L(\pi)) \to H^q(\Delta;\pi)$ is surjective by \cite[Theorem~6.23]{BaderSauer} and is obviously continuous.
    
    We now take $\pi$ to be a unitary $L$-representation. Decompose $\mathrm{Ind}_\Delta^L(\pi)$ as $\mathrm{Ind}_\Delta^L(\pi)=\pi \oplus (\pi \otimes L^2_0(L /\Delta))$. The map $H^q(L;\mathrm{Ind}_\Delta^L(\pi)) \to H^q(\Delta;\pi)$, when restricted to the component $\pi$ of this decomposition, coincides with the map induced by the restriction \cite[Theorem~6.26]{BaderSauer}. It is moreover injective on this component \cite[Theorem~6.26]{BaderSauer}. Since the map on $\mathrm{Ind}_\Delta^L(\pi)$ is surjective, to complete the algebraic isomorphism we need to show that $H^q(L;\pi \otimes L^2_0(L/\Delta))=0$. But this follows from Theorem \ref{thm:BS vanishing semisimple} for $q<|S'|$, since $\pi \otimes L^2_0(L/\Delta)$ has a spectral gap for each non-compact factor of $L$ by Theorem \ref{clozel}.

    We have a continuous bijective map $H^q(L;\pi) \to H^q(\Delta;\pi)$, and are left to show that it is a topological isomorphism. We decompose $\pi=\pi_\mathrm{cohom} \oplus \mathring{\pi}$, as in Theorem \ref{thm: reduced cohom}, and prove the claim for each summand. For $\mathring{\pi}$, the topology on $H^q(L;\mathring\pi)$ is trivial, and hence the continuous bijection to $H^q(\Delta;\mathring\pi)$ must be a topological isomorphism. For $\pi_\mathrm{cohom}$, it is enough to show that the topology on $H^q(\Delta;\pi_\mathrm{cohom})$ is Hausdorff. The representation $\pi_\mathrm{cohom}$ decomposes as a direct sum of finitely many isotypic representations, and we may work with each of them separately. Namely, assume that we have a representation of the form $\rho \otimes U$, where $\rho$ is an irreducible representation and the action on the multiplicity module $U$ is trivial. The isomorphism $H^q(L;\rho) \to H^q(\Delta;\rho)$ is topological. Indeed, $H^q(L;\rho)$ is finite dimensional by Theorem \ref{thm:product formula}, and therefore so is $H^q(\Delta;\rho)$. But finite dimensional cohomologies are automatically Hausdorff \cite[Lemma~IV.3.4]{BorelWallach}. Therefore the cohomologies of $\rho \otimes U$ are also Hausdorff, both for $L$ and for $\Delta$, and in fact are isomorphic to Hilbert spaces, by Theorem \ref{thm:cohom of amplif}. The open mapping theorem now says that the isomorphism on $\pi_\mathrm{cohom}$ is topological, and this finishes the proof.
\end{proof}
In order to treat the cases where $\abs{T}=\infty$ in Theorem \ref{main:wittness} we use the following lemma, which relates the Hausdorff cohomology of a direct limit to the Hausdorff cohomology of the step groups.
\begin{lemma}\label{lemma:HausdorffCohomLimitgroups}
    Let $\Delta=\varinjlim \Delta_i$ be a direct limit of an increasing sequence of subgroups, and consider a unitary representation $V$ of $\Delta$. Fix $q \geq 0$. If $\bar{H}^q(\Delta ; V) \ne 0$ then $\bar{H}^q(\Delta_i ; V) \ne 0$ for $i$ large enough.
\end{lemma}
\begin{proof}
    The claim is trivial for $q=0$. We assume that $q>0$. We use the inhomogeneous bar complex (see Remark \ref{rem:inhomogeneous}). Suppose that $\bar{H}^q(\Delta ; V) \ne 0$, and assume towards contradiction that $\bar{H}^q(\Delta_i ; V) = 0$ for infinitely many $i$'s. Let $v \in C(\Delta^q,V)$ be a $q$-cocycle that represents a non-zero \textit{Hausdorff}cohomology class. Restricting $v$ to $C(\Delta_i^q,V)$ yields a cocycle $v|_{\Delta_i}$. The assumption that $\bar{H}^q(\Delta_i ; V) = 0$ for infinitely many $i$'s implies that $v|_{\Delta_i}$ is in the closure of the coboundaries for infinitely many $i$'s which we enumerate $i_k$.
    % If it is a coboundary for some $i$, then it is also a coboundary for any $j < i$. Therefore, the assumption that $\bar{H}^q(\Delta_i ; V) = 0$ for infinitely many $i$'s implies that $v$ is in the closure of the coboundaries for infinitely many $i$'s, or that it is an actual coboundary for \textit{all} $i$.
    Fix an increasing sequence of finite sets $A_k$ such that $A_k \subseteq \Delta_{i_k}$ and $A_k$ exhausts $\Delta$.
     There exists $u_k \in C(\Delta_{i_k}^{q-1},V)$ such that $|d_q(u_k)(x) -v(x)|<\frac{1}{k}$ for all $x \in A_k^q$. We choose an arbitrary extension $\tilde{u}_k$ of $u_k$ to $C(\Delta^{q-1},V)$ and note that $d_q(\tilde{u}_k)|_{\Delta_{i_k}} = d_q(u_k)$. Since $A_k$ exhaust $\Delta$, $d_q(\tilde{u}_k) \xrightarrow[k \to \infty]{} v$ pointwise which is a contradiction.
\end{proof}
We are now ready to prove Theorem \ref{main:wittness}.
\begin{proof}[Proof of Theorem \ref{main:wittness}]
    We begin by showing that the restriction map $\widehat{G} \to \widehat{\Gamma}$ is a homeomorphism onto its image. It is enough to show that for every two $G$-representations $\pi$ and $\rho$ such that $\pi$ is irreducible and $\pi \prec_{\Gamma} \rho$, we have that $\pi \prec_G \rho$. Let $\pi$ and $\rho$ be such representations. By continuity of the induction we have that $\mathrm{Ind}_\Gamma^M(\pi) \prec_M \mathrm{Ind}_\Gamma^M(\rho)$. Restricting the representations $\pi$ and $\rho$ from $G$ to $\Gamma$ is the same as pulling them back to $M$ and then restricting to $\Gamma \leq M$. Therefore the induced representations satisfy $\mathrm{Ind}_\Gamma^M(\pi)=\pi \oplus (\pi \otimes L^2_0(M/\Gamma))$ and $\mathrm{Ind}_\Gamma^M(\rho)=\rho \oplus (\rho \otimes L^2_0(M/\Gamma))$. We therefore get
    $$\pi \prec_M \rho \oplus (\rho \otimes L^2_0(M/\Gamma))$$
    By Lemma \ref{weakContain}, $\pi$ is weakly contained in one of the summands. Assume that it is weakly contained in $\rho \otimes L^2_0(M/\Gamma)$. Choose a place $v \in T-S$ such that $\mathbf{G}(K_v)$ is non-compact. Then $\pi|_{\mathbf{G}(K_v)}$ and $\rho|_{\mathbf{G}(K_v)}$ are both trivial, and therefore $1 \prec_{\mathbf{G}(K_v)} L^2_0(M/\Gamma)$, contradicting Clozel's Theorem \ref{clozel}. So $\pi \prec_M \rho$, and restricting to $G$ finishes the proof.

    Next, let $\pi$ be a representation of $G$, and assume that there exists a $\Gamma$-representation $\rho$, weakly contained in $\pi|_\Gamma$, such that $\barH^q(\Gamma;\rho) \ne 0$. If $T$ is infinite, Lemma \ref{lemma:HausdorffCohomLimitgroups} shows that there exists a finite set $S\subseteq T' \subseteq T$ such that $\barH^q(\mathbf{G}(\mathcal O_{T'});\rho) \ne 0$. We assume without loss of generality that $\abs{T'_\mathrm{is}} > q$. Write $\Gamma_{T'}=\mathbf{G}(\mathcal O_{T'})$ and $M_{T'}=\prod_{v \in {T'}} \mathbf{G}(K_v)$, so that $\Gamma_{T'}$ is a lattice in $M_{T'}$. If $T$ were finite to begin with, we take $T'=T$, $M_{T'}=M$ and $\Gamma_{T'}=\Gamma$. Theorem \ref{lem:poly shapiro} implies that $\bar{H}^q(M_{T'};\mathrm{Ind}_{\Gamma_{T'}}^{M_{T'}}(\rho))$ is non-zero. So, $\mathrm{Ind}_{\Gamma_{T'}}^{M_{T'}}(\rho)$ contains an irreducible $q$-cohomological subrepresentation $\sigma$. Since $|T'_\mathrm{is}| > q$, the product formula Theorem \ref{thm:product formula} implies that there exists a place $v \in T'_\mathrm{is}$ such that $\sigma|_{\mathbf{G}(K_v)}$ is trivial. Since $\sigma$ is irreducible, and
    $$\sigma \leq \mathrm{Ind}_{\Gamma_{T'}}^{M_{T'}}(\rho) \prec_{M_{T'}} \mathrm{Ind}_{\Gamma_{T'}}^{M_{T'}}(\pi)=\pi \oplus (\pi \otimes L^2_0(M_{T'} / \Gamma_{T'})),$$$\sigma$ is weakly contained in one of the summands $\pi$ and $\pi \otimes L^2_0(M_{T'} / \Gamma_{T'})$. Because $\sigma|_{\mathbf{G}(K_v)}$ is trivial, $\sigma$ can not be weakly contained in $\pi \otimes L^2_0(M_{T'} / \Gamma_{T'})$ by Theorem \ref{clozel} so $\sigma \prec_{M_{T'}} \pi$.

    Lastly, consider the restriction from $G$ to $\Gamma$. Let $\pi$ be a $G$-representation, and $q \leq \abs{T_\mathrm{is}} - 1$, then $H^q(G;\pi) \cong H^q(M ; \pi)$ by Lemma \ref{lem:topological HS}. We need to show that the restriction from $M$ to $\Gamma$ yields an isomorphism $H^q(M;\pi) \to H^q(\Gamma ; \pi)$. This follows from Theorem \ref{lem:poly shapiro} in case $T$ is finite. In case $T$ is infinite, this is \cite[Theorem~E]{BaderSauer}. In \cite{BaderSauer} the assumption is that $T=V(K)$ is the set of all places of $K$, but the proof applies verbatim for any infinite set $T$ which contains all the Archimedean places. We note that in our case, the isomorphism in \cite[Theorem~E]{BaderSauer} for the restriction of a representation from an adelic group to a lattice is topological. Indeed, the representation $\pi$ is pulled back from $G$, and as a $G$-representation we can decompose it as $\pi=\pi_\mathrm{cohom} \oplus \mathring\pi$ according to Theorem \ref{thm: reduced cohom}. The proof that the continuous bijective map $H^q(M;\pi) \to H^q(\Gamma;\pi)$ is a topological isomorphism is the same as the proof of the last part of Theorem \ref{lem:poly shapiro}. That is, working with $\mathring \pi$ and $\pi_\mathrm{cohom}$ separately, and using the fact that $\pi_\mathrm{cohom}$ is a finite direct sum of isotypic representations, and each isotypic representation has Hausdorff cohomology for $M$ and for $\Gamma$.
\end{proof}
\begin{proof}[Proof of Theorem \ref{intro:witness-Lie}]
    This is a special case of Theorem \ref{main:wittness}.
\end{proof}
The cohomological witnesses of infinite degree that Theorem \ref{main:wittness} provides are not of finite type, and are therefore not useful for exploiting the results of \S\ref{sec3}. For semisimple groups arising from simple groups over a global field, Theorem \ref{main:wittness} provides also a cohomological witness of type $FP_\infty$ up to any degree. We would like to have this tool for general semisimple groups so we include the following variation of the theorem:
\begin{theorem}[Theorem~\ref{intro:finite witness}]\label{thm:witness exists}
    Let $G=\prod_{i=1}^n \mathbf{G}_i(k_i)$ be a semisimple group, where each $k_i$ is a local field of characteristic zero and each factor $\mathbf{G}_i$ is an almost simple, connected and simply connected $k_i$-algebraic group. Then for every $N$ there exists a countable dense subgroup $\Gamma<G$ of type $\mathrm{FP}_\infty(\Q)$
    which is a cohomological witness up to degree $N$.
    In particular, there exists a countable dense subgroup $\Gamma<G$ of type $\mathrm{FP}_\infty(\Q)$
    which is a cohomological witness up to the cohomological degree of $G$, $d(G)$.\footnote{By the \textit{cohomological dimension} of a locally compact group we mean the smallest $N$ such that $H^q(G;V)=0$ for any Fr\'echet $G$-module $V$ and any $q>N$. The cohomological dimension equals the dimension of the symmetric space for semisimple Lie groups, and for semisimple $p$-adic groups it equals the rank. The cohomological dimension is additive in products of groups.}
\end{theorem}
\begin{proof}
    For each $i$ we choose a number field $K_i$ such that $k_i=(K_i)_v$ for some place $v$ and a $K_i$-form of $\mathbf{G}_i$. By \cite[Corollary~1.12]{BorelHarder}, we may choose the number field $K_i$ and the $K_i$-form of $\mathbf{G}_i$ such that for every $v' \in V_\infty(K)-\{v\}$ the group $\mathbf{G}_i$ is $K_{v'}$-anisotropic. We choose a set of places $T_i$, which contains all the Archimedean places, such that $\abs{T_{i,\mathrm{is}}} > N$. By Theorem \ref{main:wittness}, the group $\Gamma_i=\mathbf{G}_i(\mathcal O_{T_i})$ is a cohomological witness up to degree $N$ for $G=\mathbf{G}(k_i)$. This group is a lattice in the group $M_i=G_i \times H_i$, where $M_i=\prod_{v' \in T_i} \mathbf{G}((K_i)_{v'})$ and $H_i=\prod_{v' \in T_i - \{v\}} \mathbf{G}((K_i)_{v'})$. Note that since we chose the $K_i$-form to satisfy that for any $v' \in V_\infty(K)-\{v\}$ the group $\mathbf{G}_i$ is $K_{v'}$-anisotropic, we have that $H_i$ is a product of compact and $p$-adic groups. By construction, $\Gamma=\prod \Gamma_i$ is a uniform lattice in $M$.

    We claim that $\Gamma=\prod \Gamma_i$ is a cohomological witness up to degree $N$ for $G$. This does not follow abstractly, but rather from repeating the proof of Theorem \ref{main:wittness}. The main caveat is that whenever we induce a representation $\pi$ from $\Gamma$ to $M$, the induced representation $\pi \otimes L^2(M/\Gamma)$ is a direct sum of several summands and not just $\pi \oplus (\pi \otimes L^2_0(M/ \Gamma)$. That is, we have $\mathrm{Ind}_\Gamma^M(\pi)=\pi \otimes L^2(M/\Gamma)=\pi \oplus \pi'$, where $\pi'$ is the sum of representations of the form $\pi \otimes L^2_0(M_{i_1}/\Gamma_{i_1}) \otimes ... \otimes L^2_0(M_{i_k}/\Gamma_{i_k})$ for some non-empty set of indices $i_1, ... , i_k$. Working with each summand separately, the proof of items $(1)$ and $(2)$ in the definition of a cohomological witness follow exactly the same steps as in the proof of Theorem \ref{main:wittness}. For part $(3)$, we need to show that restriction in from $G$ to $\Gamma$ yields a topological isomorphism in degrees $N$ and below. Pulling back a representation $\pi$ from $G$ to $M$ is the same as in the simple case (using Lemma \ref{lem:topological HS}). We are left to show that $H^q(M;\pi) \cong H^q(\Gamma;\pi)$. We use Shapiro's lemma \ref{lem: shapiro lemma}, and are left to show that $H^q(M;\pi')=0$. That is, to show the vanishing of cohomology for $q \leq N$ for every summand of the form  $\pi \otimes L^2_0(M_{i_1}/\Gamma_{i_1}) \otimes ... \otimes L^2_0(M_{i_k}/\Gamma_{i_k})$. Fix one such summand $\sigma$, and let $i$ be an index such that $L^2_0(M_i / \Gamma_i)$ appears in it. We consider it as an $M_i$-representation. Then it has a spectral gap for every non-compact factor of $M_i$ by Theorem \ref{clozel}, and hence $H^q(M_i;\sigma)=0$ for $q \leq N$ by Theorem \ref{thm:BS vanishing semisimple}. This means that in the Hochschild-Serre spectral sequence for the decomposition $1 \to M_i \to M \to M/M_i$, the first $N$ rows of the $E^2$ page vanish. Thus $H^q(M;\sigma)$ must vanish as well for $q \leq N$. 
\end{proof}
\subsection{Torsion cohomology}\label{subsec:5.2}
Our next goal is to exploit the results of \S\ref{sec3} in order to study torsion in cohomology. Let $G$ be a semisimple group of the form $G= \prod_{i=1}^n \mathbf{G}_i(k_i)$, where each $k_i$ is a local field of characteristic zero, and $\mathbf{G}_i$ is connected and simply connected almost simple algebraic $k_i$-group. We denote by $G_i$ the group of $k_i$-points $\mathbf{G}_i(k_i)$. Note that in Terminology \ref{terminology:semisimple}, we include under ``semisimple group'' the case of connected semisimple Lie groups with finite center that are not necessarily simply connected algebraic groups. The results for these cases follow easily from our main results, via the following observation.
\begin{remark}\label{rem:connected semisimple}
    Let $G'$ be a connected semisimple Lie group with finite center. There exists a group $G$ which satisfies the assumptions above such that we have either $p:G \rightarrow G'$ or $p:G \rightarrow G'$, where $p$ is a surjection with finite kernel. It follows from \cite[Lemma~3.4]{BaderSauer} that every cohomological representation factors through $p$, and that $p$ induces a topological isomorphism on cohomology. Therefore all of the results of this section carry on to the connected semisimple case.
\end{remark}
We will mostly be interested in the Archimedean case: the case of simple $p$-adic groups is well known (in higher-rank they never admit torsion cohomology, and in rank-$1$ the torsion is fully understood: it is always in degree $1$, and occurs exactly when there are almost invariant vectors). However, the mixed fields case is interesting enough, and our methods work similarly for both cases, so we chose to work in wider generality. 
\begin{theorem}[Theorem \ref{intro: main thm1}]\label{main1}
    Let $G$ be a semisimple group and let $\pi$ be a unitary representation of $G$.
    For every degree $q$: 
     \[ \torH^q(G;\pi)\neq 0  \Longrightarrow \mathrm{supp}(\pi)' \cap \widehat{G}_{\mathrm{q-cohom}} \neq \emptyset ~~\text{and}~~\mathrm{supp}(\pi)' \cap \widehat{G}_{\mathrm{(q-1)-cohom}} \neq \emptyset.\]
    Moreover, $\pi$ essentially weakly contains irreducible representations $\rho_1\in \widehat{G}_{\mathrm{q-cohom}}$ and $\rho_2\in \widehat{G}_{\mathrm{(q-1)-cohom}}$ such that $\rho_1$ and $\rho_2$ are inseparable by neighborhoods in $\mathrm{supp}(\pi)$. That is, if $U_1$ is a neighborhood of $\rho_1$ in $\widehat{G}$, and $U_2$ is a neighborhood of $\rho_2$, then $\mathrm{supp}(\pi) \cap U_1 \cap  U_2 \ne \emptyset$.
\end{theorem}
\begin{remark}
It is important to note that the representations $\rho_1$ and $\rho_2$ in Theorem $\ref{main1}$ might be the same representation which is cohomological in two successive degrees. If however $\rho_1 \neq \rho_2$ then $\rho_1$ and $\rho_2$ are points which demonstrate that $\widehat{G}$ is not Hausdorff. 
\end{remark}
\begin{proof}[Proof of Theorem \ref{main1}]
Let $\pi$ be a representation of $G$ such that $\torH^q(G;\pi) \ne 0$. Let $N$ be the cohomological dimension of $G$, and let $\Gamma$ be a cohomological witness of type $\mathrm{FP}_\infty(\Q)$ up to degree $N+1$, as promised by Theorem \ref{thm:witness exists}. Restricting the representation $\pi$ to $\Gamma$, we get that $\mathring{H}^q(\Gamma;\pi) \ne 0$. By Theorem \ref{thm:main thm finit prop 1} for $\Gamma$, we get that $\pi$ weakly contains irreducible cohomological representations in degrees $q$ and $q-1$. By property $(2)$ of the cohomological witness, we get the existence of irreducible $G$-representations $\rho_1\in \widehat{G}_{\mathrm{q-cohom}}$ and $\rho_2 \in \widehat{G}_{\mathrm{(q-1)-cohom}}$ that are weakly contained in $\pi$ as $\Gamma$-representations. By property $(1)$ of the cohomological witness, they are weakly contained in $\pi$ as $G$-representations. 

To show that $\rho_1$ and $\rho_2$ are essentially weakly contained in $\pi$, we may assume without loss of generality that $\pi=\mathring\pi$. By definition, any cohomological representation in $\mathrm{supp}(\mathring\pi)$ is in $\mathrm{supp}(\mathring\pi)'$. We are left to show that there exists such a pair which is inseparable by neighborhoods in $\mathrm{supp}(\pi)$. Assume towards contradiction that every such pair is separable by neighborhoods in $\mathrm{supp}(\pi)$. Recall from Theorems \ref{thm:cass_padic}, \ref{thm:vogan-zuckerman} and \ref{thm:product formula} that $\widehat{G}_\mathrm{cohom}$ is finite. Therefore we can find disjoint open subsets $U$ and $V$ in $\mathrm{supp}(\pi)$, such that $U$ contains $\mathrm{supp}(\pi)\cap\widehat{G}_{\mathrm{q-cohom}}$ and $V$ $\mathrm{supp}(\pi)\cap\widehat{G}_{\mathrm{(q-1)-cohom}}$. Since $G$ is type I, we can decompose $\pi$ uniquely as a direct integral over $\mathrm{supp}(\pi)$ (Theorem \ref{thm:type I groups}). Decomposing the measure into disjoint measures supported on $U$ and $\mathrm{supp}(\pi) \setminus U$ yields a decomposition $\pi=\pi_1 \oplus \pi_2$, with $\mathrm{supp}(\pi_1) \subseteq \overline{U} \subseteq \supp(\pi)\backslash V$ and $\mathrm{supp}(\pi_2)=\mathrm{supp}(\pi) \setminus U$. $\pi_1$ does not weakly contain any $(q-1)$-cohomological representation, and $\pi_2$ does not weakly contain any $q$-cohomological representation. Using what we showed above nither  $\pi_1$ nor $\pi_2$ admits non-zero torsion in degree $q$, contradicting the fact that $H^q(G;\pi) \cong H^q(G;\pi_1) \oplus H^q(G;\pi_2)$.
\end{proof}

\begin{theorem}[Theorem \ref{intro: main thm2}]\label{main2}
    Let $G$ be a semisimple group and let $\pi$ be a unitary representation of $G$.
    Then $\pi$ has torsion cohomology in degrees $q$ or $q+1$ if and only if it essentially weakly contains a cohomological representation in degree $q$, that is
    \[ \torH^q(G;\pi)\neq 0~\text{or}~\torH^{q+1}(G;\pi)\neq 0 \iff \mathrm{supp}(\pi)' \cap \widehat{G}_{\mathrm{q-cohom}} \neq \emptyset.  \]
\end{theorem}

\begin{proof}
In the first direction, assume that $\torH^q(G;\pi)\neq 0$ or that  $\torH^{q+1}(G;\pi)\neq 0$. Both cases imply, by Theorem \ref{main1}, that $\mathrm{supp}(\pi)' \cap \widehat{G}_{\mathrm{q-cohom}} \neq \emptyset$.

We prove the opposite direction. Let $N$ be the cohomological dimension of $G$, and $\Gamma$ be a cohomological witness of type $\mathrm{FP}_\infty(\Q)$ up to degree $N+1$, as promised by Theorem \ref{thm:witness exists}. Let $\sigma \in \mathrm{supp}(\pi)' \cap \widehat{G}_{\mathrm{q-cohom}}$. We may assume without loss of generality that $\pi=\mathring{\pi}$. Indeed, we are interested in torsion cohomology, and passing to $\mathring{\pi}$ does not affect the weak containment of $\sigma$ in $\pi$ since we assumed the weak containment is essential. Applying Theorem \ref{thm:main thm finit prop 2} to $\Gamma$, we get that $H^q(\Gamma; \pi) \ne 0$ (cases $(1)$ and $(2)$ in Theorem \ref{thm:main thm finit prop 2}), or that $\torH^{q+1}(\Gamma; \pi) \ne 0$ (case $(3)$). Since $\Gamma$ is a cohomological witness, so restriction from $G$ to $\Gamma$ yields an isomorphism in cohomology in degrees $q$ and $q+1$. But out assumption that $\pi=\mathring{\pi}$ does not allow $H^q(\Gamma;\pi)$ to be Hausdorff, as may happen by case $(1)$ of Theorem \ref{thm:main thm finit prop 2}.
\end{proof}
Combining Theorems \ref{main1} and \ref{main2}, we get the following corollary:
\begin{corollary}\label{nonIsol}
    Assume that $\pi \in \widehat{G}_{\mathrm{q-cohom}}$ is non-isolated in the unitary dual of $G$. Then, there exists a representation $\rho$ in $\widehat{G}_{\mathrm{(q+1)-cohom}}$ or in $\widehat{G}_{\mathrm{(q-1)-cohom}}$ that is inseparable by neighborhoods from $\pi$. 
\end{corollary}
 Corollary \ref{nonIsol} does not imply necessarily that $\pi$ is a non-Hausdorff point in the unitary dual. It may be that the promised representation $\rho$ is $\pi$ itself, that is, that $\pi$ contributes cohomology in two successive degrees.

Next we will discuss which cohomological properties of representations are \emph{topological properties}. Let $G$ be a semisimple group, $\pi$ a unitary representation of $G$ and $\mu_\pi$ the unique measure class on $\widehat{G}$ that realizes $\pi$ as a direct integral (see Theorem \ref{thm:type I groups}). We say that a property of the unitary representation $\pi$ is \emph{topological} if it depends only on the support $\mathrm{supp}(\pi)$, and not on the specific measure class $\mu_\pi$. Proposition \ref{prop:red cohom direct integral} implies that the existence of Hausdorff cohomology and the existence of Hausdorff cohomology in a certain degree $q$ both depend on the atoms of the measure $\mu_\pi$. As a result, these properties are not topological. The following theorem shows that the existence of cohomology and the existence of torsion cohomology are topological properties.
\begin{theorem}[Theorem \ref{intro:topological propery1}]\label{thm:topological property1}
    Let $G$ be a semisimple group and let $\pi$ be a unitary representation of $G$.
    Then $\pi$ has cohomology in some degree if and only if it weakly contains a cohomological representation, and $\pi$ has cohomological torsion in some degree if and only if it essentially weakly contains a cohomological representation, that is 
    \[ \exists_{q\in\N}~H^q(G;\pi)\neq 0 \iff \mathrm{supp}(\pi) \cap \widehat{G}_{\mathrm{cohom}} \neq \emptyset,  \]
    and 
    \[ \exists_{q\in\N}~\torH^q(G;\pi)\neq 0 \iff \mathrm{supp}(\pi)' \cap \widehat{G}_{\mathrm{cohom}} \neq \emptyset.  \]
\end{theorem}
\begin{proof}
    The second part of the theorem is an immediate corollary from Theorem \ref{main2}. One implication of the first part is Proposition \ref{prop:red cohom direct integral}. The second implication follows from the second part of the theorem, and from the observation that if $(\mathrm{supp(\pi)}-\mathrm{supp(\pi)'}) \cap \widehat{G}_\mathrm{cohom} \ne \emptyset$, then $\pi$ has a cohomological subrepresentation and its Hausdorff cohomology is non-zero. 
\end{proof}
The next result shows that having torsion cohomology in a certain degree $q$ is also a topological property.
\begin{theorem}[Theorem \ref{intro:topological propery}]\label{thm:non-Hausdorff is topological}
    Let $G$ be a semisimple group and let $\pi$ and $\rho$ be unitary representations of $G$.
    If $\rho$ is weakly contained in $\pi$ and $\torH^q(G;\rho) \ne 0$, then $\torH^q(G;\pi) \ne 0$. In particular, if $\mathrm{supp}(\pi)=\mathrm{supp}(\rho)$ then 
    \[ \torH^q(G;\pi)\neq 0 \qquad \Longleftrightarrow \qquad \torH^q(G;\rho) \neq 0. \]
\end{theorem}

\begin{proof}
    Let $\pi$ and $\rho$ be as above. Assume without loss of generality that $\bar{H}^q(G;\rho)=0$. Indeed, we may decompose $\rho$ as in Theorem \ref{thm: reduced cohom} into $\rho_\mathrm{cohom} \oplus \mathring{\rho}$ such that the cohomology of $\mathring{\rho}$ is purely torsion. Let $\Gamma$ be a group of type $\mathrm{FP}_\infty(\Q)$ which is a cohomological witness up to degree $N+1$, where $N$ is the cohomological dimension of $G$. We have that $H^q(G;\pi) \cong H^q(\Gamma;\pi)$ and $H^q(G;\rho) \cong H^q(\Gamma;\rho)$ are topological isomorphisms. The desired conclusion follows from Proposition \ref{prop:non-red is topological finit prop}.
\end{proof}

Theorem \ref{thm:topological property1} shows that every representation that essentially weakly contains a cohomological representation admits non-zero torsion in cohomology. This completely resolves the question of identifying which unitary representations admit non-zero torsion in cohomology. The question about having non-zero torsion in a certain degree is more subtle. Indeed, if a representation $\pi$ essentially weakly contains some irreducible $q$-cohomological representation, Theorem \ref{main2} allows two options for the degree where non-zero torsion occurs: $q$ or $q+1$. Both options can occur. To illustrate this, consider the group $G=\SL_2(\R)$, and the trivial representation $1$. It is cohomological in degrees $0$ and $2$. Let $\sigma$ be a representation of $\SL_2(\R)$ with almost invariant vectors. Considering $1$ as a cohomological representation in degree $q=0$, by Theorem \ref{main2} $\sigma$ must admit torsion cohomology in degree $0$ or $1$. The cohomology in degree $0$ is always Hausdorff, so we must have torsion in degree $q+1 = 1$. Now consider $1$ as a cohomological representation in degree $q=2$. By Theorem \ref{main2}, $\sigma$ must admit torsion cohomology in degree $2$ or $3$. The cohomological dimension of $\SL_2(\R)$ is $2$, so the torsion must occur in degree $q=2$. In  a more general setting, we point out the following general phenomenon: 
\begin{remark}\label{rem:4-degrees}
    Let $G$ be a semisimple group, and let $\pi$ be a unitary representation of $G$. 
    Suppose that $\pi$ essentially weakly contains a representation that is cohomological in degree $q$ we wish to understand if $\pi$ has torsion in degree $q$. 
    
     Assume $\pi$ does not have torsion in degree $q$. By Theorem \ref{main1} $\pi$ essentially weakly contains a $(q-1)$-cohomological representation. By Theorem \ref{main2} $\pi$ must have torsion in degree $q$ or $q-1$, so by our assumption it has torsion in degree $q-1$. It now follows from Theorem \ref{main1} that $\pi$ essentially weakly contains a $(q-2)$-cohomological representation. In addition, since there is no torsion in degree $q$, $\pi$ must have torsion in degree $q+1$ by Theorem \ref{main2}; and so must essentially weakly contain a $(q+1)$-cohomological representation by Theorem \ref{main1}. To sum it all up, if there is no torsion in degree $q$, $\pi$ must essentially weakly contain representations in all four degrees $q-2,q-1,q$ and $q+1$.

    On the other hand, if we assume that $\pi$ essentially weakly contains representations in all four degrees $q-2,q-1,q$ and $q+1$ Theorems \ref{main1} and \ref{main2} do not give us enough information to determine weather $\pi$ has torsion in degree $q$. In other words, if $\pi$ essentially weakly contains irreducible cohomological representations in four successive degrees, we can not tell weather there is torsion cohomology in the third degree. 
\end{remark}
Situations as in Remark \ref{rem:4-degrees} give rise to the following natural question, which asks whether Theorem \ref{main1} has an inverse direction:
\begin{question}\label{q: inverse main1}
    Let $G$ be a semisimple group. Suppose that $\pi$ is a representation of $G$ that essentially weakly contains representations $\sigma_1 \in \widehat{G}_{\mathrm{q-cohom}}$ and $\sigma_2 \in \widehat{G}_{\mathrm{(q-1)-cohom}}$, such that $\sigma_1$ and $\sigma_2$ are inseparable by neighborhoods in $\mathrm{supp}(\pi)$. Is it necessarily that $\torH^q(G;\pi) \ne 0$?
\end{question}
The discussion in Remark \ref{rem:4-degrees} naturally occurs in many examples. The problem with determining the degree always occurs, as illustrated above, whenever a representation $\pi$ essentially weakly contains cohomological representations in four successive degrees that are inseparable by neighborhoods in $\mathrm{supp}(\pi)$. We will refer to these situations as situations of ``four successive degrees''. In \S\ref{sec6}, we will study torsion cohomology for the groups $\SO^\circ(n,1)$ and $\mathrm{SU}(n,1)$. As we will see, for the group $\SO^\circ(2k,1)$ the situation of four successive degrees never occurs,
and we will use this to give a full solution to Problem~\ref{prob:main} for these groups.
This is not the case for the groups $\SO^\circ(2k+1,1)$ and $\mathrm{SU}(n,1)$. For these groups the situation of four successive degrees occurs. For the group $\mathrm{SU}(2,1)$ we are able to provide a complete answer, and show the existence of torsion cohomology in the third of four successive degrees as above. For $\SO^\circ(2k+1,1)$ we are able to resolve the situation using Lie algebra cohomology, this time showing vanishing of torsion in the third of four successive degrees. In \S \ref{sec7} an alternative, constructive proof is provided for $\SO^\circ(3,1)=\PSL_2(\C)$. In particular, we get the following.
\begin{theorem}\label{thm:neg answer sl2C}
      The answer to Question \ref{q: inverse main1} is negative, and a counterexample may be given in $\SO^{\circ}(2k+1,1)$.
\end{theorem}

After we answering Question \ref{q: inverse main1} in the negative, we are left with the problem of finding an ``if and only if" criterion for the existence of torsion cohomology in a certain degree $q$. To elaborate on the specific problem that we point out, we consider the following situation: let $\pi \in \widehat{G}_\mathrm{q-cohom}$ be non-isolated point, and $\{\pi_n \}$ be a sequence of distinct representations that converges to $\pi$ in $\widehat{G}$. Up to passing to a subsequence, \cite[Theorem~3]{voganIsol} tells us that we can always understand this convergence in terms of convergence of characters of tori and parabolic induction. More precisely, it guarantees the existence of a parabolic subgroup $P=MN$, an irreducible admissible representation $\sigma$ of $M$ and a sequence of one-dimensional characters $\chi_n$ of $M$  such that:
\begin{enumerate}
    \item $\chi_n$ converges to the trivial character of $M$.
    \item $\mathrm{Ind}_P^G(\sigma \times \chi_n)$ is infinitesimally equivalent with $\pi_n$.
    \item $\pi$ is a composition factor of $\mathrm{Ind}_P^G(\sigma)$.
\end{enumerate}
Building on Theorems \ref{thm:topological property1} and \ref{thm:non-Hausdorff is topological}, we conjecture that the existence of torsion in $H^q(G;\bigoplus \pi_n)$ does not depend on the specific choice of the characters $\chi_n$, but only on the the parabolic $P$, the representation $\sigma$ and the set of limit points of $\pi_n$ (which is the set of composition factors of $\mathrm{Ind}_P^G(\sigma)$).
\begin{conjecture}\label{conj:torsionUltraTopological}
    The existence of torsion in cohomology in degree $q$ for the representation $\bigoplus \pi_n$ does not depend on the choice of the characters $\chi_n$ in the above construction.
\end{conjecture}
We finish with the following problem, which we could not resolve.
\begin{prob}\label{prob:remainingDetailed}
    Find an ``if and only if" statement regarding the existence of torsion in degree $q$ for a representation $\bigoplus \pi_n$ as in the above construction.
\end{prob}

\subsection{Proofs for \S\ref{subsec:isolation}}\label{subsec:isolation proofs}
We consider the following invariants for semisimple groups.
\begin{definition}
    Let $G$ be a semisimple group. 
    \begin{enumerate}
        \item The invariant $n(G)$ is the first degree such that there exists a unitary representation $V$ with $V^G=0$ and $H^{n(G)}(G;V) \ne 0$.
        \item The invariant $m(G)$ is the first degree such that there exists a unitary representation $V$ with $\torH^{m(G)}(G;V) \ne 0$.
    \end{enumerate}
\end{definition}
\begin{theorem}[Theorem \ref{intro: theorem isolation}]\label{thm:isolation}
    Let $G$ be a non-compact semisimple group. Then $n(G)=q$, where $q$ is the first degree in which there exists a non-trivial representation in $\widehat{G}_{q-\mathrm{cohom}}$, 
    and $m(G)=q+1$, where $p$ is the first degree in which there exists a representation in $\widehat{G}_{p-\mathrm{cohom}}$ which is not isolated in $\widehat{G}$.
\end{theorem}
\begin{proof}
     Assume that $G$ does not have property (T). The trivial representation, which is the unique element of $\widehat{G}_\mathrm{0-cohom}$, is non-isolated in $\widehat{G}$. By the results of Delorme \cite[Th\'eor\`em~V.1]{Delorme} and Guichardet \cite{Guichardet72}, this is equivalent to $n(G)=m(G)=1$.
    
    Assume $G$ has (T), and therefore the trivial representation is isolated and $n(G),m(G)>1$.  We need to show that if $n(G)=q$, then there exists a representation $\pi \in \widehat{G}_{q-\mathrm{cohom}}$. Let $V$ be a representation with $V^G=0$ and $H^q(G;V) \ne 0$. If $\bar{H}^q(G;V)\ne 0$, then by Theorem \ref{thm: reduced cohom} $V$ has an irreducible cohomological subrepresentation, which can not be trivial by our assumption. If $\torH^q(G;V) \ne 0$, the Theorem \ref{main1} implies the existence of a representation $\pi \in \widehat{G}_{q-\mathrm{cohom}}$ which is non-isolated in the support of $V$. But we assumed that the trivial representation is isolated, so $\pi$ is non-trivial.
    
    The second part, regarding $m(G)$, follows from Theorems \ref{main1} and \ref{main2}.
\end{proof}
The invariants $n(G)$ and $m(G)$ were discussed in \S\ref{subsec:isolation}. In this subsection we will provide proofs for the theorems stated there, often without repeating them. We begin with noting that $m(G)>n(G)>1$ if and only if $G$ has property (T), and that both invariants can be calculated from the simple factors of a group (Theorem \ref{intro: m>n>1}).
\begin{proof}[Proof of Theorem \ref{intro: m>n>1}]
    If $G$ does not have property (T), the conclusion follows from the Delorme-Guichardet Theorem, as in the proof of Theorem \ref{thm:isolation}. If $G$ has property (T), then $n(G)>1$. Obviously, $m(G) \geq n(G)$. In this case, Theorem \ref{main1} does not allow an equality $m(G)=n(G)$. Indeed, if $m(G)=n(G)$, by Theorem \ref{main1} we can find a non-isolated point in $\widehat{G}_{(n(G)-1)-\mathrm{cohom}}$, but this representation is not trivial by the property (T) assumption. Therefore $m(G)>n(G)>1$. The part regarding $n(G)$ and $m(G)$ being the minimum among the corresponding invariants for the factors follows from Theorem \ref{thm:isolation} and the product formula \ref{thm:product formula}.
\end{proof}

\begin{proof}[Proof of Theorem \ref{intro: p-adic}]
    Under some assumptions on the size of the residue field, this theorem is due to Dymara-Januszkiewicz \cite{Dymara-Januszkiewicz}. In light of Theorem \ref{main1}, however, it follows immediately from the classification of the cohomological dual of $p$-adic groups (Theorem \ref{thm:cass_padic}). This theorem gives $n(G)=r(G)$, and for $G$ simple of higher-rank there is no pair of representations that are cohomological in two successive degrees. This shows that torsion in cohomology can not exist in this case, and therefore $m(G)=\infty$. 
\end{proof}
\begin{proof}[Proof of Theorem \ref{intro:Lie-n}]
    By Theorem \ref{thm:isolation} it is enough to consider irreducible representations. Therefore the theorem follows from inspecting the table at \cite[Theorem B and Appendix A]{BaderSauer}.
\end{proof}
Having theorems \ref{intro: m>n>1}, \ref{intro: p-adic} and \ref{intro:Lie-n} we are left with studying the invariant $m(G)$ for simple Lie groups. An interesting phenomenon is that for simple Lie groups with property (T), $m(G)$ always grows to infinity with the dimension of the symmetric space. In light of Theorem \ref{thm:isolation}, calculating $m(G)$ is equivalent to finding non-isolated cohomological representations. Isolation of cohomological representations was considered by Vogan \cite{voganIsol}, Bergeron \cite{BergeronIsolation, BergeronLefschetz} and Bergeron-Clozel \cite[Section~5.4]{BergeronClozel}. To prove the growth of $m(G)$ goes to infinity with the dimension of the symmetric space, we consider the three doubly-parametrized families $\mathrm{SO}(p,q)$, $\mathrm{SU}(p,q)$ and $\mathrm{Sp}(p,q)$. The work of Bergeron allows us to calculate $m(G)$ explicitly for these.
\begin{theorem}[cf. {\cite[Corollary~4.3]{BergeronLefschetz}}]\label{thm:so(p,q)}
    For the group $G=\SO^\circ(p,q)$ with $p\geq q \geq 1$ we have $n(G)=r(G)=q$ and
    \[ m(G)= \begin{cases}
1 & q=1 \\
\lceil \frac{p}{2} \rceil+1 & q=2 \textit{ and }p>q \\
 p+q -2 & q \geq 3
\end{cases}\]
\end{theorem}
Note that we excluded the non-simple group $\SO^\circ(2,2)$.
\begin{theorem}[cf. {\cite[Corollary~5.4.2]{BergeronClozel}}]\label{thm:su(p,q)}
    For the group $G=\mathrm{SU}(p,q)$ with $p\geq q \geq 1$ we have $n(G)=r(G)=q$ and
    \[ m(G)= \begin{cases}
1 & q=1 \\
 p+q -1 & q \geq 2
\end{cases}\]
\end{theorem}
That fact that the bound given in \cite[Corollary~4.3]{BergeronLefschetz} and \cite[Corollary~5.4.2]{BergeronClozel} is sharp, and therefore indeed calculates $m(G)$, is explained in the remarks following these results. Bergeron does not give the explicit bound for the group $\mathrm{Sp}(p,q)$, but \cite[Corollarie~2.3]{BergeronIsolation} provides a combinatorial data that is enough in order to conclude the following.
\begin{theorem}\label{thm:sp(p,q)}
    For the group $G=\mathrm{Sp}(p,q)$ with $p\geq q \geq 1$ and $p \geq 2$ we have $n(G)=2r(G)=2q$ and
    \[ m(G)= \begin{cases}
p+1 & q=1 \\
 2p+2q-2 & q \geq 2
\end{cases}\]
\end{theorem}
The three doubly-parameterized families are the only ones for which the rank remains bounded and the dimension of the symmetric space grows to infinity. Combining the above theorems with the table in \cite[Appendix A]{BaderSauer} which calculates $n(G)$, and the fact that for groups with property (T) we have $n(G)<m(G)$, we can deduce the following:
\begin{theorem}[Theorem~\ref{intro:Lie-m}]\label{thm:lie-m}
    For every non-compact simple Lie group $G$ with property T,
    we have 
    \[ \sqrt{d(G)} \leq m(G) \leq \lceil d(G)/2 \rceil. \]
\end{theorem}
\begin{proof}
    The upper bound follows from inspecting tempered cohomological representations. First, assume that $G$ has does not have discrete series representations. Since $G$ has a uniform lattice,  Theorem \ref{thm:l2 invariants} provides a range of degrees around $d(G)/2$ where the left regular representation has non-zero torsion in cohomology. Second, assume that $G$ has discrete series representations. Again, Theorem \ref{thm:l2 invariants} implies the existence of a cohomological discrete series representation in degree $d(G) / 2$. This cohomological discrete series representation is isolated in the tempered dual, but it is never isolated in the entire unitary dual. This follows from \cite[Remark 2 on p.10]{BergeronIsolation}. Therefore, we have torsion in every representation that essentially weakly contains it in degree $d(G) / 2$ or $d(G) / 2 + 1$. But Poincar\'e duality for torsion, Theorem \ref{thm:Poincare duality}, shows that torsion in degrees $d(G) / 2$ and $d(G) / 2 + 1$ is equivalent. 

    For the lower bound, the three doubly parameterized families SO$(p,q)$,  SU$(p,q)$ and Sp$(p,q)$ are considered in Theorems \ref{thm:so(p,q)}, \ref{thm:su(p,q)} and \ref{thm:sp(p,q)} above. For any other simple Lie group, \cite[Appendix A]{BaderSauer} shows that $\sqrt{d(G)}< n(G) + 1$, so that $\sqrt{d(G)} \leq m(G)$. 
\end{proof}

\section{Some rank-1 cases}\label{sec6}
In this section, we focus on the cases of $G=\SO^\circ(n,1)$ and $G=\mathrm{SU}(n,1)$. The cohomological dual of these groups is understood very well, as well as the topology near the cohomological representations. Most of this can be deduced from Borel-Wallach \cite[Section~VI.4]{BorelWallach}. Note that Borel and Wallach work in the category of differentiable admissible representations, or in the category of $(\mathfrak{g},K)$-modules. This is not a problem, as the continuous cohomology coincides with the differential one (Lemma \ref{lem:smooth and cont cohom}), and with the cohomology of the $(\mathfrak{g}, K)$-module by van Est's theorem \cite[Corollary~IX.5.6.ii]{BorelWallach}. Furthermore, all of the admissible cohomological representations of $G$ are indeed unitary \footnote{This is far from being the case for other semisimple Lie groups.} \cite[Theorem~VI.4.12]{BorelWallach}. We will employ in this case our Theorems \ref{main1} and \ref{main2}. This will give a rich family of examples of representations that admit non-zero torsion in cohomology, as well as a counterexample for Question \ref{q: inverse main1}, namely a proof of Theorem \ref{thm:neg answer sl2C}.

\subsection{Unitary representations of rank-$1$ Lie groups}\label{subsec:6.1} First, we briefly recall some notions from representation theory of reductive Lie groups. Let $\mathbf{G}$ be a connected, almost simple algebraic $\mathbb{R}$-group, and $G=\mathbf{G}(\mathbb{R})$ the Lie group of $\mathbb{R}$-points of $\mathbf{G}$. We further assume that $\mathbf{G}$ is of $\R$-rank $1$. Fix a maximal compact subgroup of $G$ and denote it by $K$. Fix a minimal $\R$-parabolic subgroup $\mathbf{P}$ with Levi decomposition $\mathbf{P}=\mathbf{MN}$, where the Levi $\mathbf{M}$ is the centralizer of a split torus $\mathbf{A}$. We denote by $A_0$ the identity component of $\mathbf{A}(\R)$, by $M$ the Lie group $\mathbf{M}(\R)$, and by $^0M$ the group:
$$^0M=\bigcap_{\chi \in X(M)} \mathrm{ker}(\abs{\chi})$$
Where $X(M)$ is the group of continuous homomorphisms of $M$ into $\mathbb{R}^*$. Now, the Levi decomposition $P=MN$ further decomposes into a Langlands decomposition, $P={^0M} A_0 N$, as $M = {^0M} \times A_0$. We denote the Lie algebra of $A_0$ by $\mathfrak{a}$. Consider the the complexification of its dual, $\mathfrak{a}_\mathbb{C}^*$. Since $G$ is rank-$1$, we may identify $\mathfrak{a}_\mathbb{C}^*$ with the one dimensional vector space over $\C$, and consider elements $z \in \mathfrak{a}_\mathbb{C}^*$ as representations of $A_0$ in the obvious way, via the exponential map.

The \emph{Langlands classification} allows one to understand non-tempered representations of $G$ in terms of induction from tempered representations of Levi factors of parabolic subgroups. If $\sigma$ is an irreducible unitary representation of ${^0M}$ on a Hilbert space $H_\sigma$, and $z \in \mathbb{C}$, then $\sigma \times z$ is a representation of $M={^0M} \times A_0$. Note that unless $z \in i\R$, this representation is not unitary. We extend it trivially across $N$ to $P$, and denote by $I(\sigma, z)$ the normalized parabolic induction \cite[Definition~III.3.2]{BorelWallach}. If $z \in i\mathbb{R}$, then every irreducible component of $I(\sigma,z)$ is a tempered representation of $G$ which is not in the discrete series. These are usually called \emph{principal series} representations. If $z \in \R$, in which case we will replace the letter $z$ by $s$, $I(\sigma,s)$ might admit a unitary structure, and might yield a non-tempered unitary representation. If $s > 0$, $I(\sigma, s)$ admits a unique irreducible quotient \cite[Corollary~IV.4.6]{BorelWallach}. This quotient is called the \emph{Langlands quotient} of $(\sigma,s)$, and is denoted by $J(\sigma,s)$. The Langlands classification theorem states that every unitary, non-tempered representation of $G$ is realized in a unique way as a Langlands quotient $J(\sigma,s)$ for some $\sigma$ and $s > 0$ \cite[Theorem~IV.4.11]{BorelWallach}. Let $\sigma$ be an irreducible representation of ${^0M}$. The set
$$\{ J(\sigma,s) \mid s > 0, \quad J(\sigma,s) \text{ has a unitary structure}\}$$
Is called the \emph{complementary series} associated with $\sigma$.

To understand the topology on the complementary series, we further restrict to the case where $G$ is isomorphic to $\SO^\circ(n,1)$ or $\mathrm{SU}(n,1)$. In these cases, let $\sigma \in {^0M}$, and assume that $I(\sigma,0)$ is irreducible (note that if $I(\sigma,0)$ is reducible, then $J(\sigma,s)$ is never unitary for $s>0$ \cite[Corollary~4.3]{KnappStein}). Then, there exists a unique $s_0 > 0$ such that $I(\sigma,s_0)$ is reducible, and $I(\sigma,s)=J(\sigma,s)$ is irreducible for $0 < s< s_0$. The set of $s \in (0, \infty)$ such that $J(\sigma,s)$ is in the complementary series is precisely $(0,s_0]$ \cite[Corollary~4.2]{KnappStein}. The set $\{J(\sigma,s) \mid 0 \leq s < s_0\} \subseteq \widehat{G}$ is homeomorphic, with the Fell topology, to the interval $[0,s_0)$. As the parameter $s$ converges to $s_0$, the representations $J(\sigma,s)$ converge to all of the irreducible subquotients of $I(\sigma,s_0)$. This is a finite set of unitary representations. These subquotients are said to be \emph{the end of the complementary series associated with $\sigma$}, and $J(\sigma,s_0)$ is called \emph{the Langlands quotient at the end of the complementary series associated with $\sigma$}.
\subsection{The case of $\SO^\circ(2k,1)$}
Denote the group $\SO^\circ(2k,1)$ by $G$. $G$ is a rank-$1$ form of $\mathrm{B}_k$, and it has discrete series representations. The following theorem describes $\widehat{G}_\mathrm{cohom}$ and the topology around it. 
\begin{theorem}\label{thm:cohomological dual so(2k,1)}
    $\widehat{G}_\mathrm{cohom}$ consists of $k+2$ representations:
    \[
    \widehat{G}_\mathrm{cohom}=\{J_0, J_1, ..., J_{k-1},D_1,D_2\}
    \]
    The representations $D_1$ and $D_2$ are discrete series representations. The representations $J_i$, $0 \leq i \leq k-1$ are non-tempered. $J_0$ is the trivial representation.
    \begin{enumerate}
        \item  For $i \in \{1,2\}$, $H^q(G;D_i) \cong \mathbb{C}$ if $q=k$ and $H^q(G;D_i) = 0$ otherwise.
        \item $H^q(G;J_i) \cong \mathbb{C}$ if $q=i$ or $q=2k - i$. $H^q(G;J_i)=0$ otherwise.
        \item The representations $J_i$ are Langlands quotients $J(\sigma_i,s_i)$ at ends of certain complementary series. For $0 \leq i \leq k-2$, the kernel of the map $I(\sigma_i,s_i) \to J(\sigma_i,s_i)$ is $J_{i+1}$. For $J_{k-1}$, the kernel of the map $I(\sigma_{k-1},s_{k-1}) \to J(\sigma_{k-1},s_{k-1})$ is $D_1\oplus D_2$.
        \item Let $\pi_n$ be a sequence of of distinct representations in $\widehat{G}$ that converges to a cohomological representation $\pi \in \widehat{G}_{\mathrm{cohom}}$. Then, almost every $\pi_n$ belong to one of the complementary series whose Langlands quotient at the end is $J_i$ for $0 \leq i \leq k-1$.
    \end{enumerate}
\end{theorem}
\begin{proof}
    The identification of the representations, as well as items $(1)$ and $(2)$, is \cite[VI.4.4]{BorelWallach} and \cite[Theorem~4.5]{BorelWallach}. Item $(3)$ for $J_i$, $i \leq k-2$ follows from $(7)$ in \cite[VI.4.2]{BorelWallach} (see also \cite[VI.4.6]{BorelWallach}). For $J_{k-1}$, it follows from \cite[VI.1.7]{BorelWallach}.
    
    For item $(4)$, we use \cite[Theorem~3]{voganIsol}. It states that the only way to converge in the unitary dual is through parabolic induction. That is, in the rank-$1$ case, almost every $\pi_n$ belongs to a certain principal series or complementary series that converges to $\pi$ as well. Since discrete series representations are isolated in the tempered dual, and the only tempered cohomological representations are discrete, the only way to converge is through complementary series. Hence, we are left to show that there do not exist any other complementary series that can converge to a cohomological representation $\pi$. As explained in the discussion in the last paragraph of \S\ref{subsec:6.1}, every complementary series can be identified by the Langlands quotient that appears at the end of it. Hence, if the convergence happens along a complementary series other than the ones mentioned above, it must be a complementary series whose Langlands quotient at the end is some non-cohomological representation $\sigma$, and $\pi$ is some other subquotient in the same parabolically induced representation. Now, it follows from \cite[Theorems~VI.3.2-VI.3.4]{BorelWallach} that $\widehat{G}_\mathrm{cohom}$ is exactly the set of representations that have the same infinitesimal character as the trivial representation, and a trivial central character. But if $\sigma$ and $\pi$ are subquotients of the same parabolically induced representation, they share the infinitesimal character \cite[III.3.2]{BorelWallach}, so such a $\sigma$ does not exist.
\end{proof}

We will describe the topology around $\widehat{G}_{\mathrm{cohom}}$ pictorially. The topology of $\widehat{G}$ around $\widehat{G}_{\mathrm{cohom}}$ is described in Figure \ref{fig:cohom_reps_so(2k,1)}, with the following conventions:
\begin{enumerate}
    \item continuous lines represents areas in $\widehat{G}$ where the topology is that of an interval, for example, a complementary series without its ends.
    \item dotted lines represent the situation where several non-Hausdorff points of $\widehat{G}$ are "glued" at the end of a line. 
\end{enumerate}
The prototypical situation is when a complementary series converges to a finite set of several representations that lie at its end. Different complementary series might share some (usually not all) of their endpoints. 
\begin{figure}[htbp]
    \centering
    \begin{tikzpicture}
        \draw (0,0) -- (2,0);
        \draw[dotted] (2,0) -- (2.5, 0.5) node[anchor=west] {$J_1$};
        \draw[dotted] (2,0) -- (2.5, -0.5) node[anchor=west] {$J_0$};
        \draw (0,1) -- (2,1);
        \draw[dotted] (2,1) -- (2.5, 1.5) node[anchor=west] {$J_2$};
        \draw[dotted] (2,1) -- (2.5, 0.5);
        \draw[dashed] (1, 2) -- (1, 3);
        \draw (0,4) -- (2,4);
        \draw[dotted] (2,4) -- (2.5, 3.5) node[anchor=west] {$J_{k-1}$};
        \draw[dotted] (2,4) -- (2.5, 4) node[anchor=west] {$D_1$};
        \draw[dotted] (2,4) -- (2.5, 4.5) node[anchor=west] {$D_2$};
    \end{tikzpicture}
    \caption{The cohomological representations of $\SO^\circ(2k+1,1)$.}
    \label{fig:cohom_reps_so(2k,1)}
\end{figure}
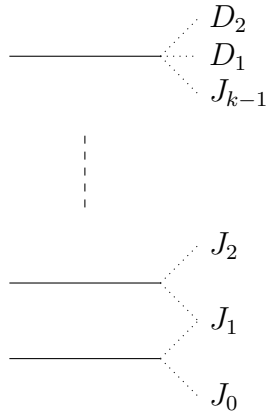

Let $\sigma$ be a representation of ${^0M}$. In the notations of \S\ref{subsec:6.1}, suppose that the complementary series associated with it consists precisely of the Langlands quotients $J(\sigma,s)$, for $0<s \leq s_0$. We will use the following terminology throughout this section and the next one.
\begin{terminology}\label{term:supp on comp}
    A unitary representation $\pi$ of $G$ is said to be \emph{supported on the end} of the complementary series associated with $\sigma$ if:
    \begin{enumerate}
        \item $\mathrm{supp}(\pi)$ is contained in the complementary series $\{J(\sigma,s) \mid 0<s<s_0\}$ together with its end points.
        \item $\pi$ does not contain an irreducible cohomological subrepresentation.
        \item $\mathrm{supp}(\pi)$ contains a sequence of representations $J(\sigma,s_i)$ such that $s_i \rightarrow s_0$.    
    \end{enumerate}
\end{terminology}
With this terminology at hand, the following theorem follows from Theorem \ref{thm:cohomological dual so(2k,1)}, Theorem \ref{main1} and Theorem \ref{main2}.
It provides a full solution of Problem~\ref{prob:main} for $G=\SO^\circ(2k,1)$.
\begin{theorem}\label{thm:non-Hausdorff for so(2k,1)}
    Let $\pi$ be a unitary representation of $G$. $\pi$ has non-zero torsion cohomology in degree $q$ for some $1 \leq q \leq k$ if and only if it contains a representation supported on the end of the complementary series whose Langlands quotient at the end is $J_{q-1}$. It has non-zero torsion cohomology in degree $q$ for some $k+1 \leq q \leq 2k$ if and only if it contains a representation supported on the end of the complementary series whose Langlands quotient at the end is $J_{2k-q}$.
\end{theorem}

\subsection{The case of $SO^\circ(2k+1,1)$}
Denote the group $\SO^\circ(2k+1,1)$ by $G$. $G$ is a rank-$1$ form of $\mathrm{D}_{k+1}$, and it does not have discrete series representations. The following theorem describes $\widehat{G}_\mathrm{cohom}$ and the topology around it. 
\begin{theorem}\label{thm:cohomological dual of SO(2k+1,1)}
    $\widehat{G}_\mathrm{cohom}$ consists of $k+1$ representations:
    \[
    \widehat{G}_\mathrm{cohom}=\{J_0, J_1, ..., J_{k-1},I_k,\}
    \]
    The representation $I_k$ is tempered, and is not a discrete series representation. The representations $J_i$, $0 \leq i \leq k-1$ are non-tempered. $J_0$ is the trivial representation.
    \begin{enumerate}
        \item $H^q(G;I_k) \cong \mathbb{C}$ if $q=k,k+1$ and $H^q(G;I_k) = 0$ otherwise.
        \item $H^q(G;J_i) \cong \mathbb{C}$ if $q=i$ or $q=2k+1 - i$. $H^q(G;J_i)=0$ otherwise.
        \item The representations $J_i$ are Langlands quotients $J(\sigma_i,s_i)$ at ends of certain complementary series. For $0 \leq i \leq k-2$, the kernel of the quotient $I(\sigma_i,s_i) \to J_i$ is $J_{i+1}$. For $J_{k-1}$, the kernel of the quotient $I(\sigma_{k-1},s_{k-1}) \to J_{k-1}$ is $I_k$.
        \item Let $\pi_n$ be a sequence of distinct representations in $\widehat{G}$ that converges to a cohomological representation $\pi \in \widehat{G}_{\mathrm{cohom}}$. Then, almost every $\pi_n$ belong to one of the complementary series whose Langlands quotient at the end is $J_i$ for $0 \leq i \leq k-1$, or is tempered and belong to a certain principal series of which $I_k$ is the origin. 
    \end{enumerate}
\end{theorem}
\begin{proof}
    The description of the representations, as well as items $(1)$ and $(2)$, is \cite[Theorem~VI.4.2]{BorelWallach}. Item $(3)$ follows from $(7)$ in \cite[VI.4.2]{BorelWallach} for $i < k-1$, and for $i=k-1$ from \cite[VI.3.8]{BorelWallach}. The proof of item $(4)$ is precisely the same as in the case of $\SO^\circ(2k,1)$. The only difference is that in this case it is possible to converge to $I_k$ from the tempered dual, that is, from principal series. 
\end{proof}
The picture of the dual in this case is described in Figure \ref{fig:cohom_reps_so(2k+1,1)}, where the vertical line starting at $I_k$ is a principal series, and all other lines are complementary series. The conventions are the same as the ones stated for $\SO^\circ(2k,1)$.
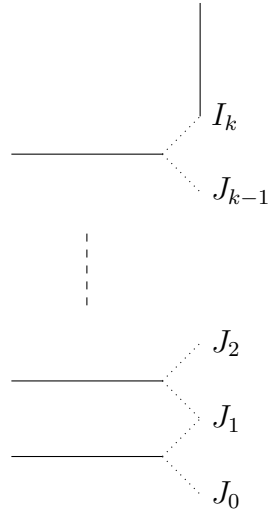
\begin{figure}[htbp]
    \centering
    \begin{tikzpicture}
        \draw (0,0) -- (2,0);
        \draw[dotted] (2,0) -- (2.5, 0.5) node[anchor=west] {$J_1$};
        \draw[dotted] (2,0) -- (2.5, -0.5) node[anchor=west] {$J_0$};
        \draw (0,1) -- (2,1);
        \draw[dotted] (2,1) -- (2.5, 1.5) node[anchor=west] {$J_2$};
        \draw[dotted] (2,1) -- (2.5, 0.5);
        \draw[dashed] (1, 2) -- (1, 3);
        \draw (0,4) -- (2,4);
        \draw[dotted] (2,4) -- (2.5, 3.5) node[anchor=west] {$J_{k-1}$};
        \draw[dotted] (2,4) -- (2.5, 4.5) node[anchor=west] {$I_k$};
        \draw(2.5, 4.5) -- (2.5, 6);
    \end{tikzpicture}
    \caption{The cohomological representations of $\SO^\circ(2k,1)$.}
    \label{fig:cohom_reps_so(2k+1,1)}
\end{figure}
The following theorem states that Problem \ref{prob:main} is fully resolved for $\SO^\circ(2k+1,1)$ (and thus for $SO(n,1)$ in general, combining with Theorem \ref{thm:non-Hausdorff for so(2k,1)}). We use the same terminology (Terminology \ref{term:supp on comp}) as in Theorem \ref{thm:non-Hausdorff for so(2k,1)}. 
\begin{theorem}\label{thm:non-Hausdorff for so(2k+1,1)}
     Let $\pi$ be a unitary representation of $G$.
     \begin{enumerate}
         \item $\pi$ has non-zero torsion in cohomology in degree $q$ for some $1 \leq q \leq k$ if and only if it contains a representation supported on the end of the complementary series whose Langlands quotient at the end is $J_{q-1}$.
         \item $\pi$ has non-zero torsion in cohomology in degree $q$ for some $k+2 \leq q \leq 2k+1$ if and only if it contains a representation supported on the end of the complementary series whose Langlands quotient at the end is $J_{2k+1-q}$.
         \item $\pi$ has non-zero torsion in cohomology in degree $k+1$ if and only if it weakly contains a sequence of (distinct) principal series representations converging to $I_k$.
     \end{enumerate}
\end{theorem}
Part $(1)$ and $(2)$ in Theorem \ref{thm:non-Hausdorff for so(2k+1,1)} from from Theorems \ref{thm:cohomological dual of SO(2k+1,1)}, \ref{main1} and \ref{main2}. Part $(3)$ is more involved. To emphasize the situation that is not fully resolved, let $\pi$ be a representation that is supported on the end of the complementary series that converges to $J_{k-1}$ and $I_k$. Then, $\pi$ has torsion cohomology in degrees $k$ and $k+2$. Its cohomology in degrees different than $k$, $k+2$ and $k+1$ vanishes. We are left to understand whether or not its cohomology in degree $k+1$ vanishes. This is the content of part $(3)$ of Theorem \ref{thm:non-Hausdorff for so(2k+1,1)}, and we single it out as a separate theorem.
\begin{theorem}\label{greatQuestion}
    Let $\pi$ be a representation supported on the end of the complementary series that converges to $J_{k-1}$ and $I_k$. Then $H^{k+1}(G;\pi) = 0$. Thus, part $(3)$ of Theorem \ref{thm:non-Hausdorff for so(2k+1,1)} holds.
\end{theorem}
In the next subsection we will prove Theorem \ref{greatQuestion}, using an approach different than the one we we have taken up until now. This will also yield a proof of Theorem \ref{thm:neg answer sl2C}.
\subsection{Proof of Theorem \ref{greatQuestion}}\label{subsec:6.4}
To prove Theorem \ref{greatQuestion}, we use methods from relative Lie algebra cohomology. We start with a general connected, algebraic semisimple Lie group $G=\mathbf{G}(\R)$. Let $(\sigma,V)$ be a unitary representation of $G$. The continuous cohomology  $H^*(G;V)$ is isomorphic to the smooth cohomology $H^*_d(G;V^\infty)$ (Lemma \ref{lem:smooth and cont cohom}), and the latter is isomorphic to the relative Lie algebra cohomology $H^*(\mathfrak{g},\mathfrak{k},V^\infty)$ by the van-Est theorem \cite[Theorem~3.2]{Kyed}, where $\mathfrak{g}$ denotes the Lie algebra of $G$ and $\mathfrak{k}$ the Lie algebra of a fixed maximal compact subgroup $K$. We will not recall the definition of relative Lie algebra cohomology; the reader is referred to \cite[Chapter~I]{BorelWallach} for an introduction to this theory. We will use the standard complex which calculates this cohomology, namely, the \textit{Chevalley-Eilenberg} complex. The $q$'th cochain space is $\mathrm{Hom}_\mathfrak{k}(\Lambda^q (\mathfrak{g}/\mathfrak{k}), V^{\infty})$. The  differential $d_q: \mathrm{Hom}_\mathfrak{k}(\Lambda^{q-1} (\mathfrak{g}/\mathfrak{k}), V^{\infty}) \to \mathrm{Hom}_\mathfrak{k}(\Lambda^q (\mathfrak{g}/\mathfrak{k}), V^{\infty})$ is given by:

\begin{equation} \label{Chevalley-Eilenberg}
\begin{split}
    df(x_1, \dots, x_{q}) &= \sum_{i=1}^{q} (-1)^{i+1} x_i \cdot f(x_1, \dots, \hat{x}_i, \dots, x_q) \\
    &\quad + \sum_{1 \le i < j \le q} (-1)^{i+j} f([x_i, x_j], x_1, \dots, \hat{x}_i, \dots, \hat{x}_j, \dots, x_q)
\end{split}
\end{equation}

\begin{remark}
    Both passing to the smooth vectors (Lemma \ref{lem:smooth and cont cohom}) and to the relative Lie algebra cohomology (the van-Est theorem \cite[Theorem~3.2]{Kyed}) do not require the representation $V$ to be irreducible, or $V^\infty$ to be admissible. But, in case $V^\infty$ is admissible, the finite multiplicity of the $K$-types implies that the cochain spaces of the Chevalley-Eilenberg complex are finite dimensional.
\end{remark}
We will need the following immediate lemma that relates cohomology of irreducible representations to the that of a direct integral.
\begin{lemma}\label{lem:d=0 integral}
    Let $(V,\sigma)$ be a unitary representation of $G$, and let $V = \int^\oplus V_t d\mu(t)$ be a direct integral decomposition such that almost every constituent $(\sigma_t,V_t)$ is irreducible. Assume that for $\mu$-almost every $t$ the space of $q$-cocycles in the Chevalley-Eilenberg complex for $V_t^\infty$ is zero. Then $H^q(G;V)=0$.
\end{lemma}
\begin{proof}
    Let $v \in V^\infty$ be a smooth vector. By \cite[Lemma~2]{Arnal76}, $v$ is almost everywhere smooth, i.e it admits a unique decomposition $v = \int^\oplus v_t d\mu(t)$ with $v_t \in V_t^\infty$ almost everywhere. Furthermore, if $X \in \mathfrak{g}$, then $Xv = \int^\oplus Xv_t d\mu(t)$. Given $f \in \mathrm{Hom}_\mathfrak{k}(\Lambda^q (\mathfrak{g}/\mathfrak{k}), V^{\infty})$ we can decompose $f(x_1,...x_q) = \int^\oplus f_t(x_1,...x_q) d\mu(t)$. It follows from Formula \ref{Chevalley-Eilenberg}
    that $d_{q+1}f = \int^\oplus d_{q+1}f_t d\mu(t)$ so $f$ is a cocycle if and only if $f_t$ is a cocycle for $\mu$-almost every $t$. Hence the space of cocycles is zero also in the Chevalley-Eilenberg complex for $V^\infty$, and therefore $H^q(G;V)=0$.
\end{proof}
We now shift our attention to the specific case of $G = \SO^\circ(2k+1,1)$. In this case the maximal compact subgroup $K$ is $\SO(2k+1)$. We use the notations for Langlands classification given in the beginning of \S\ref{sec6}. In our case ${^0}M = SO(2k) \subset K$, and we are interested in a representation $(V,\pi)$ supported on the end of the complementary series whose Langlands quotient at the end is $J_{k-1}$. By \cite[Subsection~VI.4.2]{BorelWallach}, the representation $\sigma_{k-1}$ of ${^0}M$ giving the relevant complementary series is $\sigma_{k-1} \cong \Lambda^{k-1}\C^{2k}$. Fix a direct integral decomposition of $(\pi,V)$:

$$V = \int V_t d\mu(t).$$

For almost every $t$, $V_t$ is isomorphic to a representation of the form $I(\sigma_{k-1},s)$ on the complementary series

We will need the following lemma on the dimensions of the cochain spaces for representations on our complementary series. We fix a representation of $V_s$ on the complementary series of $\sigma_{k-1}$. Namely, $V_s=I(\sigma_{k-1},s)$ for some $s$. Denote by $C_q$ the space of $q$-cochains in the Chevalley-Eilenberg complex for the representation $V_s$. That is, $C_q=\mathrm{Hom}_\mathfrak{k}(\Lambda^q \mathfrak({g}/\mathfrak{k}), V_s^\infty)$

\begin{lemma} \label{lem:one dimensional}
    We have that $\mathrm{dim}(C_q)=1$ for $k-1 \leq q \leq k+2$, and $\mathrm{dim}(C_q)=0$ otherwise.
\end{lemma}

\begin{proof}
The representation $I(\sigma_{k-1},s)$ is by definition $\Ind_P^G( \sigma_{k-1}\times s)$. The Iwasawa decomposition $G =KP$, together with the fact that and ${^0}M = K \cap P$, implies by Mackey's theorem on induction and restriction that $$\Res^G_K\Ind_P^G(\sigma_{k-1}\times s) \cong \Ind_{{^0}M}^K(\sigma_{k-1}) = \Ind_{^0M}^K(\Lambda^{k-1}\C^{2k}).$$
Note that in our case, $\mathfrak{g}/\mathfrak{k} \cong \R^{2k+1}$ as $K$-representations. We use Frobenius reciprocity, extend scalars to $\C$ and pass from the Lie algebra $\mathfrak t$ to the Lie group $K$:
\begin{align*}
\mathrm{Hom}_\mathfrak{k}(\Lambda^q (\mathfrak{g}/\mathfrak{k}), I(\sigma_{k-1},s)) 
&\cong \mathrm{Hom}_K(\Lambda^{q}(\R^{2k+1}), \Ind_{^0M}^K \Lambda^{k-1}(\C^{2k})) \\
&\cong \mathrm{Hom}_K(\Lambda^{q}(\C^{2k+1}), \Ind_{^0M}^K \Lambda^{k-1}(\C^{2k})) \\
&\cong \mathrm{Hom}_{^0M}(\Res_{^0M}^K \Lambda^{q}(\C^{2k+1}), \Lambda^{k-1}(\C^{2k}))
\end{align*}
As a ${^0}M$ representation, $\C^{2k+1} \cong \C^{2k}\oplus \C$ where the first summand is the standard representation and the second is trivial. Therefore, as $^0M$ representations:
$$\Lambda^q(\C^{2k+1}) \cong \Lambda^q(\C^{2k})\oplus\Lambda^{q-1}(\C^{2k})$$
and all in all we have
$$\mathrm{Hom}_K(\Lambda^q (\mathfrak{g}/\mathfrak{k}), I(\sigma_{k-1},s)) \cong \mathrm{Hom}_{^0M}(\Lambda^q(\C^{2k})\oplus\Lambda^{q-1}(\C^{2k}), \Lambda^{k-1}(\C^{2k}))$$
It is well known that as a $^0M =\SO(2k)$ representation, $\Lambda^q(\C^{2k})$ is irreducible for all $q \neq k$, and that $\Lambda^k(\C^{2k})$ and splits into two irreducible components which are not isomorphic to any other $\Lambda^{q}(\C^{2k})$. Moreover $\Lambda^q(\C^{2k}) \cong \Lambda^{q'}(\C^{2k})$ if and only if $q=q'$ or $q=2k-q'$. The claim now follows from Schur's Lemma.
\end{proof}

\begin{proof}[Proof of Theorem \ref{greatQuestion}]
    Since the representation $I(\sigma_{k-1},s)$ is not cohomological, its Chevalley-Eilenberg complex must be exact. By Lemma \ref{lem:one dimensional} the complex is non-zero only in degrees $k-1,k,k+1$ and $k+2$ where it is one dimensional. It is easily seen that in such an exact sequence the differential $d_{k+2}$ connecting degrees $k+1$ and $k+2$ must be a bijection, and in particular the space of $k+1$-cocycles is zero. So, in the decomposition $V=\int^\oplus V_t d\mu(t)$, for almost every $t$ the space of $k+1$-cocycles in the corresponding Chevalley-Eilenberg complex is zero. The result now follows from Lemma \ref{lem:d=0 integral}.
\end{proof}
\begin{proof}[proof of Theorem \ref{thm:neg answer sl2C}]
    This follows from Theorem Theorem \ref{thm:cohomological dual of SO(2k+1,1)} and Theorem \ref{greatQuestion}.
\end{proof}
\subsection{The case of SU(n,1)}
Similarly to $\SO^\circ(n,1)$, the cohomological dual of $G=\mathrm{SU}(n,1)$ is understood very well. The representation theoretic results described in \S\ref{subsec:6.1} apply to this case as well, and the classification of the unitary representations is mostly due to Kraljevi\'c \cite{Krlajevic}. The description of the cohomological representations given here is taken, as in the case of $\SO^\circ(n,1)$, mostly from Borel-Wallach \cite[Section~VI.3-VI.4]{BorelWallach}.

\begin{theorem}\label{thm:cohomological dual of su(n,1)}
    Let $G=\mathrm{SU}(n,1)$. The cohomological dual of $G$ is:
    \[
    \widehat{G}_\mathrm{cohom}=\{J_{i,j} \mid 0 \leq i+j \leq n-1 \} \cup \{D_i \mid 0 \leq i \leq n\}
    \]
    The representations $J_{i,j}$ are non-tempered, and the representation $D_i$ are discrete series representations. $J_{0,0}$ is the trivial representation.
    \begin{enumerate}
        \item $H^q(G;J_{i,j}) \cong \mathbb{C}$ if $q=i+j+2l$, $0 \leq l \leq n-i-j$ and $H^q(G;J_{i,j})=0$ otherwise.
        \item $H^q(G;D_i) \cong \C$ if $q=n$, and $H^q(G;D_i)=0$ otherwise.
        \item If $i+j < n-2$, then $J_{i,j}$ is a Langlands quotient appearing at the end of a certain complementary series. The other endpoints of this complementary series are $J_{i+1,j}$, $J_{i,j+1}$ and $J_{i+1,j+1}$. 
        \item If $i+j = n-2$, then $J_{i,j}$ is a Langlands quotient appearing at the end of a certain complementary series. The other endpoints of this complementary series are $J_{i+1,j}$, $J_{i,j+1}$ and $D_{i+1}$.
        \item If $i+j = n-1$, then $J_{i,j}$ is a Langlands quotient appearing at the end of a certain complementary series. The other endpoints of this complementary series are $D_{i+1}$ and $D_{i}$. 
        \item If $\pi_n$ is a sequence of representations of $G$ that converges to a cohomological representation, then almost all of the $\pi_n$ belong to one of the above-mentioned complementary series (namely, a complementary series whose Langlands quotient at the end is $J_{i,j}$ for some $0 \leq i+j \leq n-1$)
    \end{enumerate}
\end{theorem}
\begin{proof}
    The identification of the representations, as well as $(1)$ and $(2)$, are \cite[Theorem~VI.4.11]{BorelWallach}, and items $(3)-(5)$ follow from the proof of \cite[Theorem~VI.4.11]{BorelWallach}. Item $(6)$ is proved in the same way as the analogues result for $\SO^\circ(n,1)$.
\end{proof}
Theorem \ref{main1} and Theorem \ref{thm:cohomological dual of su(n,1)} imply the following:
\begin{corollary}
    Let $\pi$ be a representation of $G$ such that $\torH^q(G;\pi) \ne 0$ for some $q > 0$. Then, $\pi$ contains a subrepresentation supported on the end of a complementary series whose Langlands quotient at the end is one of the representations $J_{i,j}$.
\end{corollary}
Theorem \ref{main2} imply that every representation that contains a subrepresentation supported on the end of a complementary series admits non-zero torsion in cohomology. In fact, if a representations $\pi$ is supported on the end of a complementary series whose Langlands quotient at the end is $J_{i,j}$, then for every $i+j+1 \leq q \leq 2n-i-j$, $\pi$ admits torsion cohomology either in degree $q$ or in degree $q+1$. However, the situation of ``four successive degrees'' discussed in \S\ref{subsec:5.2} often occurs, and we can not usually determine whether torsion occurs both in degree $q$ and $q+1$. For the case $\mathrm{SU}(2,1)$ we are able to provide a complete answer,
thus solving Problem~\ref{prob:main}.
\begin{theorem}\label{thm:non-Hausdorff-su21}
    Let $G=\mathrm{SU}(2,1)$. In the notation of Theorem \ref{thm:cohomological dual of su(n,1)}, there are three relevant complementary series -- namely, the ones whose Langlands quotients at the end is $J_{0,0}$, $J_{1,0}$ and $J_{0,1}$. Let $\pi$ be a representation that contains no irreducible cohomological subrepresentations, so that $\bar{H}^*(G;\pi)=0$
    \begin{enumerate}
        \item If a representation $\pi$ is supported at the end of the complementary series whose Langlands quotient at the end is $J_{0,1}$ or $J_{1,0}$, then $\torH^2(G;\pi) \ne 0$ and $\torH^3(G;\pi) \ne 0$, and $H^q(G;\pi)=0$ for $q \ne 2,3$.
        \item If a representation $\pi$ is supported on the end of the complementary series whose Langlands quotient at the end is $J_{0,0}$, then $\torH^q(G;\pi) \ne 0$ for $1 \leq q \leq 4$, and $H^q(G;\pi)=0$ for $q=0$ or $q>4$.
    \end{enumerate}
\end{theorem}
\begin{proof}
    Denote the set of cohomological representations weakly contained in $\pi$ by $A$. $\pi$ was built such that these cohomological representations are not isolated and inseparable by neighborhoods in $\mathrm{supp}(\pi)$. In case $(1)$, the representations in $A$ admit cohomology precisely in degrees $1,2$ and $3$, so that Theorems \ref{main1} and \ref{main2} immediately imply the result. In case $(2)$, $A$ also contains representations that are cohomological in degrees $0$ and $4$. Theorems \ref{main1} and \ref{main2} again imply that the torsion cohomology is not zero in degrees $1$ and $4$. As $\pi$ weakly contains the representation $J_{0,0}$ which is cohomological in degree $2$, it must admit torsion cohomology in degree $2$ or $3$ by Theorem \ref{main2}. But, $\torH^2(G;\pi) \ne 0$ if and only if $\torH^3(G;\pi) \ne 0$ by Poincar\'e duality \ref{thm:Poincare duality}. So both are non-zero.
\end{proof}
\section{Another look at $\PSL_2(\mathbb{C})$}\label{sec7}
In \S\ref{sec6} we completely resolved the question of unitary cohomology for the groups $\SO^\circ(n,1)$. Namely, we answered Problem \ref{prob:main} for these groups. For the case $n=2k$, our Theorems \ref{main1} and \ref{main2} were enough in order to deduce everything. In the case $n=2k+1$, however, our general theorems were not sufficient to determine the degrees of cohomology for a representation supported on the end of a certain complementary series, and we used machinery of relative Lie algebra cohomology to resolve the situation in \S\ref{subsec:6.4}. The cotent of this resolution is Theorem \ref{greatQuestion}. In this section we give a different proof for Theorem \ref{greatQuestion} for the case of $\SO^\circ(3,1)$, that shows a surprising connection to the theory of $3$-manifolds.

We fix $G=\PSL_2(\C)\cong\SO^\circ(3,1)$. The group $G$ has only two cohomological representations. In the notation of \S\ref{sec6} these are:
\begin{enumerate}
    \item $J = J_0 = \C$, the trivial representation which is cohomological in degrees $0$ and $3$.
    \item $I = I_1$, a tempered representation which is cohomological in degrees $1$ and $2$.
\end{enumerate}
The representations $I$ and $J$ both lie at the end of the (unique) complementary series of $G$, the complementary series associated with the trivial representation $1$. In addition, $I$ lies at the beginning of one of the principal series of $G$. The topology around these two representations is sketched in Figure \ref{fig:cohom_reps_SL2C}.
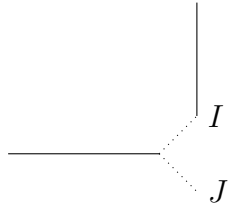
\begin{figure}[htbp]
    \centering
    \begin{tikzpicture}
        \draw (0,0) -- (2,0);
        \draw[dotted] (2,0) -- (2.5, 0.5) node[anchor=west] {$I$};
        \draw[dotted] (2,0) -- (2.5, -0.5) node[anchor=west] {$J$};
        \draw(2.5, 0.5) -- (2.5, 2);
    \end{tikzpicture}
    \caption{The cohomological representations of $\SL_2(\mathbb{C})$.}
    \label{fig:cohom_reps_SL2C}
\end{figure}
For the proof we present in this section, we will construct a unitary $G$-representation~$\pi$ of $G$ which weakly contains the trivial representation and whose second cohomology vanishes. In \S\ref{subsec:7.2} we show that the support of~$\pi$ contains an open neighborhood of the trivial representation. Therefore if $\sigma$ is supported on the end of the complementary series, then we may assume without loss of generality (after maybe passing to a subrepresentation supported on a smaller end) that $\pi$ weakly contains it. It then follows from Proposition \ref{thm:non-Hausdorff is topological} that any representation supported on the complementary series has vanishing second cohomology, and this completes the second proof of Theorem \ref{greatQuestion} for $G$.
\subsection{Constructing the unitary representation}\label{subsec:constructing the unitary rep}

Let $S_g$ be a closed surface of genus~$g \geq 2$ and $N = \pi_1(S_g)$. Let $\phi\colon S_g\to S_g$ be a pseudo-Anosov diffeomorphism, and let $M$ be the mapping torus of $\phi$, that is, the quotient of $S_g\times [0,1]$ by the equivalence relation $(s,0)\sim (\phi(s),1)$ for $s\in S_g$.
The fundamental group of $M$ is \[\Gamma\coloneqq\pi_1(M) = \Z \ltimes N=\langle t\rangle\ltimes N,\] where the generator $t$ of $\Z=\langle t\rangle$ acts on $N$ via the map induced by~$\phi$. By a celebrated theorem of Thurston \cite[Theorem~13.4]{Farb-Margalit} and because $\phi$ is pseudo-Anosov, $M$ admits a hyperbolic structure, which is unique by Mostow's rigidity theorem. Therefore, $\Gamma$ embeds as a uniform lattice into \[G=\PSL_2(\C) \cong \mathrm{Isom}^\circ(\mathbb{H}^3).\]

The intersection of closed curves on $S_g$ induces a natural symplectic form on the first homology $H_1(S_g;\C)\cong \C^{2g}$, which is preserved by $H_1(\phi)$. Via $H_1(\phi)$ the vector space $H_1(S_g;\C)$ becomes a (torsion) $\C[t]$-module. 

By \cite[Theorem~2]{PapadopoulosPseudoAnosov} and the surjectivity of the action of the mapping class group on homology onto $\mathrm{Sp}_{2g}(\Z)$ \cite[Theorem~6.4]{Farb-Margalit}, every symplectic automorphism of $H_1(S_g;\Z)$ can be realized as $H_1(\phi)$ for some pseudo-Anosov diffeomorphism $\phi\colon S_g\to S_g$. Thus, we may and will assume that, for the pseudo-Anosov diffeomorphism~$\phi$ defining the lattice~$\Gamma$ above, the induced automorphism $H_1(\phi;\C)$ is diagonalisable and its eigenvalues $\lambda_1,\dots, \lambda_{2g}$ do not lie on the unit circle. 
We have an isomorphism of $\C[t]$-modules 
\[ H_1(S_g;\C)\cong \bigoplus_{i=1}^{2g} \C[t]/(t-\lambda_i).\]

\begin{definition}\label{def:rep pi}
The unitary $G$-representation~$\pi$ is defined as the unitary induction of the (permutation) representation~$\sigma$ of $\Gamma$ on $\ell^2(\Gamma/N)=\ell^2(\langle t\rangle)$. So~$\pi$ is the quasi-regular $G$-representation on~$V_\pi\coloneqq L^2(G/N)$. 
\end{definition}

\begin{proposition}\label{prop: H2pi=0}
 $H^2(G;V_\pi)=0$.
\end{proposition}

\begin{proof}
    The Shapiro lemma (Lemma \ref{lem: shapiro lemma}) implies that $H^2(G;V_\pi) \cong H^2(\Gamma;\ell^2(\langle t\rangle))$. To show vanishing of the latter we apply the Hochschild-Serre spectral sequence \cite[Theorem~9.1]{Blanc} with respect to the extension $N \hookrightarrow \Gamma \twoheadrightarrow \langle t \rangle$. It suffices to show that $E_2^{p,q}= H^p\bigl(\langle t \rangle; H^q(N;\ell^2(\langle t\rangle))\bigr)=0$ for every $p,q\ge 0$ with $p+q=2$. 

    We have $E_2^{2,0}=0$ as $\langle t\rangle$ has cohomological dimension~$1$. 
    Next consider $E_2^{0,2}$. 
    Since $N$ acts trivially on $\ell^2(\langle t\rangle)$ and $t$ acts trivially on~$H^2(N;\C)=H^2(S_g;\C)$,
    there is an $\C[t]$-isomorphism $H^2\bigl(N;\ell^2(\langle t\rangle)\bigr) \cong H^2(N;\C)\otimes \ell^2(\langle t\rangle)\cong \ell^2(\langle t\rangle)$.
    The latter has no $t$-invariants, so $E_2^{0,2}=0$. 
    Finally, consider $E_2^{1,1}$, which is by duality of the Poincar\'e duality groups~$\langle t\rangle$ and $N=\pi_1(S_g)$ just the $t$-coinvariants of
    \[ H^1\bigl( N, \ell^2(\langle t\rangle)\bigr)\cong H^1\bigl( N;\C\bigr)\otimes \ell^2(\langle t\rangle)\cong H_1(S_g;\C)\otimes \ell^2(\langle t\rangle).\]
    These coinvariants vanish 
    \[ H_1(S_g;\C)\otimes_{\C[t]} \ell^2(\langle t\rangle)\cong \bigoplus_{i=1}^{2g} \C[t]/(t-\lambda_i)\otimes_{\C[t]}\ell^2(\langle t\rangle)=0\]
    because each $\lambda_i\not\in S^1$ and so $t-\lambda_i$ is invertible as a bounded operator on $\ell^2(\langle t\rangle)$.
\end{proof}

\subsection{Determining the support}\label{subsec:7.2}

In this subsection we will show that the support of the representation $\pi$ contains a full neighborhood of the trivial representation. We retain the n  otation from \S\ref{subsec:constructing the unitary rep}.

For $\theta \in \R$, denote by $\chi_\theta$ the character of $\Gamma$ which is pulled back from the character $t^{k} \mapsto e^{2\pi \theta k i}$ of the quotient $\Z \cong \langle t \rangle$. We begin by studying the representations $(W_\theta,\rho_\theta)$ of $G$ defined by $\rho_\theta = \Ind^G_\Gamma \chi_\theta$.

We will work with a concrete model for the induction. Since $\Gamma$ is cocompact in $G$, there exists a measurable section $s\colon G/\Gamma \to G$ with pre-compact image. Let $c\colon G\times G/\Gamma \to \Gamma$ be the associated cocycle, $c(g,x) = s(gx)^{-1}gs(x)$. The cocycle $c$ satisfies the cocycle relation:
$$c(gh,x) = c(g,hx)c(h,x)$$
Note that since $s$ has pre-compact image and the image of $c$ is in $\Gamma$, for any compact $K \subseteq G$ the image of $c|_{K\times G/\Gamma}$ is finite. Given a representation $(H,\sigma)$ of $\Gamma$, the induced representation $\pi=\mathrm{Ind}_\Gamma^G(\sigma)$ is defined on the space $L^2(G/\Gamma,H)$ by:
\begin{equation}\label{eq:cocycle-rep-equation}
    [\pi(g)f](x) = \sigma\bigl(c(g,g^{-1}x)\bigr)f(g^{-1}x).
\end{equation}
In this model, induction of representations defined on the same Hilbert space yields representations defined again on the same Hilbert space. In particular, representations induced from characters of $\Gamma$ are all defined on the space $L^2(G / \Gamma)$. We will apply this to $\rho_\theta$, which are induced from the characters $\chi_\theta$.

Let $K \subseteq G$ be a compact subset, and $\delta > 0$ a positive number. A unit vector~$v$ in a representation $(V, \pi)$ is called \emph{$(K, \delta)$-invariant} if $\|\pi(k)v - v\| < \delta$ for every $k \in K$.
For every open neighborhood $U$ of $1 \in \widehat{G}$ (which we treat as an ``end of the complementary series'') there exist $K$ and $\delta$ such that if a representation $\pi$ admits $(K,\delta)$-invariants, its support must intersect $U$ non-trivially.
So $(K,\delta)$-invariants give a quantitative realization of proximity to the trivial representation in $\widehat{G}$.  For the representation $\rho_0$ -- which is just the quasi-regular representation of $G$ on $L^2(G/\Gamma)$ -- the multiplicity of invariant vectors is $1$. We aim to prove Lemma \ref{lem:multiplicity 1 a.i}, which is an analogue for the multiplicity of $(K,\delta)$-invariants in a representation $\rho_\theta$ for small enough $\theta$. To do this, we first prove the following lemma. It says that when we realize all of the representations $\rho_\theta$ on the same space $L^2(G/\Gamma)$, then $(K,\delta)$-invariants are close to being constant. In order to simplify calculations we assume that the Haar measure is normalized so that $\mathrm{Vol}(G/\Gamma) = 1$.

\begin{lemma}\label{lem:a.i close to invariants}
    For every $\eta > 0$ there exist a compact subset $K \subseteq G$ and $\delta,\varepsilon > 0$ such that, if $f \in L^2(G/\Gamma)$ is a $(K,\delta)$-invariant unit vector for $\rho_\theta$ with $\abs{\theta} < \varepsilon$, then $\int f\ne 0$ and
    \[
        \left\|f-\frac{\int f}{|\int f|}\right\| < \eta.
    \]
\end{lemma}

\begin{proof}
    Put $\eta_0=\min\{1/2,\eta/2\}$. We start by considering the representation $\rho_0$. Denote by $P$ the projection onto $L^2_0(G/\Gamma)$. Since $\Gamma$ is cocompact, it is \emph{weakly cocompact}, namely the restriction of $\rho_0$ to $L^2_0(G/\Gamma)$ does not weakly contain the trivial representation \cite[Corollary~III.1.10]{Margulis}. Thus there exist a compact $K \subseteq G$ and $\delta>0$ such that if $f$ is a $(K,2\delta)$-invariant unit vector, then $\|Pf\| < \eta_0$. Write $m=\int f$. Since the Haar measure is normalized by $\mathrm{Vol}(G/\Gamma)=1$, we have the orthogonal decomposition
    \[
        f=m+Pf,\qquad 1=\|f\|^2=|m|^2+\|Pf\|^2.
    \]
    In particular, $|m|\geq\sqrt{1-\eta_0^2}>0$. Moreover,
    since $Pf$ is orthogonal to the constant functions,
    \[
        \left\|f-\frac{m}{|m|}\right\|^2
        =
        \left\|Pf+\left(m-\frac{m}{|m|}\right)\right\|^2
        =
        \|Pf\|^2+(1-|m|)^2
        \leq
        2\|Pf\|^2
        < \eta^2.
    \]

    Next, we consider the representation $\rho_\theta$ for arbitrary $\theta$, retaining $K$ and $\delta$ that were chosen for $\rho_0$. As explained in the discussion before the lemma, $c|_{K\times G/\Gamma}$ takes only finitely many values. Therefore if $\theta$ is small enough, $|\chi_\theta(c(g,x))-1| <  \delta$ for all $g\in K$ and $x \in G/\Gamma$. Using \eqref{eq:cocycle-rep-equation} for the induced representation, this implies that for $g \in  K$, $f \in L^2(G/\Gamma)$ and $x \in G/\Gamma$, we have:
    $$|[\rho_\theta(g)f](x) - [\rho_0(g)f](x)|< \delta\abs{f(g^{-1}x)}$$
    and therefore $\|\rho_\theta(g)f - \rho_0(g)f\|_2 < \delta \|f\|_2$. We conclude that if $f$ is a $(K,\delta)$-invariant unit vector for $\rho_\theta$, then it is $(K,2\delta)$-invariant for $\rho_0$, and then we already showed that $\int f\ne0$ and $\|f - \frac{\int f}{\abs{\int{f}}}\| < \eta$ as needed.
\end{proof}

This ``multiplicity $1$" result will allow us to control the multiplicity of representations close to the trivial representation in $\mathrm{supp}(\rho_{\theta})$. 
\begin{lemma}\label{lem:multiplicity 1 a.i}
    There exist a neighborhood $U$ of the trivial representation in $\widehat{G}$ and $\varepsilon > 0$ such that for any $\theta \in (-\varepsilon,\varepsilon)$, the representation $\rho_\theta$ is a direct sum of a representation whose support does not intersect $U$ and a unique irreducible unitary representation $\alpha_\theta \in U$.
\end{lemma}
\begin{proof}
    By Lemma \ref{lem:a.i close to invariants}, there exist $\delta>0$, $\varepsilon>0$ and $K \subseteq G$ compact such that for every $\theta$ with $\abs{\theta} < \varepsilon$, if $f$ is a $(K,\delta)$-invariant unit vector for $\rho_\theta$, then $\|f - \frac{\int f}{|\int f|}\| < \frac{\sqrt{2}}{2}$. Therefore if $f$ and $g$ are unit vectors that are $(K,\delta)$-invariant for $\rho_\theta$, then after multiplying them by scalars of absolute value $1$ we may assume that $\|f-g\|_2 < \sqrt{2}$. In particular, $f$ and $g$ cannot be orthogonal.
    
    We assume without loss of generality that $K$ is generating. There exists a neighborhood of the trivial representation $U$ in $\widehat{G}$ such that if a representation has $(K,\delta)$-invariant vectors, the support of the representation intersects $U$. If $\varepsilon$ is small enough, and  $\theta\in (-\varepsilon, \varepsilon)$, $\rho_\theta$ will have $(K,\delta)$-invariant vectors so $\mathrm{supp}(\rho_\theta)$ must intersect $U$ non-trivially. Assume that $\mathrm{supp}(\rho_\theta)$ intersects $U$ in more than one point. This means that $\rho_\theta$ has two orthogonal subrepresentations supported in $U$, and in particular one can find two orthogonal $(K,\delta)$-invariant vectors $f,g$ for $\rho_\theta$, which is a contradiction. Therefore $\mathrm{supp}(\rho_\theta) \cap U$ contains at exactly one point with multiplicity one.
\end{proof}

In view of Lemma~\ref{lem:multiplicity 1 a.i}, we get a well-defined map $\alpha\colon (-\varepsilon,\varepsilon) \to U$,
$\theta\mapsto \alpha_\theta$.

\begin{proposition}
The map $\alpha$ is continuous.
\end{proposition}

\begin{proof}
    For $\abs{\theta}<\varepsilon$ denote by $\rho_\theta = \alpha_\theta \oplus \beta_\theta$ the decomposition given by Lemma \ref{lem:multiplicity 1 a.i}. For a convergent sequence $\theta_i \to \theta$ we need to show that $\alpha_{\theta_i} \to \alpha_\theta$ in $U \subseteq \widehat{G}$. By continuity of induction we have that $\rho_\theta \prec \bigoplus_{i\in \N} \rho_{\theta_i}$ and in particular
    \[
        \alpha_\theta \prec
        \left(\bigoplus_{i\in \N} \alpha_{\theta_i}\right)
        \oplus
        \left(\bigoplus_{i\in \N} \beta_{\theta_i}\right).
    \]
    By Lemma \ref{weakContain}, $\alpha_\theta$ is weakly contained in one of the two summands. Since $\beta_{\theta_i}$ is supported outside of $U$ for all $i$, $\alpha_\theta$ cannot be weakly contained in the second summand. Therefore
    \[
        \alpha_\theta \prec \bigoplus_{i\in \N} \alpha_{\theta_i}.
    \]
    The same argument applies to every subsequence of $(\theta_i)$, because every subsequence still converges to~$\theta$. Hence $\alpha_\theta$ lies in the closure of every subsequence of $(\alpha_{\theta_i})$. Since, after shrinking $U$ if necessary, we may regard $U$ as an interval in the complementary series, this is equivalent to $\alpha_{\theta_i}\to\alpha_\theta$.
\end{proof}

Using the continuity of this map we can finally show that the support of $\pi$ contains an open neighborhood of the trivial representation.

\begin{proposition} \label{prop:pi contains the end of the complementary series}
    The support of the representation $\pi=\Ind_\Gamma^G \sigma$ contains a neighborhood of the trivial representation.
\end{proposition}

\begin{proof}
    Let $U$ and $\varepsilon>0$ be as in Lemma~\ref{lem:multiplicity 1 a.i}. Shrinking $U$ if necessary, we may assume that the part of $U$ at the trivial representation is precisely a small segment of the complementary series together with its endpoint. More precisely, by the description of the Fell topology on the complementary series recalled in \S\ref{subsec:6.1}, there is $s_0>0$ and a homeomorphism $U\cong [0,s_0)$, where $0$ corresponds to the trivial representation.
    By continuity of~$\alpha$, $\alpha_\theta\ne J$ for $\theta\ne 0$ and the intermediate value theorem we obtain that the image of $\alpha$ contains a subinterval at the end of the complementary series.

    For every $\theta$ the character $\chi_\theta$ is weakly contained in $\sigma$, because $\sigma$ is the pullback to $\Gamma$ of the regular representation of $\Z$. Hence $\rho_\theta=\Ind_\Gamma^G \chi_\theta$ is weakly contained in $\pi=\Ind_\Gamma^G\sigma$. In particular, for small enough $\theta$, we have $\alpha_\theta \prec \pi$. This finishes the proof.
\end{proof}

We conclude with the second proof of Theorem \ref{greatQuestion} for $\PSL_2(\C)$.
\begin{proof}[proof of Theorem \ref{greatQuestion} for $\PSL_2(\C)$]
    Let $\sigma$ be a representation supported on the end of the complementary series, and $\pi$ the representation from Definition \ref{def:rep pi}. By Proposition \ref{prop:pi contains the end of the complementary series} the support of $\pi$ a neighborhood of the trivial representation. Thus, any representation supported in a small enough neighborhood of the trivial representation is weakly contained in $\pi$, and in particular we may assume without loss of generality that $\sigma \prec \pi$. By Proposition \ref{prop: H2pi=0}, $H^2(G;\pi)=0$ and therefore Theorem \ref{thm:non-Hausdorff is topological} yields that $H^2(G;\sigma)=0$, and this finishes the proof.
\end{proof}
 
\bibliographystyle{amsplain}
\bibliography{bibliography}
\end{document}